\documentclass[12pt]{article}
\usepackage{amssymb,amsmath,amsthm, amsfonts}
\usepackage{graphicx}
\usepackage{epsfig}
\usepackage{tikz}

\def\disp{\displaystyle}

\def\dref#1{(\ref{#1})}

\theoremstyle{plain}
\newtheorem{theorem}{Theorem}[section]
\newtheorem{lemma}{Lemma}[section]

\theoremstyle{definition}

\newtheorem{remark}{Remark}[section]

\numberwithin{equation}{section}

\begin{document}

\title{\bf Some further progress for existence and boundedness of
solutions to a  two-dimensional chemotaxis-(Navier-)Stokes
system modeling coral fertilization
}

\author{Jiashan Zheng,
Kaiqiang Li\thanks{Corresponding author.  E-mail address:
 kaiqiangli19@163.com (K. Li)} \\
 %Pengmei Zhang
%\thanks{Corresponding author.   E-mail address:
%zhangpm2021@163.com
% %zhengjiashan2008@163.com
% },
%Xiuran Liu
%\\
% $^{1}$
%    School of Data Science and Artificial Intelligence,\\
%     Dongbei University of Finance and Economics, Dalian, 116025,  P.R.China \\
% $^{2}$
     School of Mathematics and Information Sciences,
     Yantai University,\\
      Yantai 264005,  P.R.China \\
 %$^3$Science College of Northeast Dianli University, Jilin City 132013, China
}
\date{}

\maketitle \vspace{0.3cm}
\noindent
\begin{abstract}
%In this paper, we consider the following Keller¨CSegel(-Navier)-Stokes system
In this paper, we investigate the effects exerted by the interplay among Laplacian diffusion, chemotaxis cross
diffusion and the fluid dynamic mechanism on global existence and
boundedness of the solutions. The mathematical model considered
herein appears as
%We consider a degenerate quasilinear Keller--Segel--Navier-Stokes system  with  indirect signal
%production
$$
\left\{
\begin{array}{l}
n_t+u\cdot\nabla n=\Delta n-\nabla\cdot( nS(n)\nabla c)-nm,\quad
x\in \Omega, t>0,\\
 \disp{ c_{ t}+u\cdot\nabla c=\Delta c-c+w},\quad
x\in \Omega, t>0,\\
 \disp{w_{t}+u\cdot\nabla w=\Delta w-nw},\quad
x\in \Omega, t>0,\\
u_t+\kappa(u \cdot \nabla)u+\nabla P=\Delta u+(n+m)\nabla \phi,\quad
x\in \Omega, t>0,\\
\nabla\cdot u=0,\quad
x\in \Omega, t>0,\\
\end{array}\right.\eqno(KSNF)
$$
in a bounded domain $\Omega\subset \mathbb{R}^2$ with a smooth boundary, which describes the process of coral fertilization occurring in ocean flow.
Here  $\kappa\in \mathbb{R}$ is a given constant,
$\phi\in W^{2,\infty}(\Omega)$
%Here
and $S(n) $ is a scalar function
 satisfies  $|S(n)|\leq C_S(1+n)^{-\alpha}$   {for all}  $n\geq 0$ with some $C_S>0$ and  $\alpha\in\mathbb{R}$.   It is proved  that if either
$\alpha>-1,\kappa=0$ or $\alpha\geq-\frac{1}{2},\kappa\in\mathbb{R}$ is satisfied,
  %It is proved  that if $\alpha>-1,\kappa=0$ or $\alpha\geq-\frac{1}{2},\kappa\in\mathbb{R}$,
  then for any reasonably smooth
initial data, the corresponding Neumann-Neumann-Neumann-Dirichlet initial-boundary problem $(KSNF)$ possesses a globally
classical solution.
 In case of the stronger assumption $\alpha>-1,\kappa = 0$ or  $\alpha>-\frac{1}{2},\kappa \in\mathbb{R},$  we moreover show that the corresponding initial-boundary problem
 admits a unique global classical solution which is uniformly bounded on $\Omega\times(0,\infty)$.
\end{abstract}

\vspace{0.3cm}


\newpage
\section{Introduction}
In a liquid environment, biological cells not only diffuse by themselves, but also bias their movement in
reaction to an external chemical signal. Chemotaxis, an oriented movement of cells (or organisms, bacteria) effected by the chemical gradient,
plays an important role in describing many biological phenomena. The celebrated work of Keller and Segel in 1970 proposed heuristically a chemotaxis model, when the chemical is secreted by the cells themselves, which in its classical form is as follows:
\begin{equation}
\left\{\begin{array}{ll}
n_t=\Delta n-\chi\nabla\cdot(n \nabla c),
\quad
x\in \Omega,~ t>0,\\
\disp{ v_t=\Delta c- c+n,}\quad
x\in \Omega, ~t>0,\\
\end{array}
\right.\label{sss722dff344101.ddgghhhff2ffggffggx16677}
\end{equation}
where $n$ denotes the cell density, $c$ represents the chemoattractive concentration, which is secreted by the
cells themselves.
Starting from the pioneering work of Keller and Segel \cite{Keller79}, an extensive mathematical literature has
grown on the Chemotaxis model and its variants, which  mainly concentrates on the
boundedness and blow-up of the solutions (see e.g. \cite{Cie791,Cie201712791,Tao794,Winkler79,Winkler72,Zheng00,Zhengssdefr23}). This fascinating
behavior already emerges %for the prototypical choices D(n) ¡Ô 1 and S(n, c) ¡Ô n
if either the initial
mass of cells $\int_\Omega n_0$ is large enough (\cite{Herrero710}), or for certain initial data of arbitrary initial mass in dimensions
$N \geq 3$ (\cite{Winkler793}).  Moreover,
models with prevention of overcrowding, volume effects, logistic source or involving
more than one chemoattractant have also been studied in  Cie\'{s}lak and Stinner \cite{Cie791,Cie201712791},
Tao and Winkler \cite{Tao794,Winkler79,Winkler72}, and Zheng et al. \cite{Zheng00,Zhengssdefr23}
and the references therein.

Other variants of the model \dref{sss722dff344101.ddgghhhff2ffggffggx16677} has been used in the mathematical study of coral broad-
cast spawning.  In \cite{Kiselevdd793} and \cite{Kiselevsssdd793}, Kiselev and Ryzhik have studied a two-component system modeling the fertilization process
   \begin{equation}
 \left\{\begin{array}{ll}
 \rho_t+u\cdot\nabla\rho=\Delta \rho-\chi\nabla\cdot( \rho\nabla c)-\varepsilon\rho^q,
 \quad
\\
 \disp{ 0=\Delta c +\rho,}\quad\\
 %\disp{\frac{\partial n}{\partial \nu}=\frac{\partial c}{\partial \nu}=0},\quad
%x\in \partial\Omega, ~t>0,\\
%\disp{n(x,0)=n_0(x)},\quad c(x,0)=c_0(x),
%x\in \Omega,\\
 \end{array}\right.\label{722dff344101.dddddgghggghhff2ffggffggx16677}
\end{equation}
in $\mathbb{R}^N$, where $\rho$ represents the density of egg (sperm) gametes, $u$ is a prescribed solenoidal fluid velocity, $\chi> 0 $
denotes the chemotactic sensitivity constant, and $\varepsilon\rho^q$ designates the fertilization phenomenon. Mathematical analysis on the Cauchy problem of (1.1) in the whole
two-dimensional plane shows the different effects exerted by chemotaxis on fertilization in the
case when $q > 2$ and $q = 2$, respectively \cite{Kiselevdd793} and \cite{Kiselevsssdd793}. Moreover, they proved that the total
mass $m_0(t) =\int_{\mathbb{R}^2}\rho(x,t)dx$ approaches a positive constant whose lower bound is $C(\chi,\rho_0 ,u)$
as $t\rightarrow\infty$ when $q>2$.  In the critical case of $N = q = 2$,
%was studied in \cite{Kiselevsssdd793}
%
%Whereas if $q = 2$,
 a corresponding weaker
but yet relevant effect within finite time intervals is detected (see \cite{Kiselevsssdd793}).

 In some cases of chemotactic movement
in flowing environments the mutual influence between the cells and the fluids may be
significant. However, some experiment evidence and numerical simulation indicated that liquid environment plays an important role on the
chemotactic motion of cells \cite{Bellomo1216,Tuval1215,Winkler31215}.  Considering that the motion of the fluids is determined by the incompressible (Navier-)Stokes equations,
Tuval et al. \cite{Tuval1215} proposed the following chemotaxis-(Navier)-Stokes system to describe such
coupled biological phenomena in the context of signal consumption by cells:
\begin{equation}
\left\{\begin{array}{ll}
n_t+u\cdot\nabla n=\Delta n-\nabla\cdot( nS(x,n,c)\nabla c),\quad
x\in \Omega, t>0,\\
c_t+u\cdot\nabla c=\Delta c-nc,\quad
x\in \Omega, t>0,\\
u_t+\kappa (u\cdot\nabla)u+\nabla P=\Delta u+n\nabla \phi,\quad
x\in \Omega, t>0,\\
\nabla\cdot u=0,\quad
x\in \Omega,\; t>0,\\
\end{array}\right.\label{1.1hhjffggjddssgdddgtyy}
\end{equation}
where $\kappa$,
$P$ and $\phi$ denote, respectively, the strength of nonlinear fluid convection, the pressure of the fluid
and the potential of the gravitational field.
%where
%$n, c, u$, and $P$ denote, respectively, the density of cells (or bacteria), concentration of signal,
%velocity field and pressure of the fluid.
%Here $c$ denotes the oxygen concentration and $\phi$ is mentioned above.
Here  $S$ is a given chemotactic sensitivity function, which can either be a scalar function
or, more generally, a tensor-valued function (see, e.g., Xue and Othmer \cite{Xusddeddff345511215}).
And
the term $n\nabla \phi$ characterizes that the motion of fluid is assumed to be driven by buoyant forces
within the gravitational field with potential $\phi$. %It is exhibited that \dref{1.1hhjffggdddjddsddddsggtyy} involves nonlinear cell self-diffusion measured by $mn^{m-1}$
% (degenerate if $m > 1$,
%and singular if $0 < m < 1$) \cite{Francesco791}, and possibly off-diagonal cross-diffusion mechanisms with the
%tensor-valued chemotactic sensitivity $S(x, n, v)$ (see Xue-Othmer \cite{Xusddeddff345511215}).
In this system, chemotaxis and fluid
are coupled through both the transport of the cells as well as signal  defined by the
terms $u\cdot \nabla n$ as well as $u\cdot \nabla c$. %It is exhibited that \dref{1.1hhjffggjddssgdddgtyy} involves nonlinear cell self-diffusion measured by $mn^{m-1}$  (see \cite{Francesco791,Tao71215,Duanx41215}), and possibly off-diagonal cross-diffusion mechanisms with the
%tensor-valued chemotactic sensitivity $S(x, n, c)$ (\cite{Xue1215,Xuefff1215}).
During the past years, analytical results on \dref{1.1hhjffggjddssgdddgtyy} seem to concentrate on the issue whether the solutions
of corresponding initial-boundary problem \dref{1.1hhjffggjddssgdddgtyy} are global in time and bounded (see e.g. Lorz \cite{Lorz79}, Duan-Lorz-Markowich \cite{Duan12186}, Liu-Lorz \cite{Liucvb12176}, Winkler \cite{Winkler31215,Winkler51215,Winkler61215,Winklerssdff11215},
Chae-Kang-Lee \cite{Chaexdd12176}, Zhang-Zheng \cite{Zhangcvb12176}, Zheng \cite{Zhekkllndsssdddgssddsddfff00} and references therein).  For example, if
$ S(x,n,c)=S(c)$ is a scalar function, Winkler (\cite{Winkler31215,Winkler61215}) proved  that in two-dimensional space \dref{1.1hhjffggjddssgdddgtyy}
possesses  a unique global classical solution which stabilizes to the spatially homogeneous equilibrium   $(\bar{n}_0, 0, 0)$ with $\bar{n}_0:=\frac{1}{|\Omega|}\int_{\Omega}n_0$ as $t\rightarrow\infty$. Besides the results mentioned above, for achieving the global solvability and even stabilization, some nonlinear dynamic mechanisms which coincide with the corresponding
biological contexts are introduced to be taken full advantage of in mathematics, such as nonlinear cell diffusion (\cite{Tao71215,Winklerssscvb12176,Winkleghhhr51215}), saturation effects at large densities of cells (\cite{Winkler11215,Zheklllkkkkllllndsssdddgssddsddfff00}), signal-dependent sensitivity (\cite{Tao34555571215}),
and logistic source (\cite{Wajjjjngssddff21215,Winkler444sdddssdff51215}).

In various situations, however, the interaction of chemotactic movement of the gametes and the surrounding fluid is not negligible (see   Espejo and Suzuki
\cite{Espejoss12186}, Espejo and Winkler  \cite{EspejojjEspejojainidd793}). Specially, when extra assuming that the chemical signal is also transported through the fluid as the
gametes undergone in \dref{722dff344101.dddddgghggghhff2ffggffggx16677}, and with the evolution of the velocity of the fluid being modelled
by an incompressible (Navier-)Stokes equation, the mathematical model \dref{722dff344101.dddddgghggghhff2ffggffggx16677} becomes the following  chemotaxis-(Navier-)Stokes model
\begin{equation}
\left\{\begin{array}{ll}
n_t+u\cdot\nabla n=\Delta n-\nabla\cdot( nS(x,n,c)\nabla c)-\mu n^2,\quad
x\in \Omega, t>0,\\
v_t+u\cdot\nabla c=\Delta c-c+n,\quad
x\in \Omega, t>0,\\
u_t+\kappa (u\cdot\nabla)u+\nabla P=\Delta u+n\nabla \phi,\quad
x\in \Omega, t>0,\\
\nabla\cdot u=0,\quad
x\in \Omega,\; t>0,\\
\end{array}\right.\label{1.1hhjffggjddssgkkllldddgtyy}
\end{equation}
where $\mu\geq 0$ is a constant regulating
the strength of the fertilization process.
 Note that the fluid motion is
governed by the full Navier-Stokes system with nonlinear convection for $\kappa\neq 0$ and by the
Stokes equations if $\kappa = 0$. The chemotaxis-(Navier-)Stokes model system  \dref{1.1hhjffggjddssgkkllldddgtyy} has been studied in the last few years and the main focus is on the solvability result  (see e.g. \cite{Kegssddsddfff00,Lidddfffjkkkkucvb12176,LiuZhLiuLiuandddgddff4556,Wangss21215,Tao41215,Wangssddss21215,Zhenddddgssddsddfff00,Zhengsssssssddd0}).
 If
$\mu= 0$ in system \dref{1.1hhjffggjddssgkkllldddgtyy} without logistic source, the global boundedness of classical
solutions to the Stokes-version of system \dref{1.1hhjffggjddssgkkllldddgtyy} with the tensor-valued $S = S(x, n, c)$
satisfying $|S(x, n, c)| \leq C_S(1+n)^{-\alpha}$ with some $C_S > 0$ and $\alpha > 0$ which implies that the
effect of chemotaxis is weakened when the cell density increases has been proved
for any $\alpha> 0$ in two dimensions \cite{Wang21215} and for $\alpha> \frac{1}{2}$ in three dimensions \cite{Wangss21215}.
Furthermore, similar results are also valid for the three-dimensional Stokes version ($\kappa=0$ in the first equation of \dref{1.1hhjffggjddssgkkllldddgtyy}) of system \dref{1.1hhjffggjddssgkkllldddgtyy}
with $\alpha>\frac{1}{2}$ (see Wang-Xiang \cite{Wangss21215}). In the three dimensional case,
Wang-Liu \cite{LiuZhLiuLiuandddgddff4556} showed that the Keller--Segel--Navier-Stokes ($\kappa\neq0$ in the first equation of \dref{1.1hhjffggjddssgkkllldddgtyy})
system \dref{1.1hhjffggjddssgkkllldddgtyy} admits a global weak solution for tensor-valued sensitivity $S(x, n, c)$ satisfying \dref{x1.73142vghf48rtgyhu} and \dref{x1.73142vghf48gg} with $\alpha > \frac{3}{7}$.
%Then Wang-Winkler-Xiang (\cite{Wddffang11215}) further
%shows that when $\alpha > 0$ and $\Omega\subset R^2$ is a bounded {\bf convex} domain with smooth boundary,
%system \dref{1sdfdffgggggsxdcfffggvgb.1} possesses a global-in-time classical and bounded solution.
 { Later, this restriction was improved to $\alpha > \frac{1}{3}$ ( \cite{LiuZhLiuLiuandddgddff4556}) by one of the current authors,}
%Recently,  Zheng and Ke  (\cite{Kegssddsddfff00})   extends   the results of   \cite{LiuZhLiuLiuandddgddff4556} to the   case $\alpha > \frac{1}{3}$,
which, in light of the known results for the fluid-free system (see Horstmann-Winkler \cite{Horstmann791} and
Bellomo et al. \cite{Bellomo1216}), is an optimal restriction on $\alpha$. When $S(x, n, c) \equiv 1$ and $\kappa = 1$, Winkler \cite{Winklessddffr444sdddssdff51215}
showed that if
$\|n_0\|_{L^1(\Omega)} < 2\pi$, the system \dref{1.1hhjffggjddssgkkllldddgtyy} admits a globally defined generalized
solution, in particular, this hypothesis is fully explicit and independent of the initial
size of further solution components. %When $\mu>0,$
% in the supercritical case, $q > 2$, it has been proved that the total mass ¡ä
%$\int_{\mathbb{R}^2}n(\cdot,t)$ tends
%to a positive constant $C(S, n_0, u)$ as $t \rightarrow\infty$, while in the
%$q = 2$ case, a weaker but related effect within finite time intervals is observed.
  For more works
about the chemotaxis-(Navier-)Stokes models \dref{1.1hhjffggjddssgkkllldddgtyy} and its variants, we mention that a corresponding quasilinear version
or the logistic damping has been deeply investigated by  Zheng
\cite{Zhengddfggghjjkk1,Zhengsdsd6}, Wang-Liu \cite{Lidddfffjkkkkucvb12176}, Tao-Winkler \cite{Tao41215},
 Wang et. al.
\cite{Wang21215,Wangss21215}, Lankeit \cite{Lankeitffg11}.
%Recently, because of the lack of enough regularity and compactness properties for the first equation,
%by using the idea proposed by Winkler \cite{Winklerddfff51215},
%Wang \cite{Wangssddss21215} presented the existence of global {\bf very weak} solutions for the system \dref{1.1hhjffggjddssgkkllldddgtyy} under the assumption that $S$ satisfies
%\dref{x1.73142vghf48rtgyhu} and \dref{x1.73142vghf48gg} with $\alpha > \frac{1}{3}$,
%which, in light of the known results for the fluid-free system mentioned above, is an optimal restriction on $\alpha$ (see \dref{722dff344101.ddff2ffddfffggffggx16677}).
%However, the existence of  global (stronger than the result of \cite{Wangssddss21215}) {\bf  weak} solutions is still open.
%In this paper, we try to obtain  enough regularity and compactness properties (see Lemmas \ref{lemmddaghjsffggggsddgghhmk4563025xxhjklojjkkk},
%\ref{4455lemma45630hhuujjuuyytt}, and \ref{qqqqlemma45630hhuujjuuyytt}), then show that system \dref{1.1hhjffggjddssgkkllldddgtyy} possesses a globally defined {\bf weak} solution (see Definition \ref{df1}),
%which improves the result of \cite{Wangssddss21215}.
%
%We further note that, with some exceptions such as
%[28, 65], the result on global boundedness and large time behavior properties for
%the variant of (3) with nonlinear diffusion and nonlinear cross-diffusion is absent.
More related results are obtained (see \cite{Lidddfffjkkkkucvb12176,Wangssddff21215,Wajjjjngssddff21215}).

% In the contexts of broadcast
%spawning of sea animals, for instance corals, sea urchins, and sea anemones, the chemical signal
%is not consumed but produced by the cells themselves; some variants of (1.1) have been proposed
%simultaneously, one of which is as follows:

Besides the results mentioned above, a large number of variants of
system \dref{1.1hhjffggjddssgkkllldddgtyy} have been investigated,
specially, in order  to analyze a further refinement of the model \dref{722dff344101.dddddgghggghhff2ffggffggx16677} which
explicitly distinguishes between sperms and eggs,  several works addressed %a porous
%medium-type diffusion variant model of (1.1) as
%%To describe the motion of  coral fertilization in an incompressible fluid,
%Espejo and Winkler (\cite{EspejojjEspejojainidd793}) have  considered
the  Navier-Stokes version of \dref{722dff344101.dddddgghggghhff2ffggffggx16677}:
\begin{equation}
 \left\{\begin{array}{ll}
   n_t+u\cdot\nabla n=\Delta n-\nabla\cdot(nS(n)\nabla c)-nm,\quad
x\in \Omega, t>0,\\
    c_t+u\cdot\nabla c=\Delta c-c+m,\quad
x\in \Omega, t>0,\\
 m_t+u\cdot\nabla m=\Delta m-nm,\quad
x\in \Omega, t>0,\\
u_t+\kappa(u \cdot \nabla)u+\nabla P=\Delta u+(n+m)\nabla \phi,\quad
x\in \Omega, t>0,\\
\nabla\cdot u=0,\quad
x\in \Omega, t>0,\\
% \disp{(\nabla n-nS(x, n, c))\cdot\nu=\nabla c\cdot\nu=\nabla m\cdot\nu=0,u=0,}\quad
%x\in \partial\Omega, t>0,\\
%\disp{n(x,0)=n_0(x),c(x,0)=c_0(x),m(x,0)=m_0(x),u(x,0)=u_0(x),}\quad
%x\in \Omega,\\
 \end{array}\right.\label{1.dddduiikkldffdffg1}
\end{equation}
where $n$ represents the density of the sperms and $m$ denotes the density of eggs which
release the chemical signal with concentration $c$ to attract the sperms.
 %$S$ is the given chemotactic sensitivity function.
 Here the chemotactic sensitivity $S$
is assumed to be a given scalar function satisfying
\begin{equation}\label{x1.73142vghf48rtgyhu}
S\in C^2(\bar{\Omega})~~~\mbox{and}~~~ S(n) \geq 0 ~~~\mbox{for all}~~~ n\geq 0
\end{equation}
and
\begin{equation}\label{x1.73142vghf48gg}
|S(n)|\leq C_S(1 + n)^{-\alpha} ~~~~\mbox{for all}~~~ n\geq 0
\end{equation}
with some $C_S > 0$ and $\alpha\geq 0$.
System \dref{1.1hhjffggjddssgkkllldddgtyy} models
the spatio-temporal dynamics of the coral fertilization in the fluid with velocity field $u$ satisfying the incompressible (Navier-)Stokes equations with associated pressure $P$ and external force
$(n+m)\nabla\phi$. During the past years, analytical results on \dref{1.1hhjffggjddssgkkllldddgtyy} seem to concentrate on the issue whether the solutions
of corresponding initial-boundary problem are global in time and bounded
 %Model \dref{1.dddduiikkldffdffg1} and its analogue have been extensively studied up to now
(see e.g.  \cite{Lidfff00,BlackFFGG,LiuansdddWang,Liua22nsdddWang,Liuans333dddWang,Xieghhhr51215}).
%For the four-component system (1.1), Espejo and Winkler [5] have proved the global existence of classical
%solution in two-dimensional bounded domain when the convective term k(u ¡¤ ?)u is considered.
In fact, if $N=2$ and $S\equiv1$,
Espejo and Winkler (\cite{EspejojjEspejojainidd793}) established the global existence of classical solutions to the associated initial-boundary
value problem \dref{1.1hhjffggjddssgkkllldddgtyy}, which tends  towards a spatially homogeneous equilibrium in the large time limit. When the chemotactic sensitivity $S$ is tensor-valued and
fulfills \dref{x1.73142vghf48gg} with $\alpha> \frac{1}{3}$, the global boundedness and asymptotics of classical solutions have
been established in \cite{Lidfff00} for $N = 3$ and $\kappa = 0$. In \cite{Zhengssssddd0},  we  extended
the global existence of weak solutions to values $\alpha > 0$.
A natural question is that whether or not the assumption of $\alpha$ is optimal? Can we further relax the
restriction on $\alpha$, say, to $\alpha\leq0$? In fact, when  $S$ is  tensor-valued and fulfills
 \dref{x1.73142vghf48gg}
with some $C_S > 0$ and $\alpha>-\frac{1}{2}$, it was shown in \cite{Wanssssssg21215}  that
\dref{1.1hhjffggjddssgdddgtyy} possesses the globally classical solutions in dimension two if
$\kappa\neq0$,  while  if  $\kappa=0,$
the  global  boundedness  solutions for system were proved. However, since complementary results on
possibly emerging explosion phenomena are rather barren, it is still unknown that corresponding
uniform $L^p$ bounds could be achieved for smaller values of $\alpha$.
In the present work, we attempt to make use of a different method, by which conditional
estimates for $u$ and $c$ subject to some uniform $L^p$ norms of $n$ are established, to explore how far the
saturation effects at large
densities of  sperms can prevent the occurrence of singularity formation
phenomena.

Inspired by the results mentioned above, we create a new method to further weaken the restriction on $\alpha$, under the circumstance {that}  $S$ is a scalar function.
 On the other hand, plenty of authors have also considered the chemotaxis-Navier-Stokes system with a nonlinear diffusion or  $S$ is supposed to be a chemotactic sensitivity tensor, % that means $\Delta n$ is replaced by $\Delta n^m$,
  one can
see   \cite{Winkler11215,Zhekkllndsssdddgssddsddfff00,Wangjjk5566ddfggghjjkk1}
 and the references therein.

{\bf Main Result.} Considering its physical/biological applications, we study an initial-boundary value
problem of the model \dref{1.dddduiikkldffdffg1} supplemented with the following initial and boundary
conditions:
\begin{equation}\label{ccvvx1.73142sdd6677gg}
n(x, 0) = n_0(x), c(x, 0) = c_0(x),m(x, 0) = m_0(x)~~~ \mbox{and}~~ u(x, 0) = u_0(x), x\in\Omega
\end{equation}
and the boundary conditions
\begin{equation}\label{ccvvx1.ddfff731426677gg}
\disp{(\nabla n-nS(n))\cdot\nu=\nabla c\cdot\nu=\nabla m\cdot\nu=0,\;u=0,}\quad
x\in \partial\Omega, t>0,
\end{equation}
where  $\Omega\subset \mathbb{R}^2$ is a bounded domain with smooth boundary and $\nu$ denotes the outward unit
normal vector on $\partial\Omega$.

 %In order to formulate our result, we specify the precise mathematical setting:
%we shall subsequently consider %(1.1) along with initial conditions
%%Motivated by above works, we shall consider the three-dimensional case of
%%system \dref{1.1hhjffggjddssgdddgtyy}. We will consider system
% \dref{1.1hhjffggjddssgdddgtyy}
%along with initial conditions
%\begin{equation}\label{ccvvx1.73142sdd6677gg}
%n(x, 0) = n_0(x), v(x, 0) = v_0(x),w(x, 0) = w_0(x)~~~ \mbox{and}~~ u(x, 0) = u_0(x), x\in\Omega
%\end{equation}
%and the boundary conditions
%\begin{equation}\label{ccvvx1.ddfff731426677gg}
%\disp{(\nabla n^m-nS(x, n, c))\cdot\nu=\nabla c\cdot\nu=\nabla w\cdot\nu=0,\;u=0,}\quad
%x\in \partial\Omega, t>0,
%\end{equation}
%where $\Omega\subset \mathbb{R}^3$ is a bounded domain with smooth boundary and $\nu$ denotes the outward normal
%vector on $\partial\Omega$.
%Considering its physical / biological applications, we study an initial-boundary value
%problem of the model (1.6)¨C(1.10) supplemented with the following initial and boundary
%conditions:
%(u,n,p,q)(x,0) = (u0,n0,p0,q0)(x), x¡Ê?, (1.11)
%u|?? =0, ?n
%m ¡¤n|?? = 0, ?p¡¤n|?? = 0, ?q ¡¤n|?? = 0, t¡Ý0, (1.12)
%where ??R
%d
%is a bounded convex domain with smooth boundary ??, and n is the unit
%outward normal vector to ??.
 To prepare a precise presentation
of our main result in this direction, throughout this work we assume that the given gravitational
potential function $\phi$ fulfills
\begin{equation}
\phi\in W^{2,\infty}(\Omega),
\label{dd1.1fghyuisdakkkllljjjkk}
\end{equation}
and that the quadruple $(n_0, c_0,m_0, u_0)$ of initial data satisfies
%
%Here for simplicity we shall
%assume that %0.
%% we assume that
%
%and that the initial data are such that
\begin{equation}\label{ccvvx1.731426677gg}
\left\{
\begin{array}{ll}
\displaystyle{n_0\in C^\kappa(\bar{\Omega})~~\mbox{for certain}~~ \kappa > 0~~ \mbox{with}~~ n_0\geq0 ~~\mbox{in}~~\bar{\Omega}},\\
\displaystyle{c_0\in W^{1,\infty}(\Omega)~~\mbox{with}~~c_0>0~~\mbox{in}~~\bar{\Omega},}\\
\displaystyle{m_0\in W^{1,\infty}(\Omega)~~\mbox{with}~~m_0>0~~\mbox{in}~~\bar{\Omega},}\\
\displaystyle{u_0\in D(A^\gamma)~~\mbox{for~~ some}~~\gamma\in ( \frac{{1}}{2}, 1),}\\
\end{array}
\right.
\end{equation}
where $A$ denotes the Stokes operator with domain $D(A) := W^{2,{2}}(\Omega)\cap  W^{1,{2}}_0(\Omega) \cap L^{2}_{\sigma}(\Omega)$
and
$L^{2}_{\sigma}(\Omega) := \{\varphi\in  L^{2}(\Omega)|\nabla\cdot\varphi = 0\}$. %for ${r}\in(1,\infty)$  (similar to that in \cite{Sohr}).

%Now our main result asserts temporally uniform convergence of these solutions in the
%limit ¦Ê ¡ú 0; more precisely:

Within the context of these hypothesis, our main results with regard to global solvability and boundedness  can be formulated as follows.

\begin{theorem}\label{theossddrem3}
Let $\Omega\subset \mathbb{R}^2$ be a bounded    domain with a smooth boundary,
\dref{dd1.1fghyuisdakkkllljjjkk} and \dref{ccvvx1.731426677gg} hold, and suppose that $S$ satisfies \dref{x1.73142vghf48rtgyhu} and \dref{x1.73142vghf48gg}.

If one of the
following cases holds:

(i)
 %some
%Suppose that the assumptions of Theorem \ref{theossddrem3} hold. Moreover, let
%Let $\Omega\subset \mathbb{R}^3$ be a bounded    domain with a smooth boundary.
%\dref{dd1.1fghyuisdakkkllljjjkk} and \dref{ccvvx1.731426677gg} hold, and suppose that $S$ satisfies \dref{x1.73142vghf48rtgyhu} and \dref{x1.73142vghf48gg} with some
\begin{equation}
\alpha>-1~~~\mbox{and}~~~\kappa=0;
\label{sssdddd1.1sssfghyuisdakkkllljjjkk}
\end{equation} %If $m>1$,

(ii) \begin{equation}
\alpha\geq-\frac{1}{2}~~~\mbox{and}~~~\kappa\in\mathbb{R}.
\label{sssdddd1.1sssfghyuisdakkkllljjddddjkk}
\end{equation}

Then there
exist functions
$$
 \left\{\begin{array}{ll}
 n\in C^0(\bar{\Omega}\times[0,\infty ))\cap C^{2,1}(\bar{\Omega}\times(0,\infty )),\\
  c\in  C^0(\bar{\Omega}\times[0,\infty ))\cap C^{2,1}(\bar{\Omega}\times(0,\infty )),\\
    m\in  C^0(\bar{\Omega}\times[0,\infty ))\cap C^{2,1}(\bar{\Omega}\times(0,\infty )),\\
  u\in  C^0(\bar{\Omega}\times[0,\infty ); \mathbb{R}^2)\cap C^{2,1}(\bar{\Omega}\times(0,\infty ); \mathbb{R}^2),\\
  P\in  C^{1,0}(\bar{\Omega}\times(0,\infty )),
   \end{array}\right.
 $$
 which solve \dref{1.dddduiikkldffdffg1}, \dref{ccvvx1.73142sdd6677gg}--\dref{ccvvx1.ddfff731426677gg}
   classically in $\Omega\times (0, \infty)$. Moreover,  when $\alpha>-1,\kappa = 0$ or  $\alpha>-\frac{1}{2},\kappa \in\mathbb{R},$ %this solution is uniformly bounded in the sense that
  this solution is bounded in the sense that
there exists $C > 0$ fulfilling
%In addition, this solution is bounded in
%$\Omega\times(0,\infty)$ in the sense that
\begin{equation}
\|n(\cdot, t)\|_{L^\infty(\Omega)}+\|c(\cdot, t)\|_{W^{1,\infty}(\Omega)}+\|m(\cdot, t)\|_{W^{1,\infty}(\Omega)}+\| u(\cdot, t)\|_{L^{\infty}(\Omega)}\leq C~~ \mbox{for all}~~ t>0.
\label{1.163072xggsssttyyu}
\end{equation}

%Then there exist functions satisfying
%\begin{equation*}
%\left\{\begin{array}{ll}
%n\in L^1([0,\infty); L^{1}(\Omega))\cap L_{loc}^{r}(\bar{\Omega}\times[0,\infty)),\\
%n^m\in L_{loc}^{{\frac{r}{m}}}([0,\infty);W^{1,\frac{r}{m}}(\Omega)),\\
%v \in L^2([0,\infty); L^{2}(\Omega))\cap L_{loc}^4([0,\infty);W^{1,4}(\Omega))\cap L^\infty(\Omega\times(0,\infty)),\\
%w \in  L^2([0,\infty); L^{2}(\Omega))\cap L^{\frac{10}{3}}_{loc}(\Omega\times[0,\infty))\cap L^2_{loc}([0,\infty);W^{1,2}(\Omega)),\\
%u \in  L^2([0,\infty); L^{2}_\sigma(\Omega;\mathbb{R}^3))\cap L^{\frac{10}{3}}_{loc}(\Omega\times[0,\infty);\mathbb{R}^3)\cap L^2([0,\infty);W^{1,2}_{0,\sigma}(\Omega)),\\
%\end{array}\right.\label{dffff1kkk.1fghyuisdakkklll}
%\end{equation*}
%such that $(n,v,w,u)$ is a global weak solution of the system  \dref{1.dddduiikkldffdffg1}, \dref{ccvvx1.73142sdd6677gg}--\dref{ccvvx1.ddfff731426677gg} in the sense of Definition \ref{df1}. Here
%$$
%r= \left\{\begin{array}{ll}
% m+\alpha+\frac{2}{3}~~\mbox{if}~\frac{5}{4}<\alpha\leq2,\\
% 2m+2\alpha-\frac{4}{3}~~\mbox{if}~m+\alpha>2.\\
%   \end{array}\right.
%$$
%This solution can be obtained as the pointwise limit a.e. in $\Omega\times (0,\infty)$ of a suitable sequence
%of classical solutions to the regularized problem \dref{1.1fghyuisda} below.
\end{theorem}
%\begin{theorem}\label{theossddrem3}
%Let $\Omega\subset \mathbb{R}^3$ be a bounded    domain with a smooth boundary,
%\dref{dd1.1fghyuisdakkkllljjjkk} and \dref{ccvvx1.731426677gg} hold, and suppose that $S$ satisfies \dref{x1.73142vghf48rtgyhu} and \dref{x1.73142vghf48gg} with some
%%Suppose that the assumptions of Theorem \ref{theossddrem3} hold. Moreover, let
%%Let $\Omega\subset \mathbb{R}^3$ be a bounded    domain with a smooth boundary.
%%\dref{dd1.1fghyuisdakkkllljjjkk} and \dref{ccvvx1.731426677gg} hold, and suppose that $S$ satisfies \dref{x1.73142vghf48rtgyhu} and \dref{x1.73142vghf48gg} with some
%$$
%m+\alpha>\frac{5}{4}.
%$$
%Then  system  \dref{1.dddduiikkldffdffg1}, \dref{ccvvx1.73142sdd6677gg}--\dref{ccvvx1.ddfff731426677gg} admits at least one global {\bf  weak } solution  $(n, v,w, u, P)$ in the sense of Definition \ref{df1}.
%\end{theorem}
%\begin{theorem}\label{theossddrem3}
%Let $\Omega\subset \mathbb{R}^2$ be a bounded    domain with a smooth boundary.
%\dref{dd1.1fghyuisdakkkllljjjkk} and \dref{ccvvx1.731426677gg} hold, and suppose that $S$ satisfies \dref{x1.73142vghf48rtgyhu} and \dref{x1.73142vghf48gg} with some
%$$
%m+\alpha>1.
%$$
%Then  problem \dref{1.dddduiikkldffdffg1}, \dref{ccvvx1.73142sdd6677gg}--\dref{ccvvx1.ddfff731426677gg}, \dref{ccvvx1.73142sdd6677gg} and \dref{ccvvx1.ddfff731426677gg} possesses at least one global  and bounded weak solution $(n, c, u, P)$.
%\end{theorem}
\begin{remark}
(i) From Theorem \ref{theossddrem3}, we conclude that $\alpha>-1,\kappa = 0$ or  $\alpha>-\frac{1}{2},\kappa \in\mathbb{R},$ is sufficient to guarantee the existence of global and bounded  solutions. This  extend  the results of \cite{Wanssssssg21215}, which provided the existence of a global and bounded
solution for $\alpha>-\frac{1}{2},\kappa = 0$.

(ii) Obviously, $-1<-\frac{1}{2}$,  Theorem \ref{theossddrem3} improves  the results of  Wang-Zhang-Zheng \cite{Wanssssssg21215},  who showed the global classical existence of solutions for \dref{1.dddduiikkldffdffg1}, \dref{ccvvx1.73142sdd6677gg}--\dref{ccvvx1.ddfff731426677gg} in
 the cases $S$  a tensor-valued function with $\alpha> -\frac{1}{2}$.

(iii) Since, $ -1<0$,  Theorem \ref{theossddrem3} also improves  the results of Espejo and Winkler \cite{EspejojjEspejojainidd793},  who showed the global classical existence of solutions for system  \dref{1.dddduiikkldffdffg1}, \dref{ccvvx1.73142sdd6677gg}--\dref{ccvvx1.ddfff731426677gg} in
the cases $S(n)=C_S(1+n)^{-\alpha}$  with $\alpha=0$.

\end{remark}

\section{Preliminaries}

In this Section, we give some basic results as a preparation for the
arguments in the later sections. Firstly, in what follows, without confusion, we
shall abbreviate $\int_\Omega f dx$ as $\int_\Omega f $ for simplicity. Moreover, we shall use $c_i$ for $C_i$ ($i =1, 2, 3,\ldots$) to denote a generic constant which may vary in the context.

Before going further, we list some Lemmas, which will be used throughout this paper.
Firstly, let us give  the following statement on local well-posedness, along with a convenient
extensibility criterion.
%
%The first lemma concerns the local solvability of system (22) in the classical sense.
%
%We first
%give the existence of local solutions of (4) by Schauder fixed point and the standard
%parabolic regularity theory.
%
%Next we will state the local solvability of system (17), which can be proved by a
%straightforward adaption of the corresponding procedures in Lemma 2.1 of [33] to
%our current setting.
%
%In this subsection, we give the
%local existence of solutions to the regularized problem (2.1) and the mass conservation
%of cells.
%The local solvability of \dref{1.1fghyuisda} can be derived by  a suitable
%extensibility criterion and a slight modification of the well-established fixed-point arguments in Lemma 2.1 of \cite{Winkler51215}
%(see also \cite{Winkler11215} and Lemma 2.1 of \cite{Painter55677}), so here we omit the proof.

\begin{lemma}\label{lemma70}
%Let $\varepsilon\in(0,1).$
Assume \dref{x1.73142vghf48rtgyhu}--\dref{x1.73142vghf48gg} and \dref{dd1.1fghyuisdakkkllljjjkk}--\dref{ccvvx1.731426677gg} hold.  Then %for any $\varepsilon\in (0, 1)$,
there exist $T_{max}\in(0,\infty]$, and a uniquely determined function quintuple $(n , c , m , u ,P )$ with  $n \geq0,c >0$ and $w  > 0$
in $\Omega \times[0, T_{max})$ solves \dref{1.dddduiikkldffdffg1}, \dref{ccvvx1.73142sdd6677gg}--\dref{ccvvx1.ddfff731426677gg}  in the classical sense and satisfies
%
%a uniquely
%determined quadruple $(n , c , m , u )$ of functions
%Then the problem \dref{1.dddduiikkldffdffg1}, \dref{ccvvx1.73142sdd6677gg}--\dref{ccvvx1.ddfff731426677gg} possesses
%a global classical solution $(n , c , w ,u , P )$ with
%
%Then there exist $T_{max}\in  (0,\infty]$ and
%a classical solution $(n , c , u , P )$ of \dref{1.dddduiikkldffdffg1}, \dref{ccvvx1.73142sdd6677gg}--\dref{ccvvx1.ddfff731426677gg} in
%$\Omega\times(0, T_{max})$ such that
$$
 \left\{\begin{array}{ll}
 n \in C^0(\bar{\Omega}\times[0,T_{max} ))\cap C^{2,1}(\bar{\Omega}\times(0,T_{max} )),\\
  c \in  C^0(\bar{\Omega}\times[0,T_{max} ))\cap C^{2,1}(\bar{\Omega}\times(0,T_{max} )),\\
    m \in  C^0(\bar{\Omega}\times[0,T_{max} ))\cap C^{2,1}(\bar{\Omega}\times(0,T_{max} )),\\
  u \in  C^0(\bar{\Omega}\times[0,T_{max} ); \mathbb{R}^2)\cap C^{2,1}(\bar{\Omega}\times(0,T_{max} ); \mathbb{R}^2),\\
  P \in  C^{1,0}(\bar{\Omega}\times(0,T_{max} )).
   \end{array}\right.
 $$
% which is such that $n \geq0,v >0$ and $w  > 0$ in $\Omega \times (0, \infty )$.
 %Moreover,
% \begin{equation}
%\int_{\Omega}{n }= \int_{\Omega}{n_{0}}~~\mbox{for all}~~ t>0
%\label{ddfgczhhhh2.5ghju48cfg924ghyuji}
%\end{equation}
%as well as
%\begin{equation}
%\int_{\Omega}{v }\leq \max\bigl\{\int_{\Omega}{v_{0}},\int_{\Omega}{w_{0}}\bigr\}~~\mbox{for all}~~ t>0
%\label{6666ddfgczhhhh2.5ghju48cfg924ghyuji}
%\end{equation}
%and
%\begin{equation}
%\int_{\Omega}{w }\leq \max\bigl\{\int_{\Omega}{n_{0}},\int_{\Omega}{w_{0}}\bigr\}~~\mbox{for all}~~ t>0.
%\label{6666ddfgczhhhh2kkk.5ghju48cfg924ghyuji}
%\end{equation}
 %classically solving \dref{1.dddduiikkldffdffg1}, \dref{ccvvx1.73142sdd6677gg}--\dref{ccvvx1.ddfff731426677gg} in $\Omega\times[0,T_{max})$.
Moreover, if $T_{max}<\infty$, then
%Moreover,  $n $ and $c $ are nonnegative in
%$\Omega\times(0, T_{max})$, and
$$\limsup_{t\nearrow T_{max}}\bigl(\|n (\cdot, t)\|_{L^\infty(\Omega)}+\|c (\cdot, t)\|_{W^{1,\infty}(\Omega)}+\|m (\cdot, t)\|_{W^{1,\infty}(\Omega)}+\|A^\gamma u (\cdot, t)\|_{L^{2}(\Omega)}\bigr)=\infty,
$$
where $\gamma$ is given by \dref{ccvvx1.731426677gg}.
\end{lemma}

In deriving some preliminary time-independent estimates for $n$ as well as $c$ and $u$, we shall make use of an
auxiliary statement on boundedness in a linear differential inequality, which has been proved by
\cite{Zheklllkkkkllllndsssdddgssddsddfff00}.

\begin{lemma}\label{lemma630} (Lemma 2.3 of \cite{Zheklllkkkkllllndsssdddgssddsddfff00})
Let $T>0$, $\tau\in(0,T)$, $A>0,\alpha>0$ and $B>0$, and suppose that $y:[0,T)\rightarrow[0,\infty)$ is absolutely continuous such that
\begin{equation}\label{x1.73142hjkl}
\begin{array}{ll}
\displaystyle{
 y'(t)+Ay^\alpha(t)\leq h(t)}~~\mbox{for a.e.}~~t\in(0,T)\\
\end{array}
\end{equation}
with some nonnegative function $h\in  L^1_{loc}([0, T))$ satisfying
$$
\int_{t}^{t+\tau}h(s)ds\leq B~~\mbox{for all}~~t\in(0,T-\tau).
$$
Then  there exists a positive constant $C$ %for any
 such that
\begin{equation}\label{x1.731ddfff42hjkl}
y(t)\leq\max\left\{y_0+B,\frac{1}{\tau^{\frac{1}{\alpha}}}(\frac{B}{A})^{\frac{1}{\alpha}}+2B\right\}~~\mbox{for all}~~t\in(0,T).
\end{equation}
\end{lemma}

\section{A-estimates for  Problem \dref{1.dddduiikkldffdffg1}, \dref{ccvvx1.73142sdd6677gg}--\dref{ccvvx1.ddfff731426677gg}  }
In this section, we focus on the some useful  estimates for system  \dref{1.dddduiikkldffdffg1}, \dref{ccvvx1.73142sdd6677gg}--\dref{ccvvx1.ddfff731426677gg},
 which  plays an important role in obtaining the global existence of solutions for \dref{1.dddduiikkldffdffg1}, \dref{ccvvx1.73142sdd6677gg}--\dref{ccvvx1.ddfff731426677gg}.
%
%In order to prepare an appropriate passage to the limit $\varepsilon\searrow 0$ to be performed in Lemma \ref{lemma45630223}, a
%natural goal consists in establishing suitable $\varepsilon$-independent bounds for the solutions of \dref{1.dddduiikkldffdffg1}, \dref{ccvvx1.73142sdd6677gg}--\dref{ccvvx1.ddfff731426677gg}.
To begin with, we proceed to establish some basic  a-priori estimates for $n , c $ as well as $m $ and $u $,
 which have already been proved and  used
in \cite{Wanssssssg21215} (see also in \cite{EspejojjEspejojainidd793}). And therefore, we list them here without proof.

\begin{lemma}\label{fvfgsdfggfflemma45}
Suppose that \dref{x1.73142vghf48rtgyhu}--\dref{x1.73142vghf48gg} and \dref{dd1.1fghyuisdakkkllljjjkk}--\dref{ccvvx1.731426677gg} hold. Then for all $t \in (0,T_{max} )$, the solution
of \dref{1.dddduiikkldffdffg1}, \dref{ccvvx1.73142sdd6677gg}--\dref{ccvvx1.ddfff731426677gg} from Lemma \ref{lemma70} satisfies
%For any $0<T<T_{max}$,
%There exists %$m > 0$ and
%$\lambda > 0$ %independent of $\varepsilon$
%such that the solution of \dref{1.dddduiikkldffdffg1}, \dref{ccvvx1.73142sdd6677gg}--\dref{ccvvx1.ddfff731426677gg} satisfies
%
%Under the assumption of Lemma \ref{lemma70}.
%, the solution of \dref{1.dddduiikkldffdffg1}, \dref{ccvvx1.73142sdd6677gg}--\dref{ccvvx1.ddfff731426677gg} satisfies
%Assume that %$f$ satisfies \dref{6291} and
%$(u, v)$ is the solution of \dref{1.dddduiikkldffdffg1}, \dref{ccvvx1.73142sdd6677gg}--\dref{ccvvx1.ddfff731426677gg}.
%
%Then for any $T\in (s, T_{max})$, there exists $C > 0$ such that
\begin{equation}
\|n (\cdot,t)\|_{L^1(\Omega)}\leq\|n_{0}\|_{L^1(\Omega)}~~~~\mbox{for all}~~ t\in(0, T_{max}),
\label{ddfgczhhhh2.5ghjjjssddju48cfg924ghyuji}
\end{equation}
\begin{equation}
\frac{d}{dt}\int_{\Omega}{n  }\leq0~~~~\mbox{for all}~~ t\in(0, T_{max}),
\label{ddfgczhhhh2.ssss5ghjjjssddju48cfg924ghyuji}
\end{equation}
\begin{equation}
\|c (\cdot,t)\|_{L^\infty(\Omega)}\leq M_*:=\max\{\|c_{0}\|_{L^\infty(\Omega)},\|m_{0}\|_{L^\infty(\Omega)}\}~~\mbox{for all}~~ t\in(0, T_{max}),
\label{ddfgczhhhh2.5sddddghju48cfg924ghyuji}
\end{equation}
\begin{equation}
\|m  (\cdot,t)\|_{L^\infty(\Omega)}\leq \|m_{0}\|_{L^\infty(\Omega)}~~\mbox{for all}~~ t\in(0, T_{max}),
\label{ddfgczhhhh2.5ghju48cfg924ghyuji}
\end{equation}
\begin{equation}
\int_{\Omega}{n  }-\int_{\Omega}{m }= \int_{\Omega}{n_0}-\int_{\Omega}{m_0}~~\mbox{for all}~~ t\in(0, T_{max})
\label{sssddfgczhhhh2.5ghju48cfg924ghyuji}
\end{equation}
%Moreover, for any $T\in (0, T_{max})$, there exists $C > 0$ such that
as well as
\begin{equation}
\|m (\cdot,t)\|_{L^2(\Omega)}^2+2\int_0^{t}\int_{\Omega}{|\nabla m |^{2}}\leq \|m_{0}\|_{L^2(\Omega)}^2~~\mbox{for all}~~ t\in(0, T_{max})
\label{ddczhjjjj2.5ghju48cfg9ssdd24}
\end{equation}
and
\begin{equation}
\int_0^{t}\int_{\Omega}{n m }dxds\leq \min\{\|n_{0}\|_{L^1(\Omega)},\|m_{0}\|_{L^1(\Omega)}\}~~\mbox{for all}~~ t\in(0, T_{max}).
\label{ddczhjjjj2.5ghxxccju48cfg9ssdd24}
\end{equation}
\end{lemma}

As a main ingredient for the proof of Lemma \ref{lemma4563025xxhjkloghyui} as well as \ref{789lemma4563025xxhjkloghyui} and Lemma \ref{9999lemma4563ddfff025xxhjkloghyui} below, let us separately state the following
the auxiliary interpolation lemma. %which may be viewed as a variant of
%a corresponding precedent established in [69, Lemma 3.8] for situations when the spatial
%L¡Þ norm is involved, rather than that in L6(¦¸).

%In a straightforward manner, from Lemma 3.8 and Lemma 3.10 we can moreover deduce certain
%regularity features of the time derivatives in (2.9).

\begin{lemma}\label{llllplemkklllma4563ddfff025xxhjkloghyui}
Assume that $\Omega\subset \mathbb{R}^2$ is a bounded and smooth domain, and that $p> 1$.
Then
there exists $C_\Omega > 0$ and $\tilde{C}_\Omega > 0$  such that for all $\varphi\in C^2(\bar{\Omega})$ with $\partial_\nu \varphi= 0$,
\begin{equation}\label{ddddhddddjui9sss09klllllopji115}
\disp{\int_{\Omega} |\nabla \varphi|^{2p+2}\leq C_\Omega\int_{\Omega} |\nabla \varphi|^{2p-2}|D^2 \varphi|^2\|\varphi\|_{L^{\infty}(\Omega)}^{2}+\tilde{C}_\Omega\|\varphi\|_{L^{\infty}(\Omega)}^{2p+2}.}
\end{equation}
Moreover,
for any $\delta> 0$, there exists $C(\delta) > 0$ such that for all $\varphi\in C^2(\bar{\Omega})$ with $\partial_\nu \varphi= 0$,
\begin{equation}\label{hddddjui9sss09klllllopji115}
\disp{\int_{\partial\Omega} |\nabla \varphi|^{2p-2}\frac{\partial  |\nabla c |^{2}}{\partial\nu}\leq \delta\int_{\Omega} |\nabla \varphi|^{2p-2}|D^2 \varphi|^2+C(\delta)\|\varphi\|_{L^{\infty}(\Omega)}^{2p}.}
\end{equation}
\end{lemma}
\begin{proof}
The proof of this lemma is completely similar to Lemma 3.4 in Ref. \cite{Wangjjk5566ddfggghjjkk1} (see also Lemma 2.1 of \cite{LiLiLiLisssdffssdddddddgssddsddfff00} and \cite{Ishida}), so
we omit it here.
\end{proof}

Our next first purpose will be to improve the regularity information on the chemotactic gradient $c$. As a preliminary step toward this,
Lemmas \ref{fvfgsdfggfflemma45} and \ref{llllplemkklllma4563ddfff025xxhjkloghyui}
 provide a simple functional
$\int_{\Omega} |\nabla c |^{2p}$
for any $p > 1$.

%The first term on the right of (2.24) may clearly be absorbed by the dissipative integral in (2.16). Accordingly, our
%next goal is to cope with the second appropriately.
%
%This lemma shows that the boundedness of $u$ in any $L^p(\Omega)$ implies the boundedness of $\nabla w$ in  $L^\infty(\Omega)$.
%

\begin{lemma}\label{789lemma4563025xxhjkloghyui}
Let ${p}>1$. Then there exists a positive constant $C$ %independent of $\varepsilon$
such that the solution of \dref{1.dddduiikkldffdffg1}, \dref{ccvvx1.73142sdd6677gg}--\dref{ccvvx1.ddfff731426677gg} satisfies
\begin{equation}\label{789hjui909klopji115}
\begin{array}{rl}
&\disp{\frac{1}{{2{{p}}}}\frac{d}{dt}\|\nabla c \|^{{{2{{p}}}}}_{L^{{2{{p}}}}(\Omega)}
+\frac{{{p}}-1}{2{{{p}}^2}}\int_{\Omega}\left|\nabla |\nabla c |^{{{p}}}\right|^2
+\frac{1}{2}\int_\Omega  |\nabla c |^{2{{p}}-2}|D^2c |^2+\frac{1}{2}\int_{\Omega} |\nabla c |^{2{{p}}}}
\\
\leq&\disp{
\int_\Omega (u \cdot\nabla  c )\nabla\cdot( |\nabla c |^{2{p}-2}\nabla c )+C~~\mbox{for all}~~ t\in(0,T_{max}).}
\end{array}
\end{equation}
\end{lemma}

\begin{proof}
Differentiating the second equation in f\dref{1.dddduiikkldffdffg1} and using that $\Delta|\nabla c|^2 = 2\nabla c\cdot \nabla \Delta c +
2|D^2c|^2$, we obtain the pointwise identity
\begin{equation}
\begin{array}{rl}
&\disp{\frac{1}{2}(|\nabla c|^2)_t}
\\
= &\disp{\nabla c\cdot\nabla\left\{\Delta c
-c +m -u \cdot\nabla  c\right\} }
\\
=&\disp{\frac12\Delta|\nabla c|^2-|D^2c|^2-\nabla c\cdot\nabla(c-m) }
\\
&-\disp{\nabla c\cdot\nabla(u\cdot\nabla c)~~\mbox{for all}~~ t\in(0,T_{max}).}\\
%&+\disp{\int_\Omega (u \cdot\nabla  c ) |\nabla c |^{2m-2}\Delta c
%+\int_\Omega (u \cdot\nabla  c )\nabla c \cdot\nabla( |\nabla c |^{2m-2})}
\end{array}
\label{ssdddcz2.5ghju48156}
\end{equation}
Multiplying \dref{ssdddcz2.5ghju48156} by $(|\nabla c|^2)^{p-1}$ and integrating by parts over $\Omega$, we have
%
%Considering the fact that $\nabla c \cdot\nabla\Delta c   = \frac{1}{2}\Delta |\nabla c |^2-|D^2c |^2$,
%by a straightforward computation, using the second equation in \dref{1.dddduiikkldffdffg1} and several integrations by parts, we find that
\begin{equation}
\begin{array}{rl}
&\disp{\frac{1}{{2{p}}}\frac{d}{dt} \|\nabla c \|^{{{2{p}}}}_{L^{{2{p}}}(\Omega)}}
\\
%= &\disp{\int_{\Omega} |\nabla c |^{2{p}-2}\nabla c \cdot\nabla(\Delta c
%-c +m -u \cdot\nabla  c )}
%\\
=&\disp{\frac{1}{{2}}\int_{\Omega} |\nabla c |^{2{p}-2}\Delta |\nabla c |^2
-\int_{\Omega} |\nabla c |^{2{p}-2}|D^2 c |^2-\int_{\Omega} |\nabla c |^{2{p}}}
\\
&-\disp{\int_\Omega m \nabla\cdot( |\nabla c |^{2{p}-2}\nabla c )
+\int_\Omega (u \cdot\nabla  c )\nabla\cdot( |\nabla c |^{2{p}-2}\nabla c )}
\\
=&\disp{-\frac{{p}-1}{{2}}\int_{\Omega} |\nabla c |^{2{p}-4}\left|\nabla |\nabla c |^{2}\right|^2
+\frac{1}{{2}}\int_{\partial\Omega} |\nabla c |^{2{p}-2}\frac{\partial  |\nabla c |^{2}}{\partial\nu}
-\int_{\Omega} |\nabla c |^{2{p}}}\\
&-\disp{\int_{\Omega} |\nabla c |^{2{p}-2}|D^2 c |^2
-\int_\Omega m  |\nabla c |^{2{p}-2}\Delta c -\int_\Omega m \nabla c \cdot\nabla( |\nabla c |^{2{p}-2})}
\\
&+\disp{\int_\Omega (u \cdot\nabla  c )\nabla\cdot( |\nabla c |^{2{p}-2}\nabla c )}
\\
=&\disp{-\frac{2({p}-1)}{{{p}^2}}\int_{\Omega}\left|\nabla |\nabla c |^{m}\right|^2
+\frac{1}{{2}}\int_{\partial\Omega} |\nabla c |^{2{p}-2}\frac{\partial  |\nabla c |^{2}}{\partial\nu}
-\int_{\Omega} |\nabla c |^{2{p}-2}|D^2 c |^2}
\\
&-\disp{\int_\Omega m  |\nabla c |^{2{p}-2}\Delta c
-\int_\Omega m \nabla c \cdot\nabla( |\nabla c |^{2{p}-2})-\int_{\Omega} |\nabla c |^{2{p}}}
\\
&\disp{+\int_\Omega (u \cdot\nabla  c )\nabla\cdot( |\nabla c |^{2{p}-2}\nabla c )}\\
%&+\disp{\int_\Omega (u \cdot\nabla  c ) |\nabla c |^{2m-2}\Delta c
%+\int_\Omega (u \cdot\nabla  c )\nabla c \cdot\nabla( |\nabla c |^{2m-2})}
\end{array}
\label{cz2.5ghju48156}
\end{equation}
for all $t\in(0,T_{max})$.
We now estimate the right hand side of \dref{cz2.5ghju48156} one by one. To this end,
 since $|\Delta c | \leq\sqrt{2}|D^2c |$, by utilizing the Young inequality and Lemma \ref{fvfgsdfggfflemma45}, we can estimate
\begin{equation}
\begin{array}{rl}
&\disp\int_\Omega m  |\nabla c |^{2{p}-2}\Delta c
\\
\leq&\disp{\sqrt{2}\int_\Omega m  |\nabla c |^{2{p}-2}|D^2c |}
\\
\leq&\disp{\frac{1}{4}\int_\Omega  |\nabla c |^{2{p}-2}|D^2c |^2+{2}\int_\Omega m^2  |\nabla c |^{2{p}-2}}
\\
\leq&\disp{\frac{1}{4}\int_\Omega  |\nabla c |^{2{p}-2}|D^2c |^2+{2}\| m \|^2_{L^\infty(\Omega)}\int_\Omega  |\nabla c |^{2{p}-2}}\\
\leq&\disp{\frac{1}{4}\int_\Omega  |\nabla c |^{2{p}-2}|D^2c |^2+C_1\int_\Omega  |\nabla c |^{2{p}-2}}\\
\end{array}
\label{cz2.5ghju48hjuikl1}
\end{equation}
%and, similarly,
%\begin{equation}
%\begin{array}{rl}
%&\disp\int_\Omega (u \cdot\nabla  c ) |\nabla c |^{2{m}-2}\Delta c
%\\
%\leq&\disp{\sqrt{2}\int_\Omega |u \cdot\nabla  c | |\nabla c |^{2{m}-2}|D^2c |}
%\\
%\leq&\disp{\frac{1}{4}\int_\Omega  |\nabla c |^{2{m}-2}|D^2c |^2
%+2\int_\Omega |u \cdot\nabla  c |^2 |\nabla c |^{2{m}-2}}
%\\
%\leq&\disp{\frac{1}{4}\int_\Omega  |\nabla c |^{2{m}-2}|D^2c |^2
%+2\int_\Omega |u |^2 |\nabla c |^{2{m}}}
%\\
%\leq&\disp{\frac{1}{4}\int_\Omega  |\nabla c |^{2{m}-2}|D^2c |^2
%+2\int_\Omega |u |^2 |\nabla c |^{2{m}}}
%\end{array}
%\label{cz2.5ghju48hjuikl451}
%\end{equation}
%for all $t\in(0,T_{max})$. Again, from the Young inequality, we have
and
\begin{equation}
\begin{array}{rl}
&-\disp\int_\Omega m \nabla c \cdot\nabla( |\nabla c |^{2{p}-2})
\\
= &\disp{-({p}-1)\int_\Omega m  |\nabla c |^{2({p}-2)}\nabla c \cdot\nabla |\nabla c |^{2}}
\\
\leq &\disp{\frac{{p}-1}{8}\int_{\Omega} |\nabla c |^{2{p}-4}\left|\nabla |\nabla c |^{2}\right|^2+2({p}-1)
\int_\Omega |m |^2 |\nabla c |^{2{p}-2}}
\\
\leq &\disp{\frac{({p}-1)}{2{{p}^2}}\int_{\Omega}\left|\nabla |\nabla c |^{p}\right|^2+2({p}-1)\|m \|^2_{L^\infty(\Omega)}
\int_\Omega  |\nabla c |^{2{p}-2}}\\
\leq &\disp{\frac{({p}-1)}{2{{p}^2}}\int_{\Omega}\left|\nabla |\nabla c |^{p}\right|^2+C_2
\int_\Omega  |\nabla c |^{2{p}-2}}\\
\end{array}
\label{cz2.5ghju4ghjuk81}
\end{equation}
with some $C_1>0$ and $C_2>0$.
Now since Lemma \ref{llllplemkklllma4563ddfff025xxhjkloghyui} asserts the existence of $C_3>0$ such that
%Since furthermore Lemma \ref{llllplemkklllma4563ddfff025xxhjkloghyui}  provides  $C_3>0$ such that
%Next, applying Lemma \ref{llllplemkklllma4563ddfff025xxhjkloghyui}, there is $C_3>0$ such that
\begin{equation}
\begin{array}{rl}
\disp\int_{\partial\Omega}\frac{\partial |\nabla c |^2}{\partial\nu} |\nabla c |^{2{p}-2}
%\leq &\disp{C_2\|\nabla |\nabla c |^{m}\|^a_{L^2(\Omega)}+C_2}
%\\
\leq &\disp{\frac{({p}-1)}{2{{p}^2}}\int_{\Omega}\left|\nabla |\nabla c |^{p}\right|^2+C_3},
\end{array}
\label{cz2.57151hhkkhhggyyxx}
\end{equation}
and therefore,  combining  \dref{cz2.5ghju48156}--\dref{cz2.57151hhkkhhggyyxx} and once more employing the Young inequality we can find $C_4>0$ such that
%Now,  due to the Young inequality, invoking  \dref{cz2.5ghju48156}--\dref{cz2.57151hhkkhhggyyxx} we can find $C_4>0$  fulfilling
%Now, together with \dref{cz2.5ghju48156}--\dref{cz2.57151hhkkhhggyyxx} and the Young inequality, we can derive that for some positive constant $C_4$,
\begin{equation}\label{hjui909klopsssdddji115}
\begin{array}{rl}
&\disp{\frac{1}{{2{p}}}\frac{d}{dt}\|\nabla c \|^{{{2{p}}}}_{L^{{2{p}}}(\Omega)}
+\frac{{p}-1}{2{{p}^2}}\int_{\Omega}\left|\nabla |\nabla c |^{{p}}\right|^2
+\frac{1}{2}\int_\Omega  |\nabla c |^{2{p}-2}|D^2c |^2+\frac{1}{2}\int_{\Omega} |\nabla c |^{2{p}}}
\\
\leq&\disp{
\int_\Omega (u \cdot\nabla  c )\nabla\cdot( |\nabla c |^{2{p}-2}\nabla c )+C_4~~\mbox{for all}~~ t\in(0,T_{max}),}
\end{array}
\end{equation}
which immediately leads to our conclusion.
This completes the proof of Lemma \ref{789lemma4563025xxhjkloghyui}.
\end{proof}
Another application of Lemma \ref{fvfgsdfggfflemma45} is the following $L^p$ estimate of $u$, whose proof is similar
to that of its three-dimensional version (see Corollary 3.4 in Winkler \cite{Winkler11215}, see also \cite{Wang21215}).
\begin{lemma}\label{lemma3.4}
Let $\kappa=0$. Then
for any $p>1$, there exists $C>0$ such that
\begin{equation}
\int_{\Omega}| u (\cdot,t)|^p\leq C~~\mbox{for all}~~t\in(0,T_{max}).
\label{qidjfnhf}
\end{equation}
\end{lemma}
\begin{proof}
On the basis of the variation-of-constants formula for the projected version of the fourth
equation in \dref{1.dddduiikkldffdffg1}, we derive that
$$u (\cdot, t) = e^{-tA}u_0 +\int_0^te^{-(t-\tau)A}
\mathcal{P}((n (\cdot,\tau)+m (\cdot,\tau))\nabla\phi)d\tau~~ \mbox{for all}~~ t\in(0,T_{max}).$$
Let $h :=\mathcal{P}((n (\cdot,\tau)+m (\cdot,\tau)\nabla\phi)$.
Then  in view of Lemma \ref{fvfgsdfggfflemma45} as well as  \dref{dd1.1fghyuisdakkkllljjjkk}, one has
$$
\|h (\cdot,t)\|_{L^{{1}}(\Omega)}\leq C_1 ~~~\mbox{for all}~~ t\in(0,T_{max})
$$
with some $C_1>0.$
% and $p_0={4+2\alpha}.$
For any $p > 1$, we can fix $\gamma_1$ such that $\gamma_1\in (1- \frac{1}{p}, 1)$. It then follows the standard smoothing
properties of the Stokes semigroup that
%Therefore, according to standard smoothing
%properties of the Stokes semigroup we see that
there exist $C_2>0$ as well as  $C_3> 0$ and $\lambda_1 > 0$ such that
\begin{equation}
\begin{array}{rl}
\|u (\cdot, t)\|_{L^p(\Omega)}\leq&\disp{\|A^{\gamma_1}
e^{-tA}u_0\|_{L^p(\Omega)} +\int_0^t\|A^{\gamma_1} e^{-(t-\tau)A}A^{-{\gamma_1}}[h (\cdot,\tau)]d\tau\|_{L^p(\Omega)}d\tau}\\
\leq&\disp{C_2 +C_{2}\int_0^t(t-\tau)^{-{\gamma_1}-\frac{2}{2}(\frac{1}{{1}}-\frac{1}{p})}e^{-\lambda_1(t-\tau)}\|h (\cdot,\tau)\|_{L^{{1}}(\Omega)}d\tau}\\
\leq&\disp{C_{3}~~ \mbox{for all}~~ t\in(0,T_{max}),}\\
\end{array}
\label{cz2.571hhhhh51ccvvhddfccvvhjjjkkhhggjjllll}
\end{equation}
%with ${\gamma_1}\in ( \frac{3}{4}, 1),$
where  in
the last inequality %and $p_0$ is the same as Lemma \ref{lemma45566645630223}.
 we have used the fact that
$$\begin{array}{rl}\disp\int_{0}^t(t-\tau)^{-{\gamma_1}-\frac{2}{2}(\frac{1}{{1}}-\frac{1}{p})}e^{-\lambda_1(t-\tau)}ds
\leq&\disp{\int_{0}^{\infty}\sigma^{-{\gamma_1}-\frac{2}{2}(\frac{1}{{1}}-\frac{1}{p})} e^{-\lambda_1\sigma}d\sigma<+\infty}\\
\end{array}
$$
by using  $-{\gamma_1}-\frac{2}{2}(\frac{1}{{1}}-\frac{1}{p})>-1.$
%Hence,
%\dref{cz2.571hhhhh51ccvvhddfccvvhjjjkkhhggjjllll} implies to
% \begin{equation}
%\begin{array}{rl}
%\|u (\cdot, t)\|_{L^\infty(\Omega)}\leq  \sigma_{0}~~ \mbox{for all}~~ t\in(0,T_{max})\\
%\end{array}
%\label{cz2ddfgjjj.5jkkcvvvhjkfffffkhhgll}
%\end{equation}
%by using the fact that $D(A^\gamma)$ is continuously embedded into $L^\infty(\Omega)$ (by $\gamma>\frac{3}{4}$).
%For any $p>1$, recalling Lemma \ref{fvfgsdfggfflemma45},
%by applying Sobolev embedding $W^{1,l}(\Omega)\hookrightarrow L^p(\Omega)$, we derive that  \dref{qidjfnhf} holds.
\end{proof}
Now our knowledge on regularity of $u$ (see Lemma \ref{lemma3.4}) is sufficient to allow for the derivation of $L^p$ bounds for
$n$ by means of an analysis of a functional $\int_\Omega n^p$
 with  arbitrarily
large $p$.

\begin{lemma}\label{lsssemma3.5}
Let  $\kappa=0$ as well as $|S(n)|\leq C_S(1+n)^{-\alpha}$ with $\alpha>-1$.  Then for any $p>2,$ %there exists $K(m_0) > 0$
% with the property
%that whenever  \dref{ccvvx1.731426677gg}  holds with
% $\disp\int_\Omega u_0 \leq m_0$,
one can find $C > 0$
such that
\begin{equation}
\begin{array}{rl}\label{ssddfddfff44ss55ddd5eqx45xx12112ccgghh}
\disp\int_\Omega n ^p(\cdot,t)
\leq &
C~~~~ \mbox{for all}~ t\in(0,T_{max}).\\
\end{array}
\end{equation}
%and
%\begin{equation}
%\begin{array}{rl}\label{ddfff44ss555eqx45xx12112ccgghh}
%%\disp\biggl\{
% \disp\int_{(t-\tau)_+}^t \int_\Omega{}{(n  + 1)^2}
%\leq &  C~~~\mbox{for all}~~ t \in (0,T_{max}-\tau).\\
%\end{array}
%\end{equation}
%where $m := \disp\int_\Omega u_0$.
\end{lemma}
\begin{proof}Testing the first equation in \dref{1.dddduiikkldffdffg1} by $n ^{p-1}$, integrating by parts over $\Omega$ and using that
$\nabla\cdot u  = 0$, we gain
%Testing
%the first equation
%of  \dref{1.dddduiikkldffdffg1}, \dref{ccvvx1.73142sdd6677gg}--\dref{ccvvx1.ddfff731426677gg}  by $n ^{p-1}$ and using the Young inequality,
\begin{equation}\label{ddsssdffsssf44ss55dddlll5eqx45xx12112ccgghh}
\begin{array}{rl}
& \disp\frac{1}{p}\frac{d}{dt}\disp\int_\Omega n ^p + (p-1)\disp\int_\Omega n ^{p-2}{| \nabla n  |^2}\\
\leq&\disp (p-1)\int_\Omega  n ^{p-1}{ C_S(1+n )^{-\alpha}}{}\nabla n  \cdot \nabla c \\
=&\disp \int_\Omega \nabla \int_0^{n }{ \tau^{p-1}C_S(1+\tau)^{-\alpha}}{}d\tau \cdot \nabla c \\
=&\disp -\int_\Omega \int_0^{n }{\tau^{p-1} C_S(1+\tau)^{-\alpha}}{}d\tau \Delta c \\
\leq&\disp C_S\int_\Omega n ^{p}(1+n )^{(-\alpha)_+}|\Delta c |~~~\mbox{for all}~~~ t\in(0,T_{max}).\\
%\leq& \disp\delta\disp\int_\Omega{}{(n  + 1)^2}
%+\disp\frac{1}{4\delta}\disp\int_\Omega |\Delta c |^{\frac{2}{1+\alpha}}
%~~~\mbox{for all}~~~ t > 0
\end{array}
\end{equation}
%for any $\delta>0.$
For $p>2,$  we have $\frac{2}{p}<\frac{2(p+1)}{p}<+\infty$, which allows for an application of  the Gagliardo-Nirenberg inequality along with \dref{ddfgczhhhh2.5ghjjjssddju48cfg924ghyuji} to provide $C_1> 0$ and $C_2> 0$ such that
 %In view of the Gagliardo-Nirenberg inequality giving  constants $C_1> 0$ and $C_2> 0$
% such that
 \begin{equation}\label{sssddsssdffssssssf44ss55dddlll5eqx45xx12112ccgghh}
\begin{array}{rl}
& \disp\int_\Omega{}{n ^{p+1}}\\
=&\disp \|n ^{\frac{p}{2}}\|^{\frac{2(p+1)}{p}}_{L^{\frac{2(p+1)}{p}}(\Omega)}\\
\leq&\disp C_1(\|\nabla n ^{\frac{p}{2}}\|^2_{L^2(\Omega)}\|n ^{\frac{p}{2}}\|^{\frac{2}{p}}_{L^\frac{2}{p}(\Omega)}+\|n ^{\frac{p}{2}}\|^{\frac{2(p+1)}{p}}_{L^\frac{2}{p}(\Omega)})\\
\leq&\disp C_2(\|\nabla n ^{\frac{p}{2}}|^2_{L^2(\Omega)}+1)~~~\mbox{for all}~~~ t\in(0,T_{max}),
\end{array}
\end{equation}
which implies that
$$\disp\|\nabla n ^{\frac{p}{2}}|^2_{L^2(\Omega)}\geq\frac{1}{C_2}\bigl[\int_\Omega{}{n ^{p+1}}-1\bigr]~~~\mbox{for all}~~~ t\in(0,T_{max}).$$
Inserting the above inequality into \dref{ddsssdffsssf44ss55dddlll5eqx45xx12112ccgghh} and using the Young inequality, there is $C_3>0$ such that
\begin{equation}\label{344555ddsssdffsssf44ss55dddlll5eqx45xx12112ccgghh}
\begin{array}{rl}
& \disp\frac{1}{p}\frac{d}{dt}\disp\int_\Omega n ^p
+ \disp\frac{p-1}{2}\disp\int_\Omega n ^{p-2}{| \nabla n  |^2}+\frac{1}{2C_2}\frac{p-1}{2}\frac{4}{p^2}\int_\Omega{}{n ^{p+1}}\\
%\leq&\disp \int_\Omega { C_S(1+n )^{-\alpha}}{}\nabla n  \cdot \nabla c \\
%=&\disp \int_\Omega \nabla \int_0^{n }{ C_S(1+\tau)^{-\alpha}}{}d\tau \cdot \nabla c \\
%=&\disp -\int_\Omega \int_0^{n }{ C_S(1+\tau)^{-\alpha}}{}d\tau \Delta c \\
\leq&\disp\disp\frac{1}{4C_2}\frac{p-1}{2}\frac{4}{p^2}\int_\Omega{}{n ^{p+1}}+C_3 \int_\Omega |\Delta c |^{\frac{p+1}{(1-(-\alpha)_+}}+\frac{ p-1}{C_2 p^2}~~~\mbox{for all}~~~ t\in(0,T_{max}).\\
%\leq& \disp\delta\disp\int_\Omega{}{(n  + 1)^2}
%+\disp\frac{1}{4\delta}\disp\int_\Omega |\Delta c |^{\frac{2}{1+\alpha}}
%~~~\mbox{for all}~~~ t > 0
\end{array}
\end{equation}
In order to appropriately estimate the integral on the right-hand side herein, for any $q>1,$
we rely
on the regularity theory of parabolic equations and  Lemma \ref{fvfgsdfggfflemma45}, whence it is possible to
find   positive constants $C_4$ as well as $C_5$ and $C_6$ such that
%For any $q>1,$  by the regularity theory of parabolic equations and using Lemma \ref{fvfgsdfggfflemma45},
% by the maximal Sobolev
%regularity for the %third and
% second
%equation of system \dref{1.dddduiikkldffdffg1}, \dref{ccvvx1.73142sdd6677gg}--\dref{ccvvx1.ddfff731426677gg} (see Lemma \ref{lemma45xy1222232}),
%we derive that there exist positive constants $C_4$ as well as $C_5$ and $C_6$ such that
%\begin{equation}
%\begin{array}{rl}\label{ddfff44ssssssss55dddlll5eqx45xx12112ccgghh}
%\disp\int_{s_0}^t \disp\int_\Omega| \Delta w  |^p
%\leq &C_2\disp\int_{s_0}^t \bigl[\disp\int_\Omega n ^p+w ^p+|u \cdot\nabla w |^p\bigr]~~~\mbox{for all}~~ t \in (s_0,T_{max}).\\
%\end{array}
%\end{equation}
%and
\begin{equation}
\begin{array}{rl}
&\disp{\int_{(t-\tau)_+}^t\|\Delta c(\cdot,t)\|^{q}_{L^{q}(\Omega)}ds}\\
\leq &\disp{C_4\left(\int_{(t-\tau)_+}^t
\|m(\cdot,s)-u(\cdot,s)\cdot\nabla c(\cdot,s)\|^{q}_{L^{q}(\Omega)}ds\right).}\\
\leq &\disp{C_5\left(\int_{(t-\tau)_+}^t
[\|u(\cdot,s)\|^{{3q}}_{L^{{3q}}(\Omega)}+\|\nabla c(\cdot,s)\|^{\frac{3q}{2}}_{L^{\frac{3q}{2}}(\Omega)}+1]ds\right)}\\
\leq &\disp{C_6\left(\int_{(t-\tau)_+}^t
[\|\nabla c(\cdot,s)\|^{\frac{3q}{2}}_{L^{\frac{3q}{2}}(\Omega)}+1]ds\right)~~\mbox{for all}~~ t\in(0,T_{max}-\tau),}\\
%\leq&\disp{\int_\Omega (|f(x,t)|+L)|u|^{q+1} dx}\\
%\leq&\disp{\int_\Omega (k_1|u|^\alpha+k_2)|u|^{q+1} dx}\\
%=&\disp{\int\int\triangle J(x-y)u(y)(u|u|^q(x))dydx.}
%\leq&\disp{(|g|_{L^\infty(0,\omega; L^\infty(\Omega))}+1)(|\Omega|+1)^{\frac{1}{2}}|u|^{q+1}_{q+2}.}\\
%\leq&\disp{\frac{q+1}{q+2}((|f|_{L^\infty(0,\omega; L^\infty(\Omega))}+1)(|\Omega|+1)^{\frac{1}{2}})^{\frac{q+2}{q+1}}|u|^{q+2}_{q+2}+\frac{1}{q+2}}\\
%\leq&\disp{\frac{q+1}{q+2}(|f|_{L^\infty(0,\omega; L^\infty(\Omega))}+1)^2(|\Omega|+1)|u|^{q+2}_{q+2}+\frac{1}{q+2}.}\\
\end{array}
\label{cz2.5bbsssv114}
\end{equation}
where
\begin{equation}
\tau:=\min\{1,\frac{1}{2}T_{max}\}.
\label{jvgxddrg}
\end{equation}
Next, %then
 \dref{fvfgsdfggfflemma45} and the
Gagliardo-Nirenberg interpolation inequality
provides $C_7>0$ and $C_8>0$ such that
%Next, recalling \dref{fvfgsdfggfflemma45}, then
%Gagliardo-Nirenberg interpolation inequality  yields to
\begin{equation}
\begin{array}{rl}\label{ddfff44sssss55dddlll5eqx45xx12112ccgghh}
\disp&\disp \disp\int_\Omega| \nabla c  |^{\frac{3q}{2}}\\
\leq &C_7\disp \bigl[\|\Delta c(\cdot,t)\|^{\frac{\frac{3q}{4}-1}{1-\frac{1}{q}}}_{L^{q}(\Omega)}\|c\|^{\frac{3q}{2}-\frac{\frac{3q}{4}-1}{1-\frac{1}{q}}}_{L^{\infty}(\Omega)}+\|c\|^{\frac{3q}{2}}_{L^{\infty}(\Omega)}\bigr]\\
\leq &C_8\disp \bigl[\|\Delta c(\cdot,t)\|^{\frac{\frac{3q}{4}-1}{1-\frac{1}{q}}}_{L^{q}(\Omega)}+1\bigr]~~~\mbox{for all}~~ t \in (0,T_{max}).\\
\end{array}
\end{equation}
%with some $C_7>0$ and $C_8>0$.
Since another application of the Young inequality yields $C_{9} > 0$ such that
%As moreover by the Young inequality we have
%With the help of the Young inequality, one has
\begin{equation}
\begin{array}{rl}\label{ddfff44sssss5sssdd5dddlll5eqx45xx12112ccgghh}
\disp \disp \disp\int_\Omega| \nabla c  |^{\frac{3q}{2}}
\leq &\disp\frac{C_8}{2C_6}\disp \|\Delta c(\cdot,t)\|^{q}_{L^{q}(\Omega)}+C_{9}~~~\mbox{for all}~~ t \in (0,T_{max}),\\
\end{array}
\end{equation}
 and \dref{cz2.5bbsssv114}  therefore entails
that
%which combined with \dref{cz2.5bbsssv114} implies that
\begin{equation}
\begin{array}{rl}
\disp{\int_{(t-\tau)_+}^t\|\Delta c(\cdot,t)\|^{q}_{L^{q}(\Omega)}ds}
\leq &\disp{C_{10}~~\mbox{for all}~~ t\in(0,T_{max}-\tau),}\\
%\leq&\disp{\int_\Omega (|f(x,t)|+L)|u|^{q+1} dx}\\
%\leq&\disp{\int_\Omega (k_1|u|^\alpha+k_2)|u|^{q+1} dx}\\
%=&\disp{\int\int\triangle J(x-y)u(y)(u|u|^q(x))dydx.}
%\leq&\disp{(|g|_{L^\infty(0,\omega; L^\infty(\Omega))}+1)(|\Omega|+1)^{\frac{1}{2}}|u|^{q+1}_{q+2}.}\\
%\leq&\disp{\frac{q+1}{q+2}((|f|_{L^\infty(0,\omega; L^\infty(\Omega))}+1)(|\Omega|+1)^{\frac{1}{2}})^{\frac{q+2}{q+1}}|u|^{q+2}_{q+2}+\frac{1}{q+2}}\\
%\leq&\disp{\frac{q+1}{q+2}(|f|_{L^\infty(0,\omega; L^\infty(\Omega))}+1)^2(|\Omega|+1)|u|^{q+2}_{q+2}+\frac{1}{q+2}.}\\
\end{array}
\label{cz2.5bbsssv1sss14}
\end{equation}
with some $C_{10}>0$.
On choosing $q=\frac{p+1}{1-(-\alpha)_+}$, we thus conclude from \dref{cz2.5bbsssv1sss14} that with
some $C_{11}> 0$ we have
%Choosing $q=\frac{p+1}{1-(-\alpha)_+}$ and using \dref{cz2.5bbsssv1sss14}, we have
\begin{equation}
\begin{array}{rl}
\disp{\int_{(t-\tau)_+}^t\|\Delta c(\cdot,t)\|^{\frac{p+1}{(1-(-\alpha)_+}}_{L^{\frac{p+1}{(1-(-\alpha)_+}}(\Omega)}ds}
\leq &\disp{ C_{11}~~\mbox{for all}~~ t\in(0,T_{max}-\tau).}\\
%\leq&\disp{\int_\Omega (|f(x,t)|+L)|u|^{q+1} dx}\\
%\leq&\disp{\int_\Omega (k_1|u|^\alpha+k_2)|u|^{q+1} dx}\\
%=&\disp{\int\int\triangle J(x-y)u(y)(u|u|^q(x))dydx.}
%\leq&\disp{(|g|_{L^\infty(0,\omega; L^\infty(\Omega))}+1)(|\Omega|+1)^{\frac{1}{2}}|u|^{q+1}_{q+2}.}\\
%\leq&\disp{\frac{q+1}{q+2}((|f|_{L^\infty(0,\omega; L^\infty(\Omega))}+1)(|\Omega|+1)^{\frac{1}{2}})^{\frac{q+2}{q+1}}|u|^{q+2}_{q+2}+\frac{1}{q+2}}\\
%\leq&\disp{\frac{q+1}{q+2}(|f|_{L^\infty(0,\omega; L^\infty(\Omega))}+1)^2(|\Omega|+1)|u|^{q+2}_{q+2}+\frac{1}{q+2}.}\\
\end{array}
\label{ssdddfcz2.5bbkkkksssv1sssss14}
\end{equation}
%with some $C_{10}>0.$
Recalling   \dref{344555ddsssdffsssf44ss55dddlll5eqx45xx12112ccgghh} and using \dref{ssdddfcz2.5bbkkkksssv1sssss14}, we derive that \dref{ssddfddfff44ss55ddd5eqx45xx12112ccgghh}  holds with the help of some basic  computation.
\end{proof}
By combining the above Lemmas we can now derive bounds, uniformly with respect to $t\in(0,T_{max})$, for arbitrarily
high Lebesgue norms of $\nabla c$.

\begin{lemma}\label{lemma4563025xxhjkloghyui}
Let $\kappa=0$ as well as $|S(n)|\leq C_S(1+n)^{-\alpha}$ with $\alpha\geq-1$.
%Assume that  $m>1$.
 Then for any $p>1$, there exists a positive constant $C$  such that the solution of \dref{1.dddduiikkldffdffg1}, \dref{ccvvx1.73142sdd6677gg}--\dref{ccvvx1.ddfff731426677gg} satisfies
\begin{equation}\label{hjui909klopji115}
\disp{\|\nabla c (\cdot, t)\|_{L^{2{p}}(\Omega)}\leq C~~\mbox{for all}~~ t\in(0,T_{max}).}
\end{equation}
\end{lemma}

\begin{proof}
We begin with \dref{789hjui909klopji115}. To this end, we should estimate the right-side of \dref{789hjui909klopji115}.
Indeed, using  several integrations by parts, we find that
\begin{equation}
\begin{array}{rl}
&\disp{\int_\Omega (u \cdot\nabla  c )\nabla\cdot( |\nabla c |^{2{p}-2}\nabla c )}
\\
=&\disp{\int_\Omega (u \cdot\nabla  c ) |\nabla c |^{2{p}-2}\Delta c
+\int_\Omega (u \cdot\nabla  c )\nabla c \cdot\nabla( |\nabla c |^{2{p}-2})}
\end{array}
\label{ssddfffcz2.5ghju48156}
\end{equation}
for all $t\in(0,T_{max})$.
Here, %since $|\Delta c | \leq\sqrt{2}|D^2c |$,
 by utilizing the Young inequality as well as  Lemma \ref{fvfgsdfggfflemma45} and Lemma \ref{lemma3.4}, one can pick $C_1>0,C_2>0$ as well as $C_3>0$ and $C_4>0$ such that
%\begin{equation}
%\begin{array}{rl}
%&\disp\int_\Omega w  |\nabla c |^{2m-2}\Delta c
%\\
%\leq&\disp{\sqrt{2}\int_\Omega w  |\nabla c |^{2m-2}|D^2c |}
%\\
%\leq&\disp{\frac{1}{4}\int_\Omega  |\nabla c |^{2m-2}|D^2c |^2+{2}\int_\Omega w^2  |\nabla c |^{2m-2}}
%\\
%\leq&\disp{\frac{1}{4}\int_\Omega  |\nabla c |^{2m-2}|D^2c |^2+{2}\| w \|^2_{L^\infty(\Omega)}\int_\Omega  |\nabla c |^{2m-2}}
%\end{array}
%\label{cz2.5ghju48hjuikl1}
%\end{equation}
%and, similarly,
\begin{equation}
\begin{array}{rl}
&\disp\int_\Omega (u \cdot\nabla  c ) |\nabla c |^{2{p}-2}\Delta c
\\
\leq&\disp{\sqrt{2}\int_\Omega |u \cdot\nabla  c | |\nabla c |^{2{p}-2}|D^2c |}
\\
\leq&\disp{\frac{1}{8}\int_\Omega  |\nabla c |^{2{p}-2}|D^2c |^2
+4\int_\Omega |u \cdot\nabla  c |^2 |\nabla c |^{2{p}-2}}
\\
\leq&\disp{\frac{1}{8}\int_\Omega  |\nabla c |^{2{p}-2}|D^2c |^2
+4\int_\Omega |u |^2 |\nabla c |^{2{p}}}
\\
\leq&\disp{\frac{1}{8}\int_\Omega  |\nabla c |^{2{p}-2}|D^2c |^2
+\frac{1}{16C_\Omega M_*^{2}} \int_\Omega|\nabla c |^{2{p}+2}+C_1\int_\Omega |u |^{2{p}+2}}\\
\leq&\disp{\frac{1}{8}\int_\Omega  |\nabla c |^{2{p}-2}|D^2c |^2
+\frac{1}{16C_\Omega M_*^{2}} \int_\Omega|\nabla c |^{2{p}+2}+C_2}
\end{array}
\label{cz2.5ghju48hjuikl451}
\end{equation}
and
%\begin{equation}
%\begin{array}{rl}
%&-\disp\int_\Omega w \nabla c \cdot\nabla( |\nabla c |^{2{m}-2})
%\\
%= &\disp{-({m}-1)\int_\Omega w  |\nabla c |^{2({m}-2)}\nabla c \cdot\nabla |\nabla c |^{2}}
%\\
%\leq &\disp{\frac{{m}-1}{8}\int_{\Omega} |\nabla c |^{2{m}-4}\left|\nabla |\nabla c |^{2}\right|^2+2({m}-1)
%\int_\Omega |w |^2 |\nabla c |^{2{m}-2}}
%\\
%\leq &\disp{\frac{({m}-1)}{2{{m}^2}}\int_{\Omega}\left|\nabla |\nabla c |^{m}\right|^2+2({m}-1)\|w \|^2_{L^\infty(\Omega)}
%\int_\Omega  |\nabla c |^{2{m}-2}}
%\end{array}
%\label{cz2.5ghju4ghjuk81}
%\end{equation}
%and
\begin{equation}
\begin{array}{rl}
&\disp\int_\Omega (u \cdot\nabla  c )\nabla c \cdot\nabla( |\nabla c |^{2{m}-2})
\\
= &\disp{({p}-1)\int_\Omega (u \cdot\nabla  c ) |\nabla c |^{2({p}-2)}\nabla c \cdot
\nabla |\nabla c |^{2}}
\\
\leq &\disp{\frac{{p}-1}{8}\int_{\Omega} |\nabla c |^{2{p}-4}\left|\nabla |\nabla c |^{2}\right|^2}
\\
&+\disp{2({p}-1)\int_\Omega |u \cdot\nabla  c |^2 |\nabla c |^{2{p}-2}}
\\
\leq &\disp{\frac{({p}-1)}{2{{p}^2}}\int_{\Omega}\left|\nabla |\nabla c |^{p}\right|^2
+2({p}-1)\int_\Omega |u |^2 |\nabla c |^{2{p}}}\\
\leq &\disp{\frac{({p}-1)}{2{{p}^2}}\int_{\Omega}\left|\nabla |\nabla c |^{p}\right|^2
+\frac{1}{16C_\Omega M_*^{2}}\int_\Omega |\nabla c |^{2{p}+2}+C_3\int_\Omega |u |^{2{p}+2}}\\
\leq &\disp{\frac{({p}-1)}{2{{p}^2}}\int_{\Omega}\left|\nabla |\nabla c |^{p}\right|^2
+\frac{1}{16C_\Omega M_*^{2}}\int_\Omega |\nabla c |^{2{p}+2}+C_4.}
\end{array}
\label{cz2.5ghjddfghhu4ghjuk81}
\end{equation}
Here $C_\Omega$  and $M_*$ are the same as Lemma \ref{llllplemkklllma4563ddfff025xxhjkloghyui} and Lemma \ref{fvfgsdfggfflemma45}, respectively.
On the other hand, in light of  Lemma \ref{llllplemkklllma4563ddfff025xxhjkloghyui}, the inequality \dref{ddfgczhhhh2.5sddddghju48cfg924ghyuji} guarantees that
% Lemma \ref{llllplemkklllma4563ddfff025xxhjkloghyui} and Lemma \ref{fvfgsdfggfflemma45} implies that
$$
\begin{array}{rl}
\disp\int_{\Omega} |\nabla c |^{2{p}+2}\leq& \disp{C_\Omega\int_{\Omega} |\nabla c |^{2{p}-2}|D^2 \varphi|^2\|c \|_{L^{\infty}(\Omega)}^{2}+\tilde{C}_\Omega \|c \|_{L^{\infty}(\Omega)}^{2{p}+2}}\\
\leq& \disp{C_\Omega\int_{\Omega} |\nabla c |^{2{p}-2}|D^2 \varphi|^2 M_*^{2}+\tilde{C}_\Omega M_*^{2{p}+2},}\\
\end{array}
$$
which immediately leads to
\begin{equation}\int_\Omega  |\nabla c |^{2{p}-2}|D^2c |^2\geq \frac{1}{C_\Omega M_*^{2}}\int_\Omega |\nabla c |^{2{p}+2}-C_5
\label{cz2.5ghsssjddfffu4ghjuk81}
\end{equation}
with some $C_5>0.$ Here $M_*$ is given by  \dref{ddfgczhhhh2.5sddddghju48cfg924ghyuji}.
Collecting \dref{ssddfffcz2.5ghju48156}--\dref{cz2.5ghsssjddfffu4ghjuk81}, there exists a positive constant $C_6>0$ such that
\begin{equation}
\begin{array}{rl}
&\disp{\int_\Omega (u \cdot\nabla  c )\nabla\cdot( |\nabla c |^{2{p}-2}\nabla c )}
\\
\leq&\disp{\frac{({p}-1)}{2{{p}^2}}\int_{\Omega}\left|\nabla |\nabla c |^{p}\right|^2
+\frac{1}{16C_\Omega M_*^{2}}\int_\Omega |\nabla c |^{2{p}+2}+C_4}\\
&+\disp{\frac{1}{8}\int_\Omega  |\nabla c |^{2{p}-2}|D^2c |^2
+\frac{1}{16C_\Omega M_*^{2}} |\nabla c |^{2{p}+2}+C_2}\\
\leq&\disp{\frac{({p}-1)}{2{{p}^2}}\int_{\Omega}\left|\nabla |\nabla c |^{p}\right|^2
+\frac{1}{4}\int_\Omega  |\nabla c |^{2{p}-2}|D^2c |^2
+C_6,}\\
\end{array}
\label{ssddfffcz2.5ghjeertyyyu48156}
\end{equation}
which together with \dref{789hjui909klopji115} implies that
$$
\begin{array}{rl}
\disp{\frac{1}{{2{p}}}\frac{d}{dt}\|\nabla c \|^{{{2{p}}}}_{L^{{2{p}}}(\Omega)}
+\frac{1}{4}\int_\Omega  |\nabla c |^{2{p}-2}|D^2c |^2+\frac{1}{2}\int_{\Omega} |\nabla c |^{2{p}}}
\leq&\disp{
C_7~~\mbox{for all}~~ t\in(0,T_{max}).}
\end{array}
$$
An elementary
computation shows that
\begin{equation}\label{hjui909kssddlopji115}
\disp{\|\nabla c (\cdot, t)\|_{L^{2{p}}(\Omega)}\leq C_8~~\mbox{for all}~~ t\in(0,T_{max}).}
\end{equation}
This completes the proof of Lemma \ref{lemma4563025xxhjkloghyui}.
\end{proof}
In the following we will deal with the case $\kappa\in\mathbb{R}$, which is relatively difficult. To this end, we recalling a functional inequality,
which is well-established in  Lemma 2.2 of \cite{Winklessddffr444sdddssdff51215} (see also  Lemma
3.4 of \cite{Nagaixcdf791}) upon the two-dimensional Moser-Trudinger inequality.
%To close the energy, we need to absorb the two integrals on the right hand side of (3.8). If m
%is less than 2, we shall add one more power of n¦Å in the estimation.

\begin{lemma}\label{lsssddsssddemma3.5}
 Let $\Omega \subset \mathbb{R}^2$ be a bounded domain with smooth boundary.
 Then for all $\varepsilon > 0$ there exists $M = M(\varepsilon , \Omega ) > 0$ such that if $0 \not \equiv \phi \in C_0(\Omega )$
nonnegative and $\psi \in W^{1,2}(\Omega )$, then for each $a > 0$,
\begin{equation}
\begin{array}{rl}\label{ddfsssfssssf44ss55dssssdd5eqx45xx12112ccgghh}
\disp\int_\Omega\phi | \psi | \leq \disp\frac{1}{a}\disp\int_\Omega\phi \ln \frac{\phi}{\bar{\phi}}+
\frac{(1 + \varepsilon )a}{8\pi}\cdot\biggl\{\int_\Omega\phi\biggr\}\cdot
\disp\int_\Omega| \nabla \psi |^2+M a\cdot\biggl\{ \int_\Omega\phi\biggr\}\cdot\biggl\{ \int_\Omega| \psi |\biggr\}^2
+Ma\int_\Omega\phi ,
\end{array}
\end{equation}
where $\bar{\phi} := \frac{1}{| \Omega |}\int_\Omega\phi$. %and let
%$0 \not \equiv \phi \in C_0(\Omega )$ be nonnegative. Then for any choice of $\varepsilon > 0$,
%\begin{equation}
%\begin{array}{rl}\label{ddffssssf44ss55dssssdd5eqx45xx12112ccgghh}
%&\int_\Omega
%\phi ln(\phi +1)\\
%\leq&\frac{1 + \varepsilon}{2\pi}\cdot
%\biggl\{ \int_\Omega\phi\biggr\}\cdot
%\int_\Omega\frac{| \nabla \phi |^2}{(\phi + 1)^2}+4M\cdot\biggl\{ \int_\Omega\phi\biggr\}^3\\
%&+\Biggl\{M - \ln \biggl\{\frac{1}{| \Omega |}\int_\Omega\phi\biggr\}\Biggr\}\cdot\int_\Omega\phi ,
%\end{array}
%\end{equation}
%where $M = M(\varepsilon,\Omega )$ > 0 is as in Lemma 2.2.
\end{lemma}

The latter lemma also entails the following functional inequality exclusively referring to a single function, and relating its size in $n\ln n$ to the $H^1$ norm of some
logarithmic derivative.

\begin{lemma}\label{lsssddddemma3.5}
 Let $\Omega \subset \mathbb{R}^2$ be a bounded domain with smooth boundary, and let
$0 \not \equiv \phi \in C_0(\Omega )$ be nonnegative. Then for any choice of $\varepsilon > 0$,
\begin{equation}
\begin{array}{rl}\label{ddfff44ss55dssssdd5eqx45xx12112ccgghh}
&\disp{\disp\int_\Omega
\phi \disp\ln(\phi +1)}\\
\leq&\disp{\disp\frac{1 + \varepsilon}{2\pi}\cdot
\biggl\{ \int_\Omega\phi\biggr\}\cdot
\disp\int_\Omega\frac{| \nabla \phi |^2}{(\phi + 1)^2}+4M\cdot\biggl\{ \int_\Omega\phi\biggr\}^3}\\
&+\disp{\Biggl\{M - \ln \biggl\{\frac{1}{| \Omega |}\disp\int_\Omega\phi\biggr\}\Biggr\}\cdot\disp\int_\Omega\phi},
\end{array}
\end{equation}
where $M = M(\varepsilon,\Omega )> 0 $ is as in Lemma \ref{lsssddsssddemma3.5}.
\end{lemma}

%In order to estimate the boundedness of $\nabla c $ in $L^2$, let us detect the evolution of $\int_\Omega|\nabla c |^2$.

%we perform another straightforward
%testing procedure to the second equation?in \dref{1.dddduiikkldffdffg1}, \dref{ccvvx1.73142sdd6677gg}--\dref{ccvvx1.ddfff731426677gg}. As this evidently needs to be concerned with
%the time evolution of a functional containing spatial derivatives of c, unlike the above situation
%the result will now involve the fluid velocity, where thanks to (3.2), however, this dependence
%will take the following favorable form.

To deduce the eventual $L^p$ boundedness of $n$ for any $p > 1$, we need to derive further regularity properties of the solutions to \dref{1.dddduiikkldffdffg1}, \dref{ccvvx1.73142sdd6677gg}--\dref{ccvvx1.ddfff731426677gg}. To this end,
s a direct application of Lemma \ref{fvfgsdfggfflemma45} and some careful analysis, we have the following corollary, which will be used in the sequel to obtain the estimates of $\nabla c$.

\begin{lemma}\label{ssdddlemmaghjdfgggffggssddgghhmk4563025xxhjklojjkkk}
Let $\kappa\in \mathbb{R}$ as well as $|S(n)|\leq C_S(1+n)^{-\alpha}$ with $\alpha\geq-1$.
% and $S = 0$ on $\partial\Omega$.
 %There is $C>0$  such that %for any ¦Å > 0
% the estimates
Then there exists $\varrho_1>0$ %independent of $\varepsilon$
such that the solution of \dref{1.dddduiikkldffdffg1}, \dref{ccvvx1.73142sdd6677gg}--\dref{ccvvx1.ddfff731426677gg} satisfies
\begin{equation}
\begin{array}{rl}
\disp\frac{1}{{2}}\disp\frac{d}{dt}\|\nabla{c  }\|^{{{2}}}_{L^{{2}}(\Omega)}+\frac{1}{2}
\int_{\Omega} |\Delta c  |^2+ \int_{\Omega} | \nabla c  |^2
\leq&\disp{\varrho_1\|\nabla u  \|_{L^{2}(\Omega)}^2+\varrho_1~~\mbox{for all}~~ t\in(0,T_{max}).}\\
\end{array}
\label{czfvgb2.5ghhjuyuccvviihjj}
\end{equation}
%for all $t\in(0, T_{max})$.
%
%Moreover, for $t\in(0, T_{max}-\tau)$, it holds that
%%%In addition,
%%%%Then
%%%for each $T\in(0, T_{max})$,
%%one can find a constant $C > 0$  such that
%\begin{equation}
%\begin{array}{rl}
%&\disp{\int_{t}^{t+\tau}\int_{\Omega} \left[   |\nabla {c  }|^4+ |\nabla {m }|^4+|\nabla {u  }|^2+| {u  }|^{\frac{10}{3}}\right]\leq C,}\\
%\end{array}
%\label{bnmbncz2.5ghhjuyuivvbnnihjj}
%\end{equation}
%where $\tau$ is the same as \dref{cz2.5ghju48cfg924vbhu}.
%\begin{equation}
%\tau:=\min\{1,\frac{1}{2}T_{max}\}.
%\label{cz2.5ghju48cfg924vbhu}
%\end{equation}
\end{lemma}
\begin{proof}
%Testing  the second equation in \dref{1.dddduiikkldffdffg1}, \dref{ccvvx1.73142sdd6677gg}--\dref{ccvvx1.ddfff731426677gg} by $-\Delta c  $ and using $\nabla\cdot u =0$ and \dref{ddfgczhhhh2.5ghju48cfg924ghyuji} yields that for some positive constant $C_1$ such that
We multiply the second equation in \dref{1.dddduiikkldffdffg1} by $-\Delta c  $ and integrate by parts to see that
%taking ${c }$ as the test function for the second  equation of \dref{1.dddduiikkldffdffg1}, \dref{ccvvx1.73142sdd6677gg}--\dref{ccvvx1.ddfff731426677gg}, using $\nabla\cdot u =0$ and \dref{ffggg1.ffggvddfghhghhhhbbhhjjjnxxccvvn1}, with the help of the H\"{o}lder  inequality yields  that
\begin{equation}
\begin{array}{rl}
&\disp\frac{1}{{2}}\disp\frac{d}{dt}\|\nabla{c  }\|^{{{2}}}_{L^{{2}}(\Omega)}+
\int_{\Omega} |\Delta c  |^2+ \int_{\Omega} | \nabla c  |^2
\\
=&\disp{-\int_{\Omega} m  \Delta c  +\int_{\Omega} (u  \cdot\nabla c )\Delta c  }
\\
=&\disp{-\int_{\Omega} m  \Delta c -\int_{\Omega}\nabla c  \nabla (u  \cdot\nabla c )}
\\
=&\disp{-\int_{\Omega} m  \Delta c -\int_{\Omega}\nabla c  (\nabla u  \cdot\nabla c )~~\mbox{for all}~~ t\in(0,T_{max}),}
\end{array}
\label{hhxxcsssdfvvjjczddfdddfff2.5}
\end{equation}
where we have used the fact that
$$
\disp{\int_{\Omega}\nabla c  \cdot(D^2 c  \cdot u  )
=\frac{1}{2}\int_{\Omega}  u  \cdot\nabla|\nabla c  |^2=0
~~\mbox{for all}~~ t\in(0,T_{max}).}
$$
%\begin{equation}
%\begin{array}{rl}
%\disp\int_{\Omega}(u\cdot\nabla c) \Delta c=&\disp{-\int_{\Omega} \nabla c\cdot\nabla(u\cdot\nabla c) }\\
%=&\disp{-\int_{\Omega} \nabla c\cdot\nabla(u\cdot\nabla c)-\int_{\Omega} \nabla c\cdot\nabla(u\cdot\nabla c) ~~\mbox{for all}~~ t\in(0, T_{max}),}\\
%%\leq&\disp{\|n \|_{L^{\frac{6}{5}}(\Omega)}\|c \|_{L^{6}(\Omega)}.}\\
%\end{array}
%\label{dddfggaassshhxxcdfvvjjcz2.5}
%\end{equation}
%\begin{equation}
%\begin{array}{rl}
%\disp\frac{1}{{2}}\disp\frac{d}{dt}\|{\nabla c }\|^{{{2}}}_{L^{{2}}(\Omega)}+
%\int_{\Omega} |\Delta c  |^2+\int_{\Omega} |\nabla c |^2=&\disp{-\int_{\Omega} m\Delta c  +\int_{\Omega}(u\cdot\nabla c) \Delta c  ~~\mbox{for all}~~ t\in(0, T_{max}),}\\
%%\leq&\disp{\|n \|_{L^{\frac{6}{5}}(\Omega)}\|c \|_{L^{6}(\Omega)}.}\\
%\end{array}
%\label{dddfggaassshhxxcdfvvjjcz2.5}
%\end{equation}
On the other hand,  the Young inequality in conjunction with  \dref{ddczhjjjj2.5ghju48cfg9ssdd24} ensures that
%On the other hand,  by the Young inequality and \dref{ddczhjjjj2.5ghju48cfg9ssdd24},
\begin{equation}
\begin{array}{rl}
\disp-\int_{\Omega} m \Delta c \leq&\disp{\int_{\Omega} m^2 +\frac{1}{4}\int_{\Omega}|\Delta c |^2  }\\
\leq&\disp{|\Omega|\| m_0\|^2_{L^\infty(\Omega)}+\frac{1}{4}\int_{\Omega}|\Delta c |^2  ~~\mbox{for all}~~ t\in(0, T_{max}).}\\
%\leq&\disp{\|n \|_{L^{\frac{6}{5}}(\Omega)}\|c \|_{L^{6}(\Omega)}.}\\
\end{array}
\label{ssdddaassshhxxcdfvvjjcz2.5}
\end{equation}
In the last summand in \dref{hhxxcsssdfvvjjczddfdddfff2.5}, we use the Cauchy-Schwarz inequality to obtain
\begin{equation}
\begin{array}{rl}
\disp-\int_{\Omega}\nabla c   (\nabla u  \cdot\nabla c )
\leq&\disp{\|\nabla u  \|_{L^{2}(\Omega)}\|\nabla c  \|_{L^{4}(\Omega)}^2~~\mbox{for all}~~ t\in(0,T_{max}).}
\end{array}
\label{hhxxcsssdfvvjjcddssdddffzddfdddfff2.5}
\end{equation}
Now thanks to \dref{ddfgczhhhh2.5sddddghju48cfg924ghyuji} and in view of the Gagliardo-Nirenberg inequality, we can find $C_1> 0$ and  $C_2> 0$ fulfilling
%integrate by parts to find that
%\begin{equation}
%\begin{array}{rl}
%\disp \|\nabla c \|_{L^{4}(\Omega)}^2\leq\disp{C_{6}\|\Delta c  \|_{L^{2}(\Omega)}\|\nabla c \|_{L^{2}(\Omega)}
%~~\mbox{for all}~~ t\in(0,T_{max}).}\\
%\end{array}
%\label{hhxxcsssdfvvjjcddfffddffzddfdddfff2.5}
%\end{equation}
%
%Meanwhile, we can further use Gagliardo-Nirenberg inequality and the elliptic regularity (\cite{Gilbarg4441215}) to conclude that for some $C_{6}> 0$,
\begin{equation}
\begin{array}{rl}
\disp \|\nabla c  \|_{L^{4}(\Omega)}^2\leq&\disp{C_{1}\|\Delta c  \|_{L^{2}(\Omega)}\|c  \|_{L^{\infty}(\Omega)}+C_{1}\|c  \|_{L^{\infty}(\Omega)}^2}\\
\leq&\disp{C_{2}\|\Delta c  \|_{L^{2}(\Omega)}+C_{2}
~~\mbox{for all}~~ t\in(0,T_{max}),}\\
\end{array}
\label{hhxxcsssdfvvjjcddfffddffzddfdddfff2.5}
\end{equation}
where by the  Young  inequality,
%Since clearly (3.8)entails that
%This together with the  Young  inequality, yields
\begin{equation}
\begin{array}{rl}
&\disp-\int_{\Omega}\nabla c  (\nabla u  \cdot\nabla c )
\\
\leq&\disp{\|\nabla u  \|_{L^{2}(\Omega)}[C_{2}\|\Delta c  \|_{L^{2}(\Omega)}+C_{2}]}
\\
%\leq&\disp{C_{6}\|\nabla u \|_{L^{2}(\Omega)}\|\Delta c  \|_{L^{2}(\Omega)}\|\nabla c \|_{L^{2}(\Omega)}}
%\\
\leq&\disp{C_{2}^2\|\nabla u  \|_{L^{2}(\Omega)}^2
+\frac{1}{4}\|\Delta c  \|_{L^{2}(\Omega)}^2+C_3~~\mbox{for all}~~ t\in(0,T_{max}).}
\end{array}
\label{hhxxcsssdfvvjjcddffzddfdddfff2.5}
\end{equation}
Inserting \dref{ssdddaassshhxxcdfvvjjcz2.5} and \dref{hhxxcsssdfvvjjcddffzddfdddfff2.5} into  \dref{hhxxcsssdfvvjjczddfdddfff2.5}, we have
\begin{equation}
\begin{array}{rl}
\disp\frac{1}{{2}}\disp\frac{d}{dt}\|\nabla{c  }\|^{{{2}}}_{L^{{2}}(\Omega)}+\frac{1}{2}
\int_{\Omega} |\Delta c  |^2+ \int_{\Omega} | \nabla c  |^2
\leq&\disp{C_{2}^2\|\nabla u  \|_{L^{2}(\Omega)}^2+C_4.}
\end{array}
\label{hhxxcsssdfvvjjczddfssdddddfff2.5}
\end{equation}
As a consequence of \dref{hhxxcsssdfvvjjczddfssdddddfff2.5}, \dref{czfvgb2.5ghhjuyuccvviihjj} is valid by a choice of
$\varrho_1:= \max\{C_{2}^2, C_4\}$.
 \end{proof}
 %which have been derived in Lemma 3.2 of \cite{Wangjjk5566ddfggghjjkk1} relying on the argument of
%Lemma 2.7 of \cite{Winklessddffr444sdddssdff51215}. And therefore, we omit it.

Now, we try to analyze the evolution of $\int_{\Omega}  | {u }|^2$, which contributes to absorbing
$\int_{\Omega}  |\nabla {u }|^2$ on the right-hand side of \dref{czfvgb2.5ghhjuyuccvviihjj}.
And the proof is almost precisely like that of Lemma 2.7 of \cite{Winklessddffr444sdddssdff51215} (see also Lemma 3.2 of \cite{Wangjjk5566ddfggghjjkk1}) except that the $L^1$-boundedness of $n $, instead of the $L^1$-conservation therein, is now encountered. And therefore, we omit it.
\begin{lemma}\label{33444lemmaghjssddgghhmk4563025xxhjklojjkkk}
Let $\kappa\in \mathbb{R}$ as well as $|S(n)|\leq C_S(1+n)^{-\alpha}$ with $\alpha\geq-1$.
%Let $|S(n)|\leq C_S(1+n)^{-\alpha}$ with $\alpha\geq-1$.
Then there exists $\varrho_2>0$  such that the solution of \dref{1.dddduiikkldffdffg1}, \dref{ccvvx1.73142sdd6677gg}--\dref{ccvvx1.ddfff731426677gg} satisfies
\begin{equation}
\begin{array}{rl}
&\disp{\frac{d}{dt}\int_{\Omega}  | {u }|^2+\int_{\Omega}  |\nabla {u }|^2\leq \varrho_2 \int_{\Omega} {n }\ln{n }+\varrho_2~~~ \mbox{for all}~ t\in(0,T_{max}).}\\
\end{array}
\label{333333czfvgb2.5ghhjusssyuccvviihjj}
\end{equation}
%for all $t\in(0, T_{max})$.
%In addition,
%%Then
%for each $t>0$, one can find a constant $C > 0$  such that
%\begin{equation}
%\begin{array}{rl}
%&\disp{\int_{t}^{t+1}\int_{\Omega} \left[ |\nabla {u }|^2\right]\leq C.}\\
%\end{array}
%\label{333333bnmbncz2.5ghhjuyuivvbnnihjj}
%\end{equation}
\end{lemma}
Collecting the result of Lemma \ref{fvfgsdfggfflemma45}, Lemma \ref{lsssddddemma3.5} to Lemma \ref{33444lemmaghjssddgghhmk4563025xxhjklojjkkk} and by using some careful analysis, we thus establish the following energy-type
inequality in the case of $\kappa\in\mathbb{R}$.

\begin{lemma}\label{lsssemma3.5}
Let $\kappa\in \mathbb{R}$ as well as $|S(n)|\leq C_S(1+n)^{-\alpha}$ with $\alpha\geq-\frac{1}{2}$.
%Let $|S(n)|\leq C_S(1+n)^{-\alpha}$ with $\alpha\geq-\frac{1}{2}$.  Then %there exists $K(m_0) > 0$
% with the property
%that whenever  \dref{ccvvx1.731426677gg}  holds with
% $\disp\int_\Omega u_0 \leq m_0$,
one can find $C > 0$
such that

\begin{equation}
\begin{array}{rl}\label{ddfff44ss55ddd5eqx45xx12112ccgghh}
\disp\int_\Omega
\bigl[\ln(n  + 1)+  {u ^2}+ | \nabla c  |^2\bigr]
\leq &
C~~~~ \mbox{for all}~ t\in(0,T_{max})\\
\end{array}
\end{equation}
as well as
\begin{equation}
\begin{array}{rl}\label{ddfff44ss555eqx45xx12112ccgghh}
%\disp\biggl\{
 \disp\int_{(t-\tau)_+}^t\int_\Omega \bigl[{}{(n  + 1)\ln (n  + 1)}+ |\Delta c  |^2+  | \nabla u  |^2+\frac{| \nabla n  |^2}{(n  + 1)^2}\bigr]
\leq &  C~~~\mbox{for all}~~ t \in (0,T_{max}-\tau)\\
\end{array}
\end{equation}
and
\begin{equation}
\begin{array}{rl}\label{ssddfff44ss555eqx45xx12112ccgghh}
%\disp\biggl\{
 \disp\int_{(t-\tau)_+}^t\int_\Omega \bigl[ | \nabla c  |^4\bigr]
\leq &  C~~~\mbox{for all}~~ t \in (0,T_{max}-\tau),\\
\end{array}
\end{equation}
where $\tau$ is given by  \dref{jvgxddrg}.
%where $m := \disp\int_\Omega u_0$.
\end{lemma}
\begin{proof} Integrating by parts in the first equation
from \dref{1.dddduiikkldffdffg1}, we derive from  the Young inequality and Lemma \ref{fvfgsdfggfflemma45}  that there is $C_1>0$ such that
\begin{equation}\label{wwwddsssdffsssf44ss55dddlll5eqx45xx12112ccgghh}
\begin{array}{rl}
& \disp-\frac{d}{dt}\disp\int_\Omega
\ln(n  + 1) + \disp\int_\Omega\frac{| \nabla n  |^2}{(n  + 1)^2}\\
\leq&\disp \int_\Omega \frac{n }{(n  + 1)^2}S(n )\nabla n \cdot\nabla c \\
=&\disp \int_\Omega \nabla \int_0^{n }\frac{\tau}{(\tau + 1)^2}S(\tau)d\tau \cdot \nabla c \\
%=&\disp \int_\Omega \nabla \int_0^{n }{ C_S(1+\tau)^{-\alpha}}{}d\tau \cdot \nabla c \\
=&\disp -\int_\Omega \int_0^{n }\frac{\tau}{(\tau + 1)^2}S(\tau)d\tau \Delta c \\
\leq&\disp \int_\Omega C_S(1+n )^{-1-\alpha}|\Delta c |\\
\leq&\disp\frac{\pi}{64\varrho_1\varrho_2\|n_0\|_{L^1(\Omega)}}\int_\Omega|\Delta c |^2+\frac{16\varrho_1\varrho_2\|n_0\|_{L^1(\Omega)}}{\pi}\int_\Omega C_S(1+n )^{-2-2\alpha}\\
\leq&\disp \frac{\pi}{64\varrho_1\varrho_2\|n_0\|_{L^1(\Omega)}}\int_\Omega|\Delta c |^2+\frac{16\varrho_1\varrho_2\|n_0\|_{L^1(\Omega)}}{\pi}\int_\Omega C_S(1+n )\\
\leq&\disp \frac{\pi}{64\varrho_1\varrho_2\|n_0\|_{L^1(\Omega)}}\int_\Omega|\Delta c |^2+C_1~~~\mbox{for all}~~~ t\in(0,T_{max})\\
%\leq& \disp\delta\disp\int_\Omega{}{(n  + 1)^2}
%+\disp\frac{1}{4\delta}\disp\int_\Omega |\Delta c |^{\frac{2}{1+\alpha}}
%~~~\mbox{for all}~~~ t > 0
\end{array}
\end{equation}
with some $C_1>0,$ where $\varrho_1$ and $\varrho_2$ are given by Lemma \ref{ssdddlemmaghjdfgggffggssddgghhmk4563025xxhjklojjkkk} and Lemma \ref{33444lemmaghjssddgghhmk4563025xxhjklojjkkk}, respectively.
Here we have used the fact that $-2-2\alpha\leq1$ by using $\alpha\geq-\frac{1}{2}$.
Taking  an evident linear combination of the inequalities provided by \dref{czfvgb2.5ghhjuyuccvviihjj} as well as \dref{333333czfvgb2.5ghhjusssyuccvviihjj} and  \dref{wwwddsssdffsssf44ss55dddlll5eqx45xx12112ccgghh}, one can
obtain
\begin{equation}
\begin{array}{rl}
& \disp-\frac{16\varrho_1\varrho_2\|n_0\|_{L^1(\Omega)}}{\pi}\frac{d}{dt}\disp\int_\Omega
\ln(n  + 1) + \frac{16\varrho_1\varrho_2\|n_0\|_{L^1(\Omega)}}{\pi}\disp\int_\Omega\frac{| \nabla n  |^2}{(n  + 1)^2}\\
%\leq&\disp \int_\Omega \frac{n }{(n  + 1)^2}S(n )\nabla n \cdot\nabla c \\
&\disp+\frac{1}{{2}}\disp\frac{d}{dt}\|\nabla{c  }\|^{{{2}}}_{L^{{2}}(\Omega)}+\frac{1}{4}
\int_{\Omega} |\Delta c  |^2+ \int_{\Omega} | \nabla c  |^2+4\varrho_1\varrho_2\int_{\Omega} {n }\ln{(1+n )}\\
&\disp+2\varrho_1\frac{d}{dt}\int_{\Omega}  | {u }|^2+\varrho_1\int_{\Omega} |\nabla {u }|^2\\
\leq&\disp{ 4\varrho_1\varrho_2\int_{\Omega} {n }\ln{(1+n )}+4\varrho_1\varrho_2\int_{\Omega} {n }\ln{n }+C_2}\\
=&\disp{ 8\varrho_1\varrho_2\int_{\Omega} {n }\ln{(1+n )}+C_2
~~\mbox{for all}~ ~t\in(0, T_{max})}\\
\end{array}
\label{hhxxcdfvvjjczddfgghjj2ffghhkl.5ddsdfrfggggdrr}
\end{equation}
with $$C_2=2\varrho_1\varrho_2+\varrho_1+C_1\frac{16\varrho_1\varrho_2\|n_0\|_{L^1(\Omega)}}{\pi}.$$
In order to appropriately estimate the integral on the right-hand side herein, we rely
on Lemma \ref{lsssddsssddemma3.5} and Lemma \ref{fvfgsdfggfflemma45}, whence it is possible to
find $C_3 > 0$ such that
\begin{equation}
\begin{array}{rl}\label{ddfff44ss55sssdssssdd5eqx45xx12112ccgghh}
\disp{\disp\int_\Omega
n  \disp\ln(n  +1)}
\leq&\disp{\disp\frac{1}{\pi}\|n_{0}\|_{L^1(\Omega)}8\varrho_1\varrho_2
\disp\int_\Omega\frac{| \nabla  n  |^2}{( n  + 1)^2}+8\varrho_1\varrho_2 C_3},
\end{array}
\end{equation}
which immediately implies that
\begin{equation}
\begin{array}{rl}\label{ddfff44ss5sss5sssdssssdd5eqx45xx12112ccgghh}
\disp{\disp8\varrho_1\varrho_2\int_\Omega
n  \disp\ln(n  +1)}
\leq&\disp{\disp\frac{ 8\varrho_1\varrho_2}{\pi}\|n_{0}\|_{L^1(\Omega)}
\disp\int_\Omega\frac{| \nabla  n  |^2}{( n  + 1)^2}+C_3}.
\end{array}
\end{equation}
This combined with \dref{hhxxcdfvvjjczddfgghjj2ffghhkl.5ddsdfrfggggdrr} and Lemma \ref{fvfgsdfggfflemma45} yields to
\begin{equation}
\begin{array}{rl}
& \disp\frac{16\varrho_1\varrho_2\|n_0\|_{L^1(\Omega)}}{\pi}\frac{d}{dt}\disp \bigl[\int_\Omega
n  -\int_\Omega
\ln(n  + 1)\bigr] + \frac{8\varrho_1\varrho_2\|n_0\|_{L^1(\Omega)}}{\pi}\disp\int_\Omega\frac{| \nabla n  |^2}{(n  + 1)^2}\\
%\leq&\disp \int_\Omega \frac{n }{(n  + 1)^2}S(n )\nabla n \cdot\nabla c \\
&\disp+\frac{1}{{2}}\disp\frac{d}{dt}\|\nabla{c  }\|^{{{2}}}_{L^{{2}}(\Omega)}+\frac{1}{4}
\int_{\Omega} |\Delta c  |^2+ \int_{\Omega} | \nabla c  |^2+4\varrho_1\varrho_2\int_{\Omega} {n }\ln{(1+n )}\\
&\disp+2\varrho_1\frac{d}{dt}\int_{\Omega}  | {u }|^2+\varrho_1\int_{\Omega} |\nabla {u }|^2+\int_\Omega
n \\
%\leq&\disp{ 4\varrho_1\varrho_2\int_{\Omega} {n }\ln{(1+n )}+4\varrho_1\varrho_2\int_{\Omega} {n }\ln{n }+C_2}\\
\leq&\disp{ C_4
~~\mbox{for all}~ ~t\in(0, T_{max})}\\
\end{array}
\label{hhxxcdfvvjjczddfgghjj2ffghhkl.5ddsdfsssrfggggdrr}
\end{equation}
with some positive constant $C_4.$
%Observing that for any $\varepsilon\in(0,1),$$\|n +\varepsilon\|_{L^{{1}}(\Omega)}\leq\|n_0\|_{L^{{1}}(\Omega)}+|\Omega|$,
%due to \dref{ddfgczhhhh2.5ghju48cfg924ghyuji} and $m+\alpha>\frac{10}{9}$,
%we utilize the Gagliardo-Nirenberg inequality to see that there exists a positive constant $C_{11}$ such that for all $\varepsilon\in(0, 1)$,
%$$
%\disp\frac{1}{{\alpha(1+\alpha)}}\|{n }+\varepsilon\|^{{{1+\alpha}}}_{L^{{1+\alpha}}(\Omega)}
%\leq\disp{C_{11}\left(\int_{\Omega}(n +\varepsilon)^{m+\alpha-2}|\nabla n |^2\right)^{\frac{3\alpha}{3(m+\alpha)-1}}+C_{11}~~ \mbox{for all}~~ t>0.}\\
%$$
Since the Poincar\'{e} inequality and Lemma \ref{fvfgsdfggfflemma45} provide $C_{5}>0$
such that
%Recalling  Lemma \ref{fvfgsdfggfflemma45}, according to the Poincar\'{e} inequality, we get
 \begin{equation}
 \begin{array}{rl}
0\leq&\disp{\int_\Omega
n  -\int_\Omega
\ln(n  + 1)+\int_{\Omega}  {u ^2}+\int_{\Omega} | \nabla c  |^2}\\
 \leq&\disp{\frac{1}{2C_5}(\int_{\Omega} |\nabla {u }|^2+\int_\Omega
n +\int_{\Omega} | \nabla c  |^2)~~~\mbox{for all}~ ~t\in(0, T_{max}),}\\
\end{array}\label{vgbccvssbbfssdfeerfgghhjussdkkkulsssoollgghhhyhh}
\end{equation}
%with some $C_{5}>0$.
%$C_{5}>0$
and thus, we infer from Lemma \ref{fvfgsdfggfflemma45} and \dref{hhxxcdfvvjjczddfgghjj2ffghhkl.5ddsdfsssrfggggdrr} %\dref{vgbccvssbbffeerfgghhjussdkkkulsssoollgghhhyhh1}
%and \dref{vgbccvssbbfssdfeerfgghhjussdkkkulsssoollgghhhyhh}
that there exist $C_{6}> 0$ and $C_{7}> 0$ such that %for all $\varepsilon\in(0, 1)$,
\begin{equation}
 \begin{array}{rl}
 & \disp\frac{d}{dt}\disp \bigl\{\frac{16\varrho_1\varrho_2\|n_0\|_{L^1(\Omega)}}{\pi}[\int_\Omega
n  -\int_\Omega
\ln(n  + 1)]+2\varrho_1| {u }|^2+\frac{1}{{2}}\int_{\Omega} | \nabla c  |^2\bigr\}\\
 & \disp+\disp C_{6}\bigl\{\frac{16\varrho_1\varrho_2\|n_0\|_{L^1(\Omega)}}{\pi}[\int_\Omega
n  -\int_\Omega
\ln(n  + 1)]+2\varrho_1| {u }|^2+\frac{1}{{2}}\int_{\Omega} | \nabla c  |^2\bigr\}\\
 & \disp+ \frac{8\varrho_1\varrho_2\|n_0\|_{L^1(\Omega)}}{\pi}\disp\int_\Omega\frac{| \nabla n  |^2}{(n  + 1)^2}\\
%\leq&\disp \int_\Omega \frac{n }{(n  + 1)^2}S(n )\nabla n \cdot\nabla c \\
&\disp+\frac{1}{4}
\int_{\Omega} |\Delta c  |^2+\frac{1}{2} \int_{\Omega} | \nabla c  |^2+4\varrho_1\varrho_2\int_{\Omega} {n }\ln{(1+n )}\\
&\disp+\frac{\varrho_1}{2}\int_{\Omega} |\nabla {u }|^2+\frac{1}{2}\int_\Omega
n \\
%\leq&\disp{ 4\varrho_1\varrho_2\int_{\Omega} {n }\ln{(1+n )}+4\varrho_1\varrho_2\int_{\Omega} {n }\ln{n }+C_2}\\
\leq&\disp{ C_7
~~\mbox{for all}~ ~t\in(0, T_{max}).}\\
\end{array}\label{vgbccvssbbffeerfgghhjussdkkkulsssoollgghhhyhh}
\end{equation}
%As from Lemma \ref{lemmaghjssddgghhmk4563025xxhjklojjkkk} and our boundedness  on $\int_{\Omega}\frac{|\nabla  w |^2}{w }$ it is clear that there exists $C_{15}> 0$ such
%that $\int_t^{t+1}(2\int_{\Omega}\frac{|\nabla  w |^2}{w }+C_{14})ds \leq C_{15}$ for all $t\in (0, T_{max})$,
As a consequence of Lemma \ref{lemma630} we see that \dref{vgbccvssbbffeerfgghhjussdkkkulsssoollgghhhyhh} implies the inequalities
%Thanks to Lemma \ref{lemma630} this firstly entails
\dref{ddfff44ss55ddd5eqx45xx12112ccgghh}--\dref{ddfff44ss555eqx45xx12112ccgghh}.
 %by integrating \dref{vgbccvssbbffeerfgghhjussdkkkulsssoollgghhhyhh} in time.
%Here we can use the Gagliardo¨CNirenberg inequality to find C1>0fulfilling
%Therefore, by means of Lemma \ref{lemma630}, \dref{11111czfvgb2.5ghhjuyuiihjj}--\dref{111113.10gsssdghhjuulooll} can be obtained.
%And \dref{bnmbncz2.5ghhjuyuiihjj}--\dref{cvffvbgvvcz2.5ghhjuyuiihjj} can be followed by integrating the inequality \dref{vgbccvssbbffeerfgghhjussdkkkulsssoollgghhhyhh}.

Finally, thanks to \dref{ddfgczhhhh2.5sddddghju48cfg924ghyuji}, by means of the Gagliardo-Nirenberg inequality we infer the existence of $C_8> 0,C_9>0$ and  $C_{10}> 0$ fulfilling
%Now thanks to \dref{ddfgczhhhh2.5sddddghju48cfg924ghyuji} and in view of the Gagliardo-Nirenberg inequality, we can find $C_8> 0,C_9>0$ and  $C_{10}> 0$ fulfilling
%integrate by parts to find that
%\begin{equation}
%\begin{array}{rl}
%\disp \|\nabla c \|_{L^{4}(\Omega)}^2\leq\disp{C_{6}\|\Delta c  \|_{L^{2}(\Omega)}\|\nabla c \|_{L^{2}(\Omega)}
%~~\mbox{for all}~~ t\in(0,T_{max}).}\\
%\end{array}
%\label{hhxxcsssdfvvjjcddfffddffzddfdddfff2.5}
%\end{equation}
%
%Meanwhile, we can further use Gagliardo-Nirenberg inequality and the elliptic regularity (\cite{Gilbarg4441215}) to conclude that for some $C_{6}> 0$,
\begin{equation}
\begin{array}{rl}
\disp \disp\int_{(t-\tau)_+}^t\|\nabla c  \|_{L^{4}(\Omega)}^4\leq&\disp{C_{8}\disp\int_{(t-\tau)_+}^t\bigl[\|\Delta c  \|_{L^{2}(\Omega)}^2\|c  \|_{L^{\infty}(\Omega)}^2+\|c  \|_{L^{\infty}(\Omega)}^4\bigr]}\\
\leq&\disp{C_{9}\disp\int_{(t-\tau)_+}^t[\|\Delta c  \|_{L^{2}(\Omega)}+1]}\\
\leq&\disp{C_{10}
~~\mbox{for all}~~ t\in(0,T_{max}),}\\
\end{array}
\label{sssshhxxcsssdfvvjjcddfffddffzddfdddfff2.5}
\end{equation}
and thereby proves \dref{ssddfff44ss555eqx45xx12112ccgghh}.
%This, together with the  Young  inequality, yields
%The case $\alpha=0$ can be proved similarly and easily. Therefore, we omit it.
\end{proof}

Now thanks to the bound on $\nabla u$ and $m$
gained  Lemma \ref{lsssemma3.5} and Lemma \ref{fvfgsdfggfflemma45}, respectively, %once more making strong use of the absorptive
%zero-order term in the first equation
we can successively improve our knowledge on the regularity
properties of $\nabla c$ by means of the statement which will serve as the inductive step in a recursion.

%We now consider the time evolution of the integral $\int_\Omega|\nabla c |^{2{p}}$
%for any  $p> 1$.

 \begin{lemma}\label{9999lemma4563ddfff025xxhjkloghyui}
 Let $\kappa\in \mathbb{R}$ as well as $|S(n)|\leq C_S(1+n)^{-\alpha}$ with $\alpha\geq-\frac{1}{2}$.
Let $p\geq 2$ and $L > $0, and again abbreviate $\tau:=\min\{1,\frac{1}{2}T_{max}\}$. Then there exists
$C = C(p, L) > 0$ with the property that if
\begin{equation}
\int_t^{t+\tau}\int_{\Omega} {|\nabla c | ^{2{p}}} \leq L~~\mbox{for all}~~ t\in(0, T_{max}-\tau),
\label{cz2.5ghju48cjkklfg924ghyuji}
\end{equation}
then
\begin{equation}
\int_{\Omega} {|\nabla c | ^{2{p}}} \leq C~~\mbox{for all}~~ t\in(0, T_{max})
\label{cz2.5ghju48cjkddffffklfg924ghyuji}
\end{equation}
and
\begin{equation}
\int_t^{t+\tau}\int_{\Omega}\bigl[|\nabla c | ^{2{p}+2}+ |\nabla c |^{2{p}-2}|D^2c |^2+\left|\nabla |\nabla c |^{{p}}\right|^2\bigr]\leq C~~\mbox{for all}~~ t\in(0, T_{max}-\tau).
\label{223455cz2.5ghsssjddffffffgggu48cfg924ghyuji}
\end{equation}
\end{lemma}
\begin{proof}
%We begin with \dref{789hjui909klopji115}. To this end, in the following, we will estimate the right-side of  \dref{789hjui909klopji115}. Indeed,
 Firstly, applying
%Considering the fact that $\nabla c \cdot\nabla\Delta c   = \frac{1}{2}\Delta |\nabla c |^2-|D^2c |^2$,
%by a straightforward computation using the second equation in \dref{1.dddduiikkldffdffg1}, \dref{ccvvx1.73142sdd6677gg}--\dref{ccvvx1.ddfff731426677gg} and
the several integrations by parts,  one deduces
that
$$
\begin{array}{rl}
&\disp{\int_\Omega (u \cdot\nabla  c )\nabla\cdot( |\nabla c |^{2{p}-2}\nabla c )}
\\
=&\disp{-\int_\Omega |\nabla c |^{2{p}-2}\nabla c \cdot(\nabla u \cdot\nabla  c )
-\int_\Omega |\nabla c |^{2{p}-2}\nabla c \cdot(D^2c u )}\\
=&\disp{-\int_\Omega |\nabla c |^{2{p}-2}\nabla c \cdot(\nabla u \cdot\nabla  c )}
\end{array}
$$
for all $t\in(0,T_{max})$. Here we have used the fact that
$$-\int_\Omega |\nabla c |^{2{p}-2}\nabla c \cdot(D^2c u )=\frac{1}{2{p}}\int_\Omega(\nabla\cdot u ) |\nabla c |^{2{p}}=0.$$
Here, %since $|\Delta c | \leq\sqrt{2}|D^2c |$,
 by utilizing the H\"{o}lder  inequality as well as the Gagliardo-Nirenberg inequality and  the Young   inequality
 %, we can estimate
 we thus infer that there exist $C_1>0$ as well as $C_2>0$ and $C_3>0$ such that
% by Lemma3.2, in view of Lemma4.4
%\begin{equation}
%\begin{array}{rl}
%&\disp\int_\Omega w  |\nabla c |^{2m-2}\Delta c
%\\
%\leq&\disp{\sqrt{2}\int_\Omega w  |\nabla c |^{2m-2}|D^2c |}
%\\
%\leq&\disp{\frac{1}{4}\int_\Omega  |\nabla c |^{2m-2}|D^2c |^2+{2}\int_\Omega w^2  |\nabla c |^{2m-2}}
%\\
%\leq&\disp{\frac{1}{4}\int_\Omega  |\nabla c |^{2m-2}|D^2c |^2+{2}\| w \|^2_{L^\infty(\Omega)}\int_\Omega  |\nabla c |^{2m-2},}
%\end{array}
%\label{cz2.5ghju48hjuikl1}
%\end{equation}
%and, similarly,
$$
\begin{array}{rl}
&\disp-\int_\Omega |\nabla c |^{2{p}-2}\nabla c \cdot(\nabla u \cdot\nabla  c )
\\
\leq&\disp{\left(\int_\Omega |\nabla c |^{4{p}}\right)^{\frac{1}{2}}\left(\int_\Omega |\nabla u |^{2}\right)^{\frac{1}{2}}}
\\
\leq&\disp{C_1\bigl[\left(\int_{\Omega}\left|\nabla |\nabla c |^{{p}}\right|^2\right)^{\frac{1}{2}}\left(\int_\Omega |\nabla c |^{2{p}}\right)^{\frac{1}{2}}+\left(\int_\Omega |\nabla c |^{2{p}}\right)\bigr]\left(\int_\Omega |\nabla u |^{2}\right)^{\frac{1}{2}}}
\\
\leq&\disp{\frac{({p}-1)}{4{{p}^2}}\int_{\Omega}\left|\nabla |\nabla c |^{{p}}\right|^2+C_2\int_\Omega |\nabla c |^{2{p}}\int_\Omega |\nabla u |^{2}+C_3\int_\Omega |\nabla c |^{2{p}}}
\\
\end{array}
$$
for all $t\in(0,T_{max})$.  %Here as
Inserting the above inequality into \dref{789hjui909klopji115} and using \dref{cz2.5ghsssjddfffu4ghjuk81}, we conclude that there exists a positive constant $C_4$ such that
%Inserting the above inequality into \dref{789hjui909klopji115} and using \dref{cz2.5ghsssjddfffu4ghjuk81}, we have
\begin{equation}\label{789hjuilll909klopji115}
\begin{array}{rl}
&\disp{\frac{d}{dt}\|\nabla c \|^{{{2{p}}}}_{L^{{2{p}}}(\Omega)}
+\frac{{p}-1}{2{{p}}}\int_{\Omega}\left|\nabla |\nabla c |^{{p}}\right|^2}\\
&\disp{+\frac{p}{2}\int_\Omega  |\nabla c |^{2{p}-2}|D^2c |^2+\frac{p}{2C_\Omega M_*^{2}}\int_\Omega |\nabla c |^{2{p}+2}+{p}\int_{\Omega} |\nabla c |^{2{p}}}
\\
\leq&\disp{
C_4\int_\Omega|\nabla u |^2\int_{\Omega} |\nabla c |^{2{p}}+C_4~~\mbox{for all}~~ t\in(0,T_{max}),}
\end{array}
\end{equation}
%with some $C_4>0,$
where $M_*>0$ is the same as  Lemma \ref{fvfgsdfggfflemma45}.
%(see Lemma \ref{lemma3.3}) implies that there exists a positive constant $C_{11}$ such that
%$$
%\|u (\cdot, t)\|_{L^{4m+2}(\Omega)}\leq  C_{11}~~ \mbox{for all}~~ t\in(0,T_{max}),
%$$
%which together with \dref{3.10gghhjuuloollsdffffffgghhhy} yields to \dref{hjui909klopji115} by using Lemma \ref{lemma630}.
Thus, if we write $y(t) :=\|\nabla c (\cdot, t)\|^{{{2{p}}}}_{L^{{2{p}
}}(\Omega)}$
as well as  $\rho(t) =C_4\int_\Omega|\nabla u |^2$ and $h(t)=
C_4$
for $t\in(0,T_{max})$, so that,  \dref{789hjuilll909klopji115}   implies that
\begin{equation}
\begin{array}{rl}
y'(t)+\frac{p}{2C_\Omega M_*^{2}}\disp\int_\Omega |\nabla c |^{2{p}+2}+{p}y(t)
\leq&\disp{ \rho(t)y(t)+h(t)~\mbox{for all}~t\in(0,T_{max}).}\\
\end{array}
\label{ddfghgfhggddhjjjnkkll11cz2.5ghju48}
\end{equation}
Next,  applying  Lemma
\ref{lsssemma3.5} and \dref{cz2.5ghju48cjkklfg924ghyuji} to infer the existence of $C_5> 0$ satisfying
for all $t\in(0,T_{max}-\tau)$
\begin{equation}
\begin{array}{rl}
\disp\int_{t}^{t+\tau}\rho(s)ds
\leq&\disp{ C_{5}}\\
\end{array}
\label{ddfghgffgghsdfdfffffgggggddhjjjhjjnnhhkklld911cz2.5ghju48}
\end{equation}
and
\begin{equation}
\begin{array}{rl}
\disp\int_{t}^{t+\tau}h(s)ds
\leq&\disp{ C_{5}}.
\end{array}
\label{ddfghgffgghsdfdfffffgsssggggddhjjjhjjnnhhkklld911cz2.5ghju48}
\end{equation}
Now given $t\in(0,T_{max})$, in view of \dref{cz2.5ghju48cjkklfg924ghyuji} it is possible to fix $t_0 \in  (t-\tau, t)$ such that $t_0 \geq 0$ and
\begin{equation}
y(t_0) \leq C_6
\label{cz2.5ghju48cjkddffffklfgkkk924ghyuji}
\end{equation}
with some $C_6>0$, wereafter integrating \dref{ddfghgfhggddhjjjnkkll11cz2.5ghju48}  in time,  recalling  the fact that $t-t_0\leq \tau\leq1$
and using \dref{ddfghgffgghsdfdfffffgggggddhjjjhjjnnhhkklld911cz2.5ghju48} and \dref{ddfghgffgghsdfdfffffgsssggggddhjjjhjjnnhhkklld911cz2.5ghju48} yields
%where $h(t)=
%\frac{1}{\gamma_{1}}\frac{{\alpha}}{1+\alpha}\int_\Omega  n ^{2{+\alpha}}(\cdot,t)\geq0$ and $C_{12}(\varepsilon_1)$ is a positive constant.
%Therefore,  integrating \dref{ddfghgfhggddhjjjnkkll11cz2.5ghju48} and  using the fact that $t-t_0\leq \tau\leq1$ shows that
\begin{equation}
\begin{array}{rl}
y(t)\leq&\disp{y(t_0)e^{\disp\int_{t_0}^t\rho(s)ds}+\int_{t_0}^te^{\disp\int_{s}^t\rho(\tau)d\tau}h(s)ds}\\
\leq&\disp{C_{6}e^{C_{5}}+\int_{t_0}^te^{C_{5}}C_{5}ds}\\
\leq&\disp{(C_{5}+C_{6})e^{C_{5}}~~\mbox{for all}~~t\in(0,T_{max}),}\\
\end{array}
\label{czfvgb2.5ghhddffggjuygdddhjjjuffghhhddfghhllccvjkkklllhhjkkviihjj}
\end{equation}
%by using the fact that $t-t_0\leq \tau\leq1$.
%\dref{ddfghgffgghsdfdfffffgggggddhjjjhjjnnhhkklld911cz2.5ghju48} and \dref{ddfghgffgghsdfdfffffgsssggggddhjjjhjjnnhhkklld911cz2.5ghju48}.
%Thus, if we write $y(t) :=\|n (\cdot, t)\|^{{{1+\alpha}}}_{L^{{1+\alpha}}(\Omega)}+
%\|\nabla{c }(\cdot, t)\|^{{{2}}}_{L^{{2}}(\Omega)}$
%and $\rho(t) =4C_{7}^2\int_{\Omega}|\nabla u (\cdot, t)|^2$
%for $t\in(0,T_{max})$, so that,  \dref{3333cz2.5kksssssss121hhjjj4sddfffdfff114114}  implies that
%\begin{equation}
%\begin{array}{rl}
%y'(t)+h(t)
%\leq&\disp{ \rho(t)y(t)+C_{12}(\varepsilon_1)~\mbox{for all}~t\in(0,T_{max}).}\\
%\end{array}
%\label{ddfghgfhggddhjjjnkkll11cz2.5ghju48}
%\end{equation}
and thereby verifies \dref{cz2.5ghju48cjkddffffklfg924ghyuji},
%This proves  \dref{cz2.5ghju48cjkddffffklfg924ghyuji},
and hence \dref{223455cz2.5ghsssjddffffffgggu48cfg924ghyuji} results from one
further integration of \dref{ddfghgfhggddhjjjnkkll11cz2.5ghju48} if we once more make use of \dref{ddfghgffgghsdfdfffffgggggddhjjjhjjnnhhkklld911cz2.5ghju48} and \dref{ddfghgffgghsdfdfffffgsssggggddhjjjhjjnnhhkklld911cz2.5ghju48}.
This completes the proof of Lemma \ref{9999lemma4563ddfff025xxhjkloghyui}.
\end{proof}

Along with the basic estimate from Lemma \ref{lsssemma3.5}, this immediately implies the following.

\begin{lemma}\label{jkkkklsssemma3.5}
Let $\kappa\in \mathbb{R}$ as well as $|S(n)|\leq C_S(1+n)^{-\alpha}$ with $\alpha\geq-\frac{1}{2}$. Then for all ${p} > 1$, one can find $C > 0$ such that
\begin{equation}
\begin{array}{rl}\label{ssddsssdfff44ss555eqx45xx12112ccgghh}
%\disp\biggl\{
 \disp\int_\Omega  | \nabla c (\cdot,t) |^{2{p}}
\leq &  C~~~\mbox{for all}~~ t \in (0,T_{max}).\\
\end{array}
\end{equation}
\end{lemma}
\begin{proof} This results from a straightforward induction on the basis of Lemma \ref{9999lemma4563ddfff025xxhjkloghyui}, using \dref{ssddfff44ss555eqx45xx12112ccgghh} as a
starting point.
\end{proof}

Subtly combining the acquirements from Lemma 3.2, Lemma 3.3 and Lemma 3.6, we
can obtain the uniform boundedness of $\int_\Omega |u (\cdot, t)|^2$ as well as
$\int_0^t\int_\Omega |\nabla u (\cdot, t)|^2$ and $\int_0^t\int_\Omega n^2 $ in any finite
time interval, as formulated below.

%\begin{lemma}\label{9999lemma4563ddfff02ddffff5xxhjkloghyui}
%Let $q>1$ and $\kappa\in \mathbb{R}$. Then there exists a positive constant $\gamma$  such that the solution of \dref{1.dddduiikkldffdffg1}, \dref{ccvvx1.73142sdd6677gg}--\dref{ccvvx1.ddfff731426677gg} satisfies
%\begin{equation}\label{hddddjui9dddsss09klopji115}
%\begin{array}{rl}
%&\disp{\frac{1}{{2{m}}}\frac{d}{dt}\|\nabla c \|^{{{2{m}}}}_{L^{{2{m}}}(\Omega)}
%+\frac{{m}-1}{4{{m}^2}}\int_{\Omega}\left|\nabla |\nabla c |^{{m}}\right|^2
%+\frac{1}{2}\int_\Omega  |\nabla c |^{2{m}-2}|D^2c |^2+\frac{1}{2}\int_{\Omega} |\nabla c |^{2{m}}}\\
%\leq&\disp{ \gamma\int_\Omega|\nabla u |^2\int_{\Omega} |\nabla c |^{2q}+\gamma~~\mbox{for all}~~ t\in(0,T_{max}).}
%\end{array}
%\end{equation}
%\end{lemma}
\begin{lemma}\label{ggh788999hjllsssemmkkllla3.5}
Let $S$ satisfies that \dref{x1.73142vghf48gg} with $\alpha>-\frac{1}{2}$.  Then for any $\kappa\in\mathbb{R}$ and $p>1,$ %there exists $K(m_0) > 0$
% with the property
%that whenever  \dref{ccvvx1.731426677gg}  holds with
% $\disp\int_\Omega u_0 \leq m_0$,
one can find $C > 0$
such that

\begin{equation}
\begin{array}{rl}\label{ddfff44ss55ddd5kkkkkeqx45xx12112ccgghh}
\disp\int_\Omega n ^p(\cdot,t)
\leq &
C~~~~ \mbox{for all}~ t\in(0,T_{max}).\\
\end{array}
\end{equation}
%and
%\begin{equation}
%\begin{array}{rl}\label{ddfff44ss555eqx45xx1211kkklll2ccgghh}
%%\disp\biggl\{
% \disp\int_{(t-\tau)_+}^t \int_\Omega{}{(n  + 1)^2}
%\leq &  C~~~\mbox{for all}~~ t \in (0,T_{max}-\tau).\\
%\end{array}
%\end{equation}
%and
%\begin{equation}
%\begin{array}{rl}\label{ddfff44ss555eqx45xx12112ccgghh}
%%\disp\biggl\{
% \disp\int_{(t-\tau)_+}^t \int_\Omega{}{(n  + 1)^2}
%\leq &  C~~~\mbox{for all}~~ t \in (0,T_{max}-\tau).\\
%\end{array}
%\end{equation}
%where $m := \disp\int_\Omega u_0$.
\end{lemma}
\begin{proof}Integrating by parts in the first equation
from \dref{1.dddduiikkldffdffg1}  to see that thanks to the Young inequality,
\begin{equation}\label{ddsssdffsssf44ss512cchh}
\begin{array}{rl}
& \disp\frac{d}{dt}\disp\int_\Omega n ^p+ \disp(p-1)\int_\Omega n ^{p-2}{| \nabla n  |^2}\\
\leq&\disp (p-1)\int_\Omega { C_Sn ^{p-1}(1+n )^{-\alpha}}{}|\nabla n ||\nabla c |\\
%=&\disp \int_\Omega \nabla \int_0^{n }{ C_S(1+\tau)^{-\alpha}}{}d\tau \cdot \nabla c \\
%=&\disp -\int_\Omega \int_0^{n }{ C_S(1+\tau)^{-\alpha}}{}d\tau \Delta c \\
\leq&\disp \disp\frac{(p-1)}{2}\int_\Omega n ^{p-2}{| \nabla n  |^2}+\frac{C_S^2(p-1)}{2}\int_\Omega  n ^{p}(1+n )^{-2\alpha}|\nabla c |^2~~~\mbox{for all}~~~ t\in(0,T_{max}).\\
%\leq& \disp\delta\disp\int_\Omega{}{(n  + 1)^2}
%+\disp\frac{1}{4\delta}\disp\int_\Omega |\Delta c |^{\frac{2}{1+\alpha}}
%~~~\mbox{for all}~~~ t > 0
\end{array}
\end{equation}
%for any $\delta>0.$
 In view of the Gagliardo-Nirenberg inequality giving  constants $C_1> 0$ and $C_2> 0$
 such that
 \begin{equation}\label{ddsssdffssssssf44ss55dddlll5eqx45xx12112ccgghh}
\begin{array}{rl}
& \disp\int_\Omega{}n ^{p+1}\\
=&\disp \|n  ^{\frac{p}{2}}\|^{\frac{2(p+1)}{p}}_{L^{\frac{2(p+1)}{p}}(\Omega)}\\
\leq&\disp C_1(\|\nabla n  ^{\frac{p}{2}}\|^2_{L^2(\Omega)}\|n  ^{\frac{p}{2}}\|^\frac{2}{p}_{L^\frac{2}{p}(\Omega)}+\|\nabla n  ^{\frac{p}{2}}\|^{\frac{2(p+1)}{p}}_{L^\frac{2}{p}(\Omega)})\\
\leq&\disp C_2(\|\nabla n ^{\frac{p}{2}}\|^2_{L^2(\Omega)}+1)~~~\mbox{for all}~~~ t\in(0,T_{max}),
\end{array}
\end{equation}
which entails
$$\disp \disp\int_\Omega n ^{p-2}{| \nabla n  |^2}\geq\frac{1}{C_2}(\frac{4}{p^2}\int_\Omega{}{n  ^{p+1}}-1)~~~\mbox{for all}~~~ t\in(0,T_{max}).$$
Inserting the above inequality into \dref{ddsssdffsssf44ss512cchh} as well as using the Young inequality and Lemma \ref{jkkkklsssemma3.5}, there are $C_3>0$ and
$C_4>0$ such that
\begin{equation}\label{ddffggg67788344555ddsssdffsssf44ss55dddlll5eqx45xx12112ccgghh}
\begin{array}{rl}
& \disp\frac{d}{dt}\disp\int_\Omega n  ^p +  \disp\frac{(p-1)}{4}\int_\Omega n ^{p-2}{| \nabla n  |^2}+ \disp\frac{(p-1)}{4}\frac{1}{C_2}\frac{4}{p^2}\int_\Omega{}{n  ^{p+1}}\\
%\leq&\disp \int_\Omega { C_S(1+n )^{-\alpha}}{}\nabla n  \cdot \nabla c \\
%=&\disp \int_\Omega \nabla \int_0^{n }{ C_S(1+\tau)^{-\alpha}}{}d\tau \cdot \nabla c \\
%=&\disp -\int_\Omega \int_0^{n }{ C_S(1+\tau)^{-\alpha}}{}d\tau \Delta c \\
\leq&\disp\disp\frac{1}{2}\frac{(p-1)}{4}\frac{1}{C_2}\frac{4}{p^2}\int_\Omega{}{n  ^{p+1}}+C_3 \int_\Omega |\nabla c |^{\frac{p+1}{1-2(-\alpha)_+}}+C_3\\
\leq&\disp\disp\frac{1}{2}\frac{(p-1)}{4}\frac{1}{C_2}\frac{4}{p^2}\int_\Omega{}{n  ^{p+1}}+C_4~~~\mbox{for all}~~~ t\in(0,T_{max}),\\
%\leq& \disp\delta\disp\int_\Omega{}{(n  + 1)^2}
%+\disp\frac{1}{4\delta}\disp\int_\Omega |\Delta c |^{\frac{2}{1+\alpha}}
%~~~\mbox{for all}~~~ t > 0
\end{array}
\end{equation}
which together with some basic  calculations implies that \dref{ddfff44ss55ddd5kkkkkeqx45xx12112ccgghh} holds.
\end{proof}

\begin{lemma}\label{lsssemma3788999.5}
Let $|S(n)|\leq C_S(1+n)^{-\alpha}$ and $\alpha=\frac{1}{2}$.  Then for  any $\kappa\in\mathbb{R}$ and  finite $T\in (0,T_{max}]$, %there exist constants
%$l_0> \frac{N}{2}$ and $C(l_0,T)>0$ such that
%%there exists $K(m_0) > 0$
%% with the property
%that whenever  \dref{ccvvx1.731426677gg}  holds with
% $\disp\int_\Omega u_0 \leq m_0$,
one can find $C > 0$
such that
\begin{equation}
\begin{array}{rl}\label{ddfff44ss5sddff5ddd5eqx45xx12112ccgghh}
\disp\int_\Omega
[(n  + 1)\ln(n  + 1)+  {|\nabla u |^2}]
\leq &
C~~~~ \mbox{for all}~ t\in(0,T)\\
\end{array}
\end{equation}
as well as
\begin{equation}
\begin{array}{rl}\label{ddfff44ss5s234555sdd55eqx45xx12112ccgghh}
%\disp\biggl\{
 \disp\int_{0}^t\int_\Omega \bigl[(n +1)^2+ \frac{| \nabla n  |^2}{n  + 1}+|\Delta u |^2\bigr]
\leq &  C~~~\mbox{for all}~~ t \in (0,T).\\
\end{array}
\end{equation}
%and
%\begin{equation}
%\begin{array}{rl}\label{ssddfff44ss555eqx45xx12112ccgghh}
%%\disp\biggl\{
% \disp\int_{0}^t\int_\Omega \bigl[ | \nabla c  |^4\bigr]
%\leq &  C~~~\mbox{for all}~~ t \in (0,T).\\
%\end{array}
%\end{equation}
%where $m := \disp\int_\Omega u_0$.
\end{lemma}
\begin{proof} Integrating by parts in the first equation
from \dref{1.dddduiikkldffdffg1}, we derive from  the Young inequality as well as  Lemma \ref{fvfgsdfggfflemma45} and $\alpha=-\frac{1}{2}$  that %there is $C_1>0$ such that
\begin{equation}\label{wwwddsssdff099900dddlll5eqx45xx12112ccgghh}
\begin{array}{rl}
& \disp\frac{d}{dt}\disp\int_\Omega(n  + 1)
\ln(n  + 1) + \disp\int_\Omega\frac{| \nabla n  |^2}{n  + 1}\\
\leq&\disp \int_\Omega \frac{n }{n  + 1}S(n )\nabla n \cdot\nabla c \\
=&\disp \int_\Omega \nabla \int_0^{n }\frac{\tau}{\tau + 1}S(\tau)d\tau \cdot \nabla c \\
%=&\disp \int_\Omega \nabla \int_0^{n }{ C_S(1+\tau)^{-\alpha}}{}d\tau \cdot \nabla c \\
=&\disp -\int_\Omega \int_0^{n }\frac{\tau}{\tau + 1}S(\tau)d\tau \Delta c \\
\leq&\disp \int_\Omega C_S(1+n )^{\frac{3}{2}}|\Delta c |~~~\mbox{for all}~~~ t\in(0,T_{max}).\\
%\leq&\disp\delta\int_\Omega(1+n )^{2}+\frac{ C_S^4}{4}(\delta\frac{4}{3})^{-3}\int_\Omega|\Delta c |^4~~~\mbox{for all}~~~ t\in(0,T_{max}).\\
%\leq& \disp\delta\disp\int_\Omega{}{(n  + 1)^2}
%+\disp\frac{1}{4\delta}\disp\int_\Omega |\Delta c |^{\frac{2}{1+\alpha}}
%~~~\mbox{for all}~~~ t > 0
\end{array}
\end{equation}
We apply the Helmholtz projection operator $\mathcal{P}$ to the fourth  equation of \dref{1.dddduiikkldffdffg1} and multiply the
resulting identity, $u_t+Au=-\mathcal{P}[\kappa(u\cdot\nabla)u]+\mathcal{P}[(n+m)\nabla\phi]$, $t\in(0, T_{max})$, by $Au$.
According to the Young inequality, the boundedness of $\nabla\phi$ and the orthogonal projection property of $\mathcal{P}$, one can find that there exists some positive constant $C_1$ such that
% we immediately get
\begin{equation}
\begin{array}{rl}
\disp\frac{1}{2}\frac{d}{dt}\int_{\Omega}|A^{\frac{1}{2}}u|^2+\int_{\Omega}|Au|^{2}
=&\disp-\int_{\Omega}\mathcal{P}[\kappa(u\cdot\nabla)u]\cdot Au+\int_{\Omega}\mathcal{P}[(v_1+v_2)\nabla\phi]\cdot Au\\
\leq&\disp\frac{1}{4}\int_{\Omega}|Au|^2+\int_{\Omega}|\kappa(u\cdot\nabla)u|^2+\frac{1}{4}\int_{\Omega}|Au|^2+\int_{\Omega}|(n+m)\nabla\phi|^{2}\\
\leq&\disp\frac{1}{2}\int_{\Omega}|Au|^2+\int_{\Omega}|\kappa(u\cdot\nabla)u|^2+\|\nabla\phi\|_{L^\infty(\Omega)}^2\int_{\Omega}(n+m)^2\\
\leq&\disp\frac{1}{2}\int_{\Omega}|Au|^2+\int_{\Omega}|\kappa(u\cdot\nabla)u|^2+C_1\int_{\Omega}(n^2+1)~~~\mbox{for all}~~t\in(0,T_{max}),
\end{array}
\label{dkdmks}
\end{equation}
where in the last inequality we have used \dref{ddfgczhhhh2.5ghju48cfg924ghyuji} and the fact that
$$(a+b)^2\leq 2(a^2+b^2) ~~~\mbox{for any}~~a,b\geq0.$$
%for all $t\in(0,T_{max})$. % with some positive $C_1$, $C_2$.
The Young inequality in combination with the Gagliardo-Nirenberg inequality and the outcome of Lemma \ref{lsssemma3788999.5} enables us to
find $C_2>0$ as well as $C_3>0$ and $C_4>0$ such that
%By means of the Gagliardo-Nirenberg inequality, Young's inequality and the outcome of Lemma \ref{lemma3.2}, one can find $C_3>0$ such that
\begin{equation}
\begin{array}{rl}
\disp \int_{\Omega}|\kappa(u\cdot\nabla)u|^2
\leq &\disp C_2\int_{\Omega}|(u\cdot\nabla)u|^2\\
\leq &\disp C_2\|u\|^2_{L^{\infty}(\Omega)}\int_{\Omega}|\nabla u|^{2}\\
\leq &\disp C_3\|Au\|_{L^2(\Omega)}\|u\|_{L^2(\Omega)}\|\nabla u\|^2_{L^2(\Omega)}\\
\leq &\disp C_4\|Au\|_{L^2(\Omega)}\|\nabla u\|^2_{L^2(\Omega)}\\
\leq &\disp \frac{1}{4}\int_{\Omega}|Au|^2+C_4^2(\int_{\Omega}|\nabla u|^2)^2~~~\mbox{for all}~~t\in(0,T_{max}).
\end{array}
\label{psjdnch}
\end{equation}
Summing \dref{dkdmks} and \dref{psjdnch}, we show that
%Inserting \dref{psjdnch} into \dref{dkdmks} yields that
\begin{equation}
\frac{1}{2}\frac{d}{dt}\int_{\Omega}|A^{\frac{1}{2}}u|^2+\frac{1}{4}\int_{\Omega}|Au|^2
\leq C_4^2(\int_{\Omega}|\nabla u|^2)^2+C_1\int_{\Omega}(n^2+1)~~~\mbox{for all}~~t\in(0,T_{max}).
\label{cjjskx}
\end{equation}
Now, we let $h_1(t):=2C_4^2\int_{\Omega}|\nabla u|^2$, $h_2(t):=2C_1\int_{\Omega}(n^2+1)-\frac{1}{2}\int_{\Omega}|\Delta u|^2$ and $y(t):=\int_{\Omega}|\nabla u|^2$, $t\in(0,T_{max})$. From the definition of $y(t)$, we can further write \dref{cjjskx} as
\begin{equation}
y'(t)\leq h_1(t)\cdot y(t)+h_2(t)~~~\mbox{for all}~~t\in(0,T_{max}).
\label{ivfsgh}
\end{equation}
With the aid of Lemma \ref{lsssemma3788999.5}, for  any   finite $T\in (0,T_{max}]$,
we find %a positive constant
 $C_5>0$ such that
%\begin{equation}
%\int_{t}^{t+\tau}h_2(s)ds\leq C_5~~~\mbox{for all}~~t\in(0,T_{max}-\tau)
%\label{hsnxhn}
%\end{equation}
%and
\begin{equation}
\int_{0}^{t}\int_{\Omega}|\nabla u|^2\leq C_5~~~\mbox{for all}~~t\in(0,T).
\label{whxndj}
\end{equation}
%where $\tau$ is given by \dref{jvgxddrg}. %Particularly, from the latter it follows that if $t\in(0, T_{max})$ is arbitrary
%then in both cases $t\in(0,\tau)$ and $t\geq\tau$ we can find $t_0 = t_0(t)\in(t-\tau,t)$ such that $t_0\geq0$ and
%\begin{equation}
%\int_{\Omega}|\nabla u(\cdot,t_0)|^2\leq C_7
%\label{qisndudhn}
%\end{equation}
%with certain $C_7 > 0$ for all $t \in (0, T_{max})$.
Integrating \dref{ivfsgh} from $0$ to $t$ and using the definition of $h_1$ and $h_2$, we see that
%We now integrate \dref{ivfsgh} over $(t_0, t)$ to infer that
\begin{equation}\label{ddsssdffssssssddsssf44ss55dddlll5eqx45xx12112ccgghh}
\begin{array}{rl}
y(t)\leq& y(0) e^{\int_{0}^{t}h_1(s)ds}+e^{\disp\int_{0}^{t}h_1(\tau)d\tau}\int_{0}^{t}h_2(s)e^{-\int_{0}^{s}h_1(\tau)d\tau}ds\\
=&y(0) e^{\int_{0}^{t}h_1(s)ds}+\disp\int_{0}^{t}\bigl[2C_1\int_{\Omega}(n^2+1)-\frac{1}{2}\int_{\Omega}|\Delta u|^2\bigr]\cdot e^{\int_{s}^{t}2C_4^2\int_{\Omega}|\nabla u|^2d\tau}ds\\
\leq&y(0)e^{\int_{0}^{t}h_1(s)ds}+\disp\int_{0}^{t}\bigl[2C_1\int_{\Omega}(n^2+1)\bigr]\cdot e^{2C_4^2C_5}ds-
\frac{1}{2}\int_{0}^{t}\int_{\Omega}|\Delta u|^2.
\end{array}
\end{equation}
Here we have used the fact that
$$0\leq\int_{s}^{t}2C_4^2\int_{\Omega}|\nabla u|^2d\tau\leq C_5.$$
This combined with the definition of $h_1$  and \dref{ddsssdffssssssddsssf44ss55dddlll5eqx45xx12112ccgghh} yields to  %thanks to the choice of $t_0$,
$$\int_{\Omega}|\nabla u(\cdot,t)|^2\leq C_6+2C_1\int_{\Omega}(n^2+1)e^{2C_3^2C_5}-\frac{1}{2}\int_{0}^{t}\int_{\Omega}|\Delta u|^2$$
with some $C_6>0.$
In view of the Gagliardo-Nirenberg inequality giving  constants $C_7> 0$ and $C_8> 0$
 such that
 \begin{equation}\label{ddsssdffssssssf44ss55dddlll5eqx45xx12112ccgghh}
\begin{array}{rl}
& \disp\int_\Omega{}{(n  + 1)^2}\\
=&\disp \|(n  + 1)^{\frac{1}{2}}\|^4_{L^4(\Omega)}\\
\leq&\disp C_7(\|\nabla(n  + 1)^{\frac{1}{2}}\|^2_{L^2(\Omega)}\|(n  + 1)^{\frac{1}{2}}\|^2_{L^2(\Omega)}+\|\nabla(n  + 1)^{\frac{1}{2}}\|^4_{L^2(\Omega)})\\
\leq&\disp C_8(\|\nabla(n  + 1)^{\frac{1}{2}}\|^2_{L^2(\Omega)}+1)~~~\mbox{for all}~~~ t\in(0,T_{max}),
\end{array}
\end{equation}
which yields to
$$\disp\int_\Omega\frac{| \nabla n  |^2}{(n  + 1)}\geq\frac{4}{C_8}(\int_\Omega{}{(n  + 1)^2}-1)~~~\mbox{for all}~~~ t\in(0,T_{max}).$$
Inserting the above inequality into \dref{ddsssdffsssf44ss512cchh} and using the Young inequality, there is $C_9>0$ such that
\begin{equation}\label{344555ddsssdfer56677fsssf44ss55dddlll5eqx45xx12112ccgghh}
\begin{array}{rl}
& \disp\frac{d}{dt}\disp\int_\Omega(n  + 1)
\ln(n  + 1) + \disp\frac{1}{2}\int_\Omega\frac{| \nabla n  |^2}{(n  + 1)}+\frac{4}{2C_8}\int_\Omega{}{(n  + 1)^2}\\
%\leq&\disp \int_\Omega { C_S(1+n )^{-\alpha}}{}\nabla n  \cdot \nabla c \\
%=&\disp \int_\Omega \nabla \int_0^{n }{ C_S(1+\tau)^{-\alpha}}{}d\tau \cdot \nabla c \\
%=&\disp -\int_\Omega \int_0^{n }{ C_S(1+\tau)^{-\alpha}}{}d\tau \Delta c \\
\leq&\disp\disp\frac{1}{4}\int_\Omega\frac{| \nabla n  |^2}{(n  + 1)}+C_9 \int_\Omega |\Delta c |^{4}+C_9~~~\mbox{for all}~~~
t\in(0,T_{max}).\\
%\leq& \disp\delta\disp\int_\Omega{}{(n  + 1)^2}
%+\disp\frac{1}{4\delta}\disp\int_\Omega |\Delta c |^{\frac{2}{1+\alpha}}
%~~~\mbox{for all}~~~ t > 0
\end{array}
\end{equation}
For  any   finite $T\in (0,T_{max}],$ integrating \dref{344555ddsssdfer56677fsssf44ss55dddlll5eqx45xx12112ccgghh} with respect to time,
 we have
\begin{equation}\label{ssddd344555ddsgggjjssdffsssf44ss55dddlll5eqx45xx12112ccgghh}
\begin{array}{rl}
& \disp\disp\int_\Omega(n  + 1)
\ln(n  + 1) +\frac{2}{C_8}\int_0^t\int_\Omega{}{(n  + 1)^2}+\frac{1}{4}\int_0^t\int_\Omega\frac{| \nabla n  |^2}{(n  + 1)}\\
%\leq&\disp \int_\Omega { C_S(1+n )^{-\alpha}}{}\nabla n  \cdot \nabla c \\
%=&\disp \int_\Omega \nabla \int_0^{n }{ C_S(1+\tau)^{-\alpha}}{}d\tau \cdot \nabla c \\
%=&\disp -\int_\Omega \int_0^{n }{ C_S(1+\tau)^{-\alpha}}{}d\tau \Delta c \\
\leq&\disp\int_\Omega(n_0 + 1)
\ln(n_0 + 1)+C_9 \int_{0}^t\int_\Omega |\Delta c |^{4}+C_9T~~~\mbox{for all}~~~ t\in(0,T),\\
%\leq& \disp\delta\disp\int_\Omega{}{(n  + 1)^2}
%+\disp\frac{1}{4\delta}\disp\int_\Omega |\Delta c |^{\frac{2}{1+\alpha}}
%~~~\mbox{for all}~~~ t > 0
\end{array}
\end{equation}
By the regularity theory of parabolic equations, using Lemma \ref{fvfgsdfggfflemma45} and  Lemma \ref{jkkkklsssemma3.5},
% by the maximal Sobolev
%regularity for the %third and
% second
%equation of system \dref{1.dddduiikkldffdffg1}, \dref{ccvvx1.73142sdd6677gg}--\dref{ccvvx1.ddfff731426677gg} (see Lemma \ref{lemma45xy1222232}),
we derive that there exist positive constants $C_{10}$ as well as $C_{11}$ and $C_{12}$ such that
%\begin{equation}
%\begin{array}{rl}\label{ddfff44ssssssss55dddlll5eqx45xx12112ccgghh}
%\disp\int_{s_0}^t \disp\int_\Omega| \Delta w  |^p
%\leq &C_2\disp\int_{s_0}^t \bigl[\disp\int_\Omega n ^p+w ^p+|u \cdot\nabla w |^p\bigr]~~~\mbox{for all}~~ t \in (s_0,T_{max}).\\
%\end{array}
%\end{equation}
%and
\begin{equation}
\begin{array}{rl}
&\disp{\int_{0}^t\|\Delta c(\cdot,t)\|^{4}_{L^{4}(\Omega)}ds}\\
\leq &\disp{C_{10}\left(\int_{0}^t
\|w(\cdot,s)-u(\cdot,s)\cdot\nabla c(\cdot,s)\|^{4}_{L^{4}(\Omega)}ds\right).}\\
\leq &\disp{C_{11}\left(\int_{0}^t
\bigl[\|u(\cdot,s)\|^{{4}}_{L^{{8}}(\Omega)}\|\nabla c(\cdot,s)\|^{{4}}_{L^{{8}}(\Omega)}+1\bigr]ds\right)}\\
\leq &\disp{C_{12}\left(\int_{0}^t
\bigl[\|u(\cdot,s)\|^{{4}}_{L^{{8}}(\Omega)}+1\bigr]ds\right)~~\mbox{for all}~~ t\in(0,T).}\\
%\leq&\disp{\int_\Omega (|f(x,t)|+L)|u|^{q+1} dx}\\
%\leq&\disp{\int_\Omega (k_1|u|^\alpha+k_2)|u|^{q+1} dx}\\
%=&\disp{\int\int\triangle J(x-y)u(y)(u|u|^q(x))dydx.}
%\leq&\disp{(|g|_{L^\infty(0,\omega; L^\infty(\Omega))}+1)(|\Omega|+1)^{\frac{1}{2}}|u|^{q+1}_{q+2}.}\\
%\leq&\disp{\frac{q+1}{q+2}((|f|_{L^\infty(0,\omega; L^\infty(\Omega))}+1)(|\Omega|+1)^{\frac{1}{2}})^{\frac{q+2}{q+1}}|u|^{q+2}_{q+2}+\frac{1}{q+2}}\\
%\leq&\disp{\frac{q+1}{q+2}(|f|_{L^\infty(0,\omega; L^\infty(\Omega))}+1)^2(|\Omega|+1)|u|^{q+2}_{q+2}+\frac{1}{q+2}.}\\
\end{array}
\label{cz2.5bbssse4567788v114}
\end{equation}
%where $\tau$ is the same as Lemma \ref{lemma3.1}.
Next, recalling Lemma  \ref{lsssemma3788999.5}, then
Gagliardo-Nirenberg interpolation inequality  yields to
\begin{equation}
\begin{array}{rl}\label{ddfff44sssss55dddlll5eqx45xx12112ccgghh}
\disp&\disp \disp\|u(\cdot,s)\|^{{4}}_{L^{{8}}(\Omega)}\\
\leq &C_{13}\disp \bigl[\|\Delta u(\cdot,t)\|^{\frac{3}{2}}_{L^{2}(\Omega)}\|u\|^{\frac{5}{2}}_{L^{2}(\Omega)}+\|u\|^{4}_{L^{2}(\Omega)}\bigr]\\
\leq &C_{14}\disp \bigl[\|\Delta u(\cdot,t)\|^{\frac{3}{2}}_{L^{2}(\Omega)}+1\bigr]~~~\mbox{for all}~~ t \in (0,T_{max})\\
\end{array}
\end{equation}
with some $C_{13}>0$ and $C_{14}>0$. Inserting \dref{ddfff44sssss55dddlll5eqx45xx12112ccgghh} into \dref{cz2.5bbssse4567788v114}, then for any $\delta>0$ and  finite $T\in (0,T_{max}]$, we derive from  the Young inequality that there are $C_{15}>0$ and $C_{16}>0$ such that
\begin{equation}
\begin{array}{rl}
&\disp{\int_{0}^t\|\Delta c(\cdot,t)\|^{4}_{L^{4}(\Omega)}ds}\\
\leq &\disp{C_{15}\left(\int_{0}^t
\bigl[\|\Delta u(\cdot,t)\|^{\frac{3}{2}}_{L^{2}(\Omega)}+1\bigr]ds\right)}\\
\leq &\disp{\delta\int_{0}^t
\|\Delta u(\cdot,t)\|^{2}_{L^{2}(\Omega)}ds+C_{16}~~\mbox{for all}~~ t\in(0,T),}\\
%\leq&\disp{\int_\Omega (|f(x,t)|+L)|u|^{q+1} dx}\\
%\leq&\disp{\int_\Omega (k_1|u|^\alpha+k_2)|u|^{q+1} dx}\\
%=&\disp{\int\int\triangle J(x-y)u(y)(u|u|^q(x))dydx.}
%\leq&\disp{(|g|_{L^\infty(0,\omega; L^\infty(\Omega))}+1)(|\Omega|+1)^{\frac{1}{2}}|u|^{q+1}_{q+2}.}\\
%\leq&\disp{\frac{q+1}{q+2}((|f|_{L^\infty(0,\omega; L^\infty(\Omega))}+1)(|\Omega|+1)^{\frac{1}{2}})^{\frac{q+2}{q+1}}|u|^{q+2}_{q+2}+\frac{1}{q+2}}\\
%\leq&\disp{\frac{q+1}{q+2}(|f|_{L^\infty(0,\omega; L^\infty(\Omega))}+1)^2(|\Omega|+1)|u|^{q+2}_{q+2}+\frac{1}{q+2}.}\\
\end{array}
\label{cz2.5bbssssdddsv114}
\end{equation}
which combined with \dref{ssddd344555ddsgggjjssdffsssf44ss55dddlll5eqx45xx12112ccgghh} implies  that
\begin{equation}\label{ssddd344sssd555ddsgggjjssdffsssf44ss55dddlll5eqx45xx12112ccgghh}
\begin{array}{rl}
& \disp\disp\int_\Omega(n  + 1)
\ln(n  + 1) +\frac{2}{C_7}\int_0^t\int_\Omega{}{(n  + 1)^2}+\frac{1}{4}\int_0^t\int_\Omega\frac{| \nabla n  |^2}{(n  + 1)}\\
%\leq&\disp \int_\Omega { C_S(1+n )^{-\alpha}}{}\nabla n  \cdot \nabla c \\
%=&\disp \int_\Omega \nabla \int_0^{n }{ C_S(1+\tau)^{-\alpha}}{}d\tau \cdot \nabla c \\
%=&\disp -\int_\Omega \int_0^{n }{ C_S(1+\tau)^{-\alpha}}{}d\tau \Delta c \\
\leq&\disp\int_\Omega(n_0 + 1)
\ln(n_0 + 1)+C_8[\delta\int_{0}^t
\|\Delta u(\cdot,t)\|^{2}_{L^{2}(\Omega)}ds+C_9]+C_8T~~~\mbox{for all}~~~ t\in(0,T)\\
%\leq& \disp\delta\disp\int_\Omega{}{(n  + 1)^2}
%+\disp\frac{1}{4\delta}\disp\int_\Omega |\Delta c |^{\frac{2}{1+\alpha}}
%~~~\mbox{for all}~~~ t > 0
\end{array}
\end{equation}
with
$$\delta=\frac{1}{4C_8\iota}$$
and
$$\iota=2C_1C_7e^{2C_4^2C_5}.$$
$\dref{ddsssdffssssssddsssf44ss55dddlll5eqx45xx12112ccgghh}+\iota\times\dref{ssddd344sssd555ddsgggjjssdffsssf44ss55dddlll5eqx45xx12112ccgghh}$ yields to
\begin{equation}\label{ssddd3ssdd44sssd555ddsgggjjssdffsssf44ss55dddlll5eqx45xx12112ccgghh}
\begin{array}{rl}
& \disp\disp\iota \int_\Omega(n  + 1)
\ln(n  + 1) +y(t)+\frac{\iota}{C_7}\int_0^t\int_\Omega{}{(n  + 1)^2}+\frac{1}{4}\int_0^t\int_\Omega\frac{| \nabla n  |^2}{(n  + 1)}\\
%\leq&\disp \int_\Omega { C_S(1+n )^{-\alpha}}{}\nabla n  \cdot \nabla c \\
%=&\disp \int_\Omega \nabla \int_0^{n }{ C_S(1+\tau)^{-\alpha}}{}d\tau \cdot \nabla c \\
%=&\disp -\int_\Omega \int_0^{n }{ C_S(1+\tau)^{-\alpha}}{}d\tau \Delta c \\
\leq&\disp y(0) e^{\int_{0}^{t}h_1(s)ds}-
\frac{1}{4}\int_{0}^{t}\int_{\Omega}|\Delta u|^2\\
&\disp+\iota\int_\Omega(n_0 + 1)
\ln(n_0 + 1)+C_8\iota\delta C_9+C_8T\iota~~~\mbox{for all}~~~ t\in(0,T),\\
%\leq& \disp\delta\disp\int_\Omega{}{(n  + 1)^2}
%+\disp\frac{1}{4\delta}\disp\int_\Omega |\Delta c |^{\frac{2}{1+\alpha}}
%~~~\mbox{for all}~~~ t > 0
\end{array}
\end{equation}
which implies \dref{ddfff44ss5sddff5ddd5eqx45xx12112ccgghh} and \dref{ddfff44ss5s234555sdd55eqx45xx12112ccgghh} immediately by using the definition of $y(t).$
\end{proof}

Based on a well-known continuous embedding with
the above Lemma  provides the uniform $L^p$ conditional estimates of $u$ as follows.

\begin{lemma}\label{lemmasdde45673.4dd}
Let $\kappa\in \mathbb{R}$ as well as $|S(n)|\leq C_S(1+n)^{-\alpha}$ with $\alpha=-\frac{1}{2}$.
%Let $\alpha=-\frac{1}{2}$.
Then
for any $p>1$ and finite $T\in (0,T_{max}]$, there exists $C>0$ such that
\begin{equation}
\int_{\Omega}| u (\cdot,t)|^p\leq C~~\mbox{for all}~~t\in(0,T).
\label{qidjfnhf}
\end{equation}
\end{lemma}
\begin{proof}
For any $p>1$, recalling Lemma \ref{lsssemma3788999.5},
by applying Sobolev embedding $W^{1,2}(\Omega)\hookrightarrow L^p(\Omega)$, we derive that  \dref{qidjfnhf} holds.
\end{proof}

\begin{lemma}\label{lsssemssddma3.5}
Let $|S(n)|\leq C_S(1+n)^{-\alpha}$ and $\alpha=\frac{1}{2}$.  Then for  any finite $T\in (0,T_{max}]$ and $p>2$, %there exist constants
%$l_0> \frac{N}{2}$ and $C(l_0,T)>0$ such that
%%there exists $K(m_0) > 0$
%% with the property
%that whenever  \dref{ccvvx1.731426677gg}  holds with
% $\disp\int_\Omega u_0 \leq m_0$,
one can find $C > 0$
such that
\begin{equation}
\begin{array}{rl}\label{ddfff44ss55dddserr5dddeqx45xx12112ccgghh}
\disp\int_\Omega
n  ^p(\cdot,t)
\leq &
C~~~~ \mbox{for all}~ t\in(0,T).\\
\end{array}
\end{equation}
%as well as
%\begin{equation}
%\begin{array}{rl}\label{ddfff44ss555eqx4ddfff5xx12112ccgghh}
%%\disp\biggl\{
% \disp\int_{0}^t\int_\Omega \bigl[(n +1)^2+ \frac{| \nabla n  |^2}{n  + 1}+|\Delta u |^2\bigr]
%\leq &  C~~~\mbox{for all}~~ t \in (0,T).\\
%\end{array}
%\end{equation}
%and
%\begin{equation}
%\begin{array}{rl}\label{ssddfff44ss555eqx45xx12112ccgghh}
%%\disp\biggl\{
% \disp\int_{0}^t\int_\Omega \bigl[ | \nabla c  |^4\bigr]
%\leq &  C~~~\mbox{for all}~~ t \in (0,T).\\
%\end{array}
%\end{equation}
%where $m := \disp\int_\Omega u_0$.
\end{lemma}
\begin{proof}
Using $n ^{p-1}$ as a test function for the first equation in \dref{1.dddduiikkldffdffg1}, integrating by parts and
using that
$\nabla\cdot u  = 0$,, we obtain
%Testing the first equation in \dref{1.dddduiikkldffdffg1} by $n ^{p-1}$, integrating by parts over $\Omega$ and using that
%$\nabla\cdot u  = 0$, we gain
%Testing
%the first equation
%of  \dref{1.dddduiikkldffdffg1}, \dref{ccvvx1.73142sdd6677gg}--\dref{ccvvx1.ddfff731426677gg}  by $n ^{p-1}$ and using the Young inequality,
\begin{equation}\label{ddssf44s5eqx45xx12112ccgghh}
\begin{array}{rl}
& \disp\frac{1}{p}\frac{d}{dt}\disp\int_\Omega n ^p + (p-1)\disp\int_\Omega n ^{p-2}{| \nabla n  |^2}\\
\leq&\disp (p-1)\int_\Omega  n ^{p-1}{ C_S(1+n )^{-\alpha}}{}\nabla n  \cdot \nabla c \\
=&\disp \int_\Omega \nabla \int_0^{n }{ \tau^{p-1}C_S(1+\tau)^{-\alpha}}{}d\tau \cdot \nabla c \\
=&\disp -\int_\Omega \int_0^{n }{\tau^{p-1} C_S(1+\tau)^{-\alpha}}{}d\tau \Delta c \\
\leq&\disp C_S\int_\Omega n ^{p}(1+n )^{(-\alpha)_+}|\Delta c |~~~\mbox{for all}~~~ t\in(0,T_{max}).\\
%\leq& \disp\delta\disp\int_\Omega{}{(n  + 1)^2}
%+\disp\frac{1}{4\delta}\disp\int_\Omega |\Delta c |^{\frac{2}{1+\alpha}}
%~~~\mbox{for all}~~~ t > 0
\end{array}
\end{equation}
Since $p > 2$ implies that $\frac{2}{p}<\frac{2(p+1)}{p}<+\infty$,
%For either 2 ¡Ü p < 8
%3 or 25
%12
%¡Ü p < 8
%3 , we deduce from an elementary computation
%%for any $\delta>0.$
and thereby an application of the Gagliardo-Nirenberg inequality along with Lemma \ref{fvfgsdfggfflemma45} provides  $C_1> 0$ and $C_2> 0$
such that
%We using the Gagliardo-Nirenberg inequality and Lemma \ref{fvfgsdfggfflemma45} to obtain some
%constants  $C_1> 0$ and $C_2> 0$ such that
% In view of the Gagliardo-Nirenberg inequality giving  constants $C_1> 0$ and $C_2> 0$
% such that
 \begin{equation}\label{sssddsssdffssssssf44ss55dddlll5eqx45xx12112ccgghh}
\begin{array}{rl}
& \disp\int_\Omega{}{n ^{p+1}}\\
=&\disp \|n ^{\frac{p}{2}}\|^{\frac{2(p+1)}{p}}_{L^{\frac{2(p+1)}{p}}(\Omega)}\\
\leq&\disp C_1(\|\nabla n ^{\frac{p}{2}}\|^2_{L^2(\Omega)}\|n ^{\frac{p}{2}}\|^{\frac{2}{p}}_{L^\frac{2}{p}(\Omega)}+\|n ^{\frac{p}{2}}\|^{\frac{2(p+1)}{p}}_{L^\frac{2}{p}(\Omega)})\\
\leq&\disp C_2(\|\nabla n ^{\frac{p}{2}}|^2_{L^2(\Omega)}+1)~~~\mbox{for all}~~~ t\in(0,T_{max}),
\end{array}
\end{equation}
%which implies that
and hence
$$\disp\|\nabla n ^{\frac{p}{2}}|^2_{L^2(\Omega)}\geq\frac{1}{C_2}(\int_\Omega{}{n ^{p+1}}-1)~~~\mbox{for all}~~~ t\in(0,T_{max}),$$
where we use that employing \dref{ddssf44s5eqx45xx12112ccgghh} and the Young inequality we can find $C_3>0$  satisfying
 %Combining  the above inequality and  \dref{ddssf44s5eqx45xx12112ccgghh} we obtain from the Young inequality that
%%Inserting the above inequality into \dref{ddssf44s5eqx45xx12112ccgghh} and using the Young inequality,
% there is $C_3>0$ such that
\begin{equation}\label{344555ddsssdffsss89009kif44ss55dddlll5eqx45xx12112ccgghh}
\begin{array}{rl}
& \disp\frac{1}{p}\frac{d}{dt}\disp\int_\Omega n ^p
+ \disp\frac{p-1}{2}\disp\int_\Omega n ^{p-2}{| \nabla n  |^2}+\frac{1}{2C_2}\frac{p-1}{2}\frac{4}{p^2}\int_\Omega{}{n ^{p+1}}\\
%\leq&\disp \int_\Omega { C_S(1+n )^{-\alpha}}{}\nabla n  \cdot \nabla c \\
%=&\disp \int_\Omega \nabla \int_0^{n }{ C_S(1+\tau)^{-\alpha}}{}d\tau \cdot \nabla c \\
%=&\disp -\int_\Omega \int_0^{n }{ C_S(1+\tau)^{-\alpha}}{}d\tau \Delta c \\
\leq&\disp\disp\frac{1}{4C_2}\frac{p-1}{2}\frac{4}{p^2}\int_\Omega{}{n ^{p+1}}+C_3 \int_\Omega |\Delta c |^{\frac{p+1}{(1-(-\alpha)_+}}+\frac{1}{2C_2}\frac{p-1}{2}\frac{4}{p^2}~~~\mbox{for all}~~~ t\in(0,T_{max}).\\
%\leq& \disp\delta\disp\int_\Omega{}{(n  + 1)^2}
%+\disp\frac{1}{4\delta}\disp\int_\Omega |\Delta c |^{\frac{2}{1+\alpha}}
%~~~\mbox{for all}~~~ t > 0
\end{array}
\end{equation}
In the following we will show that $\int_{(t-\hat{\tau})_+}^t\int_\Omega |\Delta c |^{\frac{p+1}{(1-(-\alpha)_+}}$ can be controlled by using Lemma \ref{lemmasdde45673.4dd} and Lemma \ref{fvfgsdfggfflemma45}, where
\begin{equation}\label{3009kif44ss55dddlll5eqx45xx12112ccgghh}\hat{\tau}=\min\{1,\frac{1}{2}T\}.
\end{equation}
 To this end, for any $q>1,$ and applying the the regularity theory of parabolic equations on parabolic equations  with the  Neumann boundary condition, we find that there exists some positive constant
$C_4$ as well as $C_5$ and $C_6$ such that
%For any $q>1,$  by the regularity theory of parabolic equations and using Lemma \ref{fvfgsdfggfflemma45} and Lemma \ref{lemmasdde45673.4dd},
%% by the maximal Sobolev
%%regularity for the %third and
%% second
%%equation of system \dref{1.dddduiikkldffdffg1}, \dref{ccvvx1.73142sdd6677gg}--\dref{ccvvx1.ddfff731426677gg} (see Lemma \ref{lemma45xy1222232}),
%we derive that there exist positive constants $C_4$ as well as $C_5$ and $C_6$ such that
%\begin{equation}
%\begin{array}{rl}\label{ddfff44ssssssss55dddlll5eqx45xx12112ccgghh}
%\disp\int_{s_0}^t \disp\int_\Omega| \Delta w  |^p
%\leq &C_2\disp\int_{s_0}^t \bigl[\disp\int_\Omega n ^p+w ^p+|u \cdot\nabla w |^p\bigr]~~~\mbox{for all}~~ t \in (s_0,T_{max}).\\
%\end{array}
%\end{equation}
%and
\begin{equation}
\begin{array}{rl}
&\disp{\int_{(t-\hat{\tau})_+}^t\|\Delta c(\cdot,t)\|^{q}_{L^{q}(\Omega)}ds}\\
\leq &\disp{C_4\left(\int_{(t-\hat{\tau})_+}^t
\|m(\cdot,s)-u(\cdot,s)\cdot\nabla c(\cdot,s)\|^{q}_{L^{q}(\Omega)}ds\right).}\\
\leq &\disp{C_5\left(\int_{(t-\hat{\tau})_+}^t
\bigl[\|u(\cdot,s)\|^{{3q}}_{L^{{3q}}(\Omega)}+\|\nabla c(\cdot,s)\|^{\frac{3q}{2}}_{L^{\frac{3q}{2}}(\Omega)}+1\bigr]ds\right)}\\
\leq &\disp{C_6\left(\int_{(t-\hat{\tau})_+}^t
[\|\nabla c(\cdot,s)\|^{\frac{3q}{2}}_{L^{\frac{3q}{2}}(\Omega)}+1]ds\right)~~\mbox{for all}~~ t\in(0,T-\hat{\hat{\tau}})}\\
%\leq&\disp{\int_\Omega (|f(x,t)|+L)|u|^{q+1} dx}\\
%\leq&\disp{\int_\Omega (k_1|u|^\alpha+k_2)|u|^{q+1} dx}\\
%=&\disp{\int\int\triangle J(x-y)u(y)(u|u|^q(x))dydx.}
%\leq&\disp{(|g|_{L^\infty(0,\omega; L^\infty(\Omega))}+1)(|\Omega|+1)^{\frac{1}{2}}|u|^{q+1}_{q+2}.}\\
%\leq&\disp{\frac{q+1}{q+2}((|f|_{L^\infty(0,\omega; L^\infty(\Omega))}+1)(|\Omega|+1)^{\frac{1}{2}})^{\frac{q+2}{q+1}}|u|^{q+2}_{q+2}+\frac{1}{q+2}}\\
%\leq&\disp{\frac{q+1}{q+2}(|f|_{L^\infty(0,\omega; L^\infty(\Omega))}+1)^2(|\Omega|+1)|u|^{q+2}_{q+2}+\frac{1}{q+2}.}\\
\end{array}
\label{cz2.5bbsssv8876899114}
\end{equation}
by using Lemma \ref{fvfgsdfggfflemma45} and Lemma \ref{lemmasdde45673.4dd},
with any finite $T\in (0,T_{max}]$.
%where $\tau$ is the same as Lemma \ref{lemma3.1}.
Next, again,
according to Gagliardo-Nirenberg interpolation inequality and Lemma \ref{fvfgsdfggfflemma45}, we can pick $C_7>0$ and $C_8>0$
such that
%Next, recalling Lemma \ref{fvfgsdfggfflemma45}, then
%Gagliardo-Nirenberg interpolation inequality  yields to
\begin{equation}
\begin{array}{rl}\label{ddfff44sssssqx45xx1789992112ccgghh}
\disp&\disp \disp\int_\Omega| \nabla c  |^{\frac{3q}{2}}\\
\leq &C_7\disp [\|\Delta c(\cdot,t)\|^{\frac{\frac{3q}{4}-1}{1-\frac{1}{q}}}_{L^{q}(\Omega)}\|c\|^{\frac{3q}{2}-\frac{\frac{3q}{4}-1}{1-\frac{1}{q}}}_{L^{\infty}(\Omega)}+\|c\|^{\frac{3q}{2}}_{L^{\infty}(\Omega)}]\\
\leq &C_8\disp \bigl[\|\Delta c(\cdot,t)\|^{\frac{\frac{3q}{4}-1}{1-\frac{1}{q}}}_{L^{q}(\Omega)}+1\bigr]~~~\mbox{for all}~~ t \in (0,T).\\
\end{array}
\end{equation}
Also thanks to $q > 1$, we can verify that
$${\frac{\frac{3q}{4}-1}{1-\frac{1}{q}}}<q,$$
%with some $C_7>0$ and $C_8>0$.
and whereupon an application of the Young inequality entails
\begin{equation*}
\begin{array}{rl}
\disp \disp \disp\int_\Omega| \nabla c  |^{\frac{3q}{2}}
\leq &\disp\frac{C_8}{2C_6}\disp \bigl[\|\Delta c(\cdot,t)\|^{q}_{L^{q}(\Omega)}+1\bigr]~~~\mbox{for all}~~ t \in (0,T).\\
\end{array}
\end{equation*}
Substitute the above inequality  into \dref{cz2.5bbsssv8876899114} to have
%which combined with \dref{cz2.5bbsssv8876899114} implies that
\begin{equation}
\begin{array}{rl}
\disp{\int_{(t-\hat{\tau})_+}^t\|\Delta c(\cdot,t)\|^{q}_{L^{q}(\Omega)}ds}
\leq &\disp{C_9~~\mbox{for all}~~ t\in(0,T-\hat{\tau})}\\
%\leq&\disp{\int_\Omega (|f(x,t)|+L)|u|^{q+1} dx}\\
%\leq&\disp{\int_\Omega (k_1|u|^\alpha+k_2)|u|^{q+1} dx}\\
%=&\disp{\int\int\triangle J(x-y)u(y)(u|u|^q(x))dydx.}
%\leq&\disp{(|g|_{L^\infty(0,\omega; L^\infty(\Omega))}+1)(|\Omega|+1)^{\frac{1}{2}}|u|^{q+1}_{q+2}.}\\
%\leq&\disp{\frac{q+1}{q+2}((|f|_{L^\infty(0,\omega; L^\infty(\Omega))}+1)(|\Omega|+1)^{\frac{1}{2}})^{\frac{q+2}{q+1}}|u|^{q+2}_{q+2}+\frac{1}{q+2}}\\
%\leq&\disp{\frac{q+1}{q+2}(|f|_{L^\infty(0,\omega; L^\infty(\Omega))}+1)^2(|\Omega|+1)|u|^{q+2}_{q+2}+\frac{1}{q+2}.}\\
\end{array}
\label{cz2.5bbsssv1ssssddffgs14}
\end{equation}
with some $C_9>0$. We abbreviate $q=\frac{p+1}{1-(-\alpha)_+}$ and invoke \dref{cz2.5bbsssv1ssssddffgs14} to
fix $C_{10}>0$ fulfilling
%Letting $q=\frac{p+1}{1-(-\alpha)_+}$ and applying \dref{cz2.5bbsssv1ssssddffgs14}, one has
%Choosing $q=\frac{p+1}{1-(-\alpha)_+}$ and using \dref{cz2.5bbsssv1ssssddffgs14}, we have
\begin{equation}
\begin{array}{rl}
\disp{\int_{(t-\hat{\tau})_+}^t\|\Delta c(\cdot,t)\|^{\frac{p+1}{(1-(-\alpha)_+}}_{L^{\frac{p+1}{(1-(-\alpha)_+}}(\Omega)}ds}
\leq &\disp{C_{10}~~\mbox{for all}~~ t\in(0,T-\hat{\tau}),}\\
%\leq&\disp{\int_\Omega (|f(x,t)|+L)|u|^{q+1} dx}\\
%\leq&\disp{\int_\Omega (k_1|u|^\alpha+k_2)|u|^{q+1} dx}\\
%=&\disp{\int\int\triangle J(x-y)u(y)(u|u|^q(x))dydx.}
%\leq&\disp{(|g|_{L^\infty(0,\omega; L^\infty(\Omega))}+1)(|\Omega|+1)^{\frac{1}{2}}|u|^{q+1}_{q+2}.}\\
%\leq&\disp{\frac{q+1}{q+2}((|f|_{L^\infty(0,\omega; L^\infty(\Omega))}+1)(|\Omega|+1)^{\frac{1}{2}})^{\frac{q+2}{q+1}}|u|^{q+2}_{q+2}+\frac{1}{q+2}}\\
%\leq&\disp{\frac{q+1}{q+2}(|f|_{L^\infty(0,\omega; L^\infty(\Omega))}+1)^2(|\Omega|+1)|u|^{q+2}_{q+2}+\frac{1}{q+2}.}\\
\end{array}
\label{ssdddfc900900z2.5bbkkkksssv1sssss14}
\end{equation}
where $\hat{\tau}$ is given by \dref{3009kif44ss55dddlll5eqx45xx12112ccgghh}.
%with some $C_{10}>0.$
Recalling   \dref{344555ddsssdffsss89009kif44ss55dddlll5eqx45xx12112ccgghh} and using \dref{ssdddfc900900z2.5bbkkkksssv1sssss14}, an application of the comparison argument of an ordinary differential
equation yields \dref{ddfff44ss55dddserr5dddeqx45xx12112ccgghh}.
%Recalling   \dref{344555ddsssdffsss89009kif44ss55dddlll5eqx45xx12112ccgghh} and using \dref{ssdddfc900900z2.5bbkkkksssv1sssss14}, we derive that \dref{ddfff44ss55dddserr5dddeqx45xx12112ccgghh}  holds with the help of some basic  computation.
\end{proof}

\begin{lemma}\label{lemma63kkllll0jklhhjj}
Let  $\alpha>-1,\kappa = 0$ or  $\alpha>-\frac{1}{2},\kappa \in\mathbb{R}$. There exists a positive constant $C$  such that
\begin{equation}
\int_{\Omega}{|\nabla u (\cdot,t)|^2}\leq C~~\mbox{for all}~~ t\in(0, T_{max}).
\label{3333ddxcvbbggddfgcz2vv.5ghju48cfg924ghyuji}
\end{equation}
%While if  $\alpha\geq-\frac{1}{2}$, $\kappa\in\mathbb{R}$,
%Then for  any finite $T\in (0,T_{max}]$, %there exist constants
%%$l_0> \frac{N}{2}$ and $C(l_0,T)>0$ such that
%%%there exists $K(m_0) > 0$
%%% with the property
%%that whenever  \dref{ccvvx1.731426677gg}  holds with
%% $\disp\int_\Omega u_0 \leq m_0$,
%one can find $C > 0$
%such that
%
%\begin{equation}
%\int_{\Omega}{|\nabla u (\cdot,t)|^2}\leq C~~\mbox{for all}~~ t\in(0, T).
%\label{1111ddxcvbbggddfgcz2vv.5ghju48cfg92llll4ghyuji}
%\end{equation}
\end{lemma}
\begin{proof}
%Case $\alpha>-\frac{1}{2},\kappa \in\mathbb{R}$:
Firstly, applying the Helmholtz projection to both sides of the fourth  equation in \dref{1.dddduiikkldffdffg1},
then multiplying the result identified by $Au $, integrating by parts, and using the Young inequality, we find that
\begin{equation}
\begin{array}{rl}
&\disp{\frac{1}{{2}}\frac{d}{dt}\|A^{\frac{1}{2}}u \|^{{{2}}}_{L^{{2}}(\Omega)}+
\int_{\Omega}|Au |^2 }
\\
=&\disp{ \int_{\Omega}Au \mathcal{P}(-\kappa
(u  \cdot \nabla)u )+ \int_{\Omega}\mathcal{P}(n +m )\nabla\phi) Au }
\\
\leq&\disp{ \frac{1}{2}\int_{\Omega}|Au |^2+\kappa^2\int_{\Omega}
|(u  \cdot \nabla)u |^2+ 2\|\nabla\phi\|^2_{L^\infty(\Omega)}\int_{\Omega}(n ^2+m ^2)~~\mbox{for all}~~t\in(0,T_{max}).}
\end{array}
\label{ddfghgghjjnnhhkklld911cz2.5ghju48}
\end{equation}
%Noticing that $\|Y u \|_{L^2(\Omega)}\leq\|u \|_{L^2(\Omega)},$
It follows from the Gagliardo-Nirenberg inequality and the Cauchy-Schwarz inequality that with some $C_1 >0$ and $C_2 > 0$
\begin{equation}
\begin{array}{rl}
&\kappa^2\disp\int_{\Omega}
|( u  \cdot \nabla)u |^2\\
\leq&\disp{ \kappa^2\| u \|^2_{L^4(\Omega)}\|\nabla u \|^2_{L^4(\Omega)}}
\\
\leq&\disp{ \kappa^2C_1[\|\nabla  u \|_{L^2(\Omega)}\| u\|_{L^2(\Omega)}]
[\|A u \|_{L^2(\Omega)}\|\nabla u \|_{L^2(\Omega)}]
}\\
\leq&\disp{ \kappa^2C_1C_{2}\|\nabla u \|_{L^2(\Omega)}
[\|A u \|_{L^2(\Omega)}\|\nabla u \|_{L^2(\Omega)}]
~~\mbox{for all}~~t\in(0,T_{max})}.
\end{array}
\label{ssdcfvgddfghgghjd911cz2.5ghju48}
\end{equation}
Now, from the fact that $D( A^{\frac{1}{2}})  :=W^{1,2}_0(\Omega;\mathbb{R}^2) \cap L_{\sigma}^{2}(\Omega)$ and
\dref{czfvgb2.5ghhjuyuccvviihjj}, it follows that
\begin{equation}
\|\nabla  u \|_{L^2(\Omega)}=\|A^{\frac{1}{2}}  u \|_{L^2(\Omega)}.
\label{ssdcfdhhgghjjnnhhkklld911cz2.5ghju48}
\end{equation}
Due to  Theorem 2.1.1 in \cite{Sohr},  $\|A(\cdot)\|_{L^{2}(\Omega)}$ defines a norm
equivalent to $\|\cdot\|_{W^{2,2}(\Omega)}$ on $D(A)$.
This, together with the Young inequality and estimates \dref{ssdcfdhhgghjjnnhhkklld911cz2.5ghju48} and \dref{ssdcfvgddfghgghjd911cz2.5ghju48}, yields
$$
\begin{array}{rl}
&\kappa^2\disp\int_{\Omega}|( u  \cdot \nabla)u |^2
\\
\leq&\disp{ C_{3}\|A u \|_{L^2(\Omega)}\|\nabla u \|_{L^2(\Omega)}^2}
\\
\leq&\disp{ \frac{1}{4}\|A u \|_{L^2(\Omega)}+\kappa^4C_{1}^2C_{2}^2\|\nabla u \|_{L^2(\Omega)}^4
~~\mbox{for all}~~t\in(0,T_{max}),}
\end{array}
$$
which combining with \dref{ddfghgghjjnnhhkklld911cz2.5ghju48}   implies that
$$
\disp\frac{1}{{2}}\frac{d}{dt}\|A^{\frac{1}{2}}u \|^{{{2}}}_{L^{{2}}(\Omega)}
\leq\disp{ \kappa^4C_{1}^2C_{2}^2\|\nabla u \|_{L^2(\Omega)}^4+2 \|\nabla\phi\|^2_{L^\infty(\Omega)}\int_{\Omega}(n ^2+m ^2)~\mbox{for all}~t\in(0,T_{max}).}
$$
By the fact that $\|A^{\frac{1}{2}}u \|^{{{2}}}_{L^{{2}}(\Omega)} = \|\nabla u \|^{{{2}}}_{L^{{2}}(\Omega)},$ we conclude that
\begin{equation}
z'(t)\leq\rho(t)z(t)+ h(t)\disp{~~\mbox{for all}~~t\in(0,T_{max})},
\label{ddfghgffgghggddhjjjhjjnnhhkklld911cz2.5ghju48}
\end{equation}
where
$$
z(t) :=\int_{\Omega}|\nabla u (\cdot, t)|^2
$$
as well as
$$
\rho(t) =2\kappa^4C_{1}^2C_{2}^2\int_{\Omega}|\nabla u (\cdot, t)|^2
$$
and
$$
h(t)=4 \|\nabla\phi\|^2_{L^\infty(\Omega)}\int_{\Omega}[n ^2(\cdot,t)+m ^2(\cdot,t)].
$$
However,  Lemma \ref{ggh788999hjllsssemmkkllla3.5} along with  and Lemma \ref{fvfgsdfggfflemma45} warrants that for  some  positive constant $\alpha_0$,
\begin{equation}
\label{3333cz2.kkk5kke345677ddfddffddddf89001ddff214114114rrggjjkk}
\disp\int_{t}^{t+\tau}\int_\Omega  |\nabla {u }|^2\leq\alpha_{0}~~\mbox{for all}~~ t\in(0,T_{max}-\tau)
\end{equation}
and
\begin{equation}
\label{3333cz2.kkk5kke345fffff677ddfdddddf89001ddff214114114rrggjjkk}
\disp\int_{t}^{t+\tau}\int_\Omega (n   ^{2}+m   ^{2})\leq\alpha_{0}~~\mbox{for all}~~ t\in(0,T_{max}-\tau)
\end{equation}
with $\tau$ is the same as \dref{jvgxddrg}.
Now, \dref{3333cz2.kkk5kke345677ddfddffddddf89001ddff214114114rrggjjkk}
and \dref{3333cz2.kkk5kke345fffff677ddfdddddf89001ddff214114114rrggjjkk}    ensure that for all $t\in(0,T_{max}-\tau)$
$$
\int_{t}^{t+\tau}\rho(s)ds\leq\disp{ 2\kappa^4C_{1}^2C_{2}^2\alpha_0}
$$
and
$$
\int_{t}^{t+\tau}h(s)ds\leq\disp{ 4 \|\nabla\phi\|^2_{L^\infty(\Omega)}\alpha_0.}
$$
For given $t\in (0, T_{max})$, applying \dref{3333cz2.kkk5kke345677ddfddffddddf89001ddff214114114rrggjjkk} again,
we can choose $t_0 \geq 0$ such that $t_0\in [t-\tau, t)$ and
$$
\disp{\int_{\Omega}|\nabla u (\cdot,t_0)|^2\leq C_{3},}
$$
which combined with \dref{ddfghgffgghggddhjjjhjjnnhhkklld911cz2.5ghju48} implies that
\begin{equation}
\begin{array}{rl}
z(t)\leq&\disp{z(t_0)e^{\int_{t_0}^t\rho(s)ds}+\int_{t_0}^te^{\int_{s}^t\rho(\tau)d\tau}h(s)ds}
\\
\leq&\disp{C_{3}e^{2\kappa^4C_{1}^2C_{2}^2\alpha_0}+\int_{t_0}^te^{2\kappa^4C_{1}^2C_{2}^2\alpha_0}h(s)ds}
\\
\leq&\disp{C_{3}e^{2\kappa^4C_{1}^2C_{2}^2\alpha_0}+e^{2\kappa^4C_{1}^2C_{2}^2\alpha_0}4 \|\nabla\phi\|^2_{L^\infty(\Omega)}\alpha_{0}~~\mbox{for all}~~t\in(0,T_{max})}
\end{array}
\label{czfvgb2.5ghhddffggjuyghjjjuffghhhddfghhccvjkkklllhhjkkviihjj}
\end{equation}
by integration. The claimed inequality \dref{ddxcvbbggddfgcz2vv.5ghju48cfg924ghyuji} thus results from \dref{czfvgb2.5ghhddffggjuyghjjjuffghhhddfghhccvjkkklllhhjkkviihjj}.

%Other cases can be proved very similarly and easily. And therefore, we omit it here.
\end{proof}

Recalling the above Lemma,
by applying Sobolev embedding $W^{1,2}(\Omega)\hookrightarrow L^p(\Omega)$, we can derive the following Lemma.

\begin{lemma}\label{leddddmma63kkllll0jklhhjj}
Let  $\alpha>-1,\kappa = 0$ or  $\alpha>-\frac{1}{2},\kappa \in\mathbb{R}$. Then for any $p>1,$ there exists a positive constant $C$ such that
\begin{equation}
\int_{\Omega}{| u (\cdot,t)|^p}\leq C~~\mbox{for all}~~ t\in(0, T_{max}).
\label{ddxcvbbggddfgcz2vv.5ghju48cfg924ghyuji}
\end{equation}
%While if %$\alpha>-1$,$\kappa=0$ or
%$\alpha\geq-\frac{1}{2}$, $\kappa\in\mathbb{R}$,
% Then for any $p>1$ and any finite $T\in (0,T_{max}]$,  %there exist constants
%%$l_0> \frac{N}{2}$ and $C(l_0,T)>0$ such that
%%%there exists $K(m_0) > 0$
%%% with the property
%%that whenever  \dref{ccvvx1.731426677gg}  holds with
%% $\disp\int_\Omega u_0 \leq m_0$,
%one can find $C > 0$
%such that
%%
%\begin{equation}
%\int_{\Omega}{|u (\cdot,t)|^p}\leq C~~\mbox{for all}~~ t\in(0, T).
%\label{ddxcvbbggddfgcz2vv.5ghju48cfg92llll4ghyuji}
%\end{equation}
\end{lemma}
\begin{proof}
%Case $\alpha>-1,\kappa = 0$ or  $\alpha>-\frac{1}{2},\kappa \in\mathbb{R}$:
For any $p>1$, recalling Lemma \ref{lemma63kkllll0jklhhjj},
by applying Sobolev embedding $W^{1,2}(\Omega)\hookrightarrow L^p(\Omega)$, we derive that  \dref{ddxcvbbggddfgcz2vv.5ghju48cfg924ghyuji} holds.

%Recalling Lemma \ref{lemmasdde45673.4dd} and Lemma  \ref{lemma63kkllll0jklhhjj}, the case $\alpha\geq-\frac{1}{2}$, $\kappa\in\mathbb{R}$   can be proved similarly. Therefore, we omit it.
\end{proof}

Based on above lemmas,
we could see the solution $(n ,c ,m ,u )$ of \dref{1.dddduiikkldffdffg1}, \dref{ccvvx1.73142sdd6677gg}--\dref{ccvvx1.ddfff731426677gg} is indeed globally solvable. %Moreover,  we will show that it is actually
%classical solution.
%By means of a straightforward extraction procedure on the basis of the Arzel`a-Ascoli theorem, the
%above estimates now enable us to construct a limit which, according to a well-known additional
%regularity argument for the limit component n, in fact can be seen to enjoy the desired smoothness
%properties and to solve (1.1), (1.8), (1.9) in the classical sense in
% ¡Á (0,¡Þ).

\begin{lemma}\label{lemmassddddff45630223}
Let  $(n ,c ,m ,u )$ be a solution of \dref{1.dddduiikkldffdffg1}, \dref{ccvvx1.73142sdd6677gg}--\dref{ccvvx1.ddfff731426677gg}.
Assume that   \dref{sssdddd1.1sssfghyuisdakkkllljjjkk} and \dref{sssdddd1.1sssfghyuisdakkkllljjddddjkk} hold. Then   $(n ,c ,m ,u )$ solves \dref{1.dddduiikkldffdffg1}, \dref{ccvvx1.73142sdd6677gg}--\dref{ccvvx1.ddfff731426677gg} in the classical sense in $\Omega\times (0,\infty).$
Furthermore, for each $T > 0$, there exists $C(T ) > 0$ such that %for all
%$\varepsilon\in(0, 1)$,
\begin{equation}
\|n (\cdot, t)\|_{L^\infty(\Omega)}+\|c (\cdot, t)\|_{W^{1,\infty}(\Omega)}+\|m (\cdot, t)\|_{W^{1,\infty}(\Omega)}+\|A^\gamma u (\cdot, t)\|_{L^2(\Omega)}\leq C~~ \mbox{for all}~~ t\in(0,T).
\label{1.163072xggtsssstsdddyyu}
\end{equation}
While if  $\alpha>-1,\kappa = 0$ or  $\alpha>-\frac{1}{2},\kappa \in\mathbb{R},$
 then this solution is bounded in
$\Omega\times(0,\infty)$ in the sense that
\begin{equation}
\|n (\cdot, t)\|_{L^\infty(\Omega)}+\|c (\cdot, t)\|_{W^{1,\infty}(\Omega)}+\|m (\cdot, t)\|_{W^{1,\infty}(\Omega)}+\|A^\gamma u (\cdot, t)\|_{L^2(\Omega)}\leq C~~ \mbox{for all}~~ t>0.
\label{1.163072xggttsdddyyu}
\end{equation}

\end{lemma}
\begin{proof}
In what follows, let $C, C_i$ denote some different constants, %which are ,
and if no special explanation, they depend at most on $\Omega, \phi, m_0, n_0, c_0$ and
$u_0$. We only need prove \dref{1.163072xggttsdddyyu}, since, \dref{1.163072xggtsssstsdddyyu} can be proved very similarly.

%
%Case

{\bf Step 1. The boundedness of $\|\nabla m  (\cdot, t)\|_{L^{2}(\Omega)}$   for all $t\in (0, T_{max})$}

We multiply the third equation in \dref{1.dddduiikkldffdffg1} by $-\Delta m  $ and integrate by parts to see that
%taking ${c }$ as the test function for the second  equation of \dref{1.dddduiikkldffdffg1}, \dref{ccvvx1.73142sdd6677gg}--\dref{ccvvx1.ddfff731426677gg}, using $\nabla\cdot u =0$ and \dref{ffggg1.ffggvddfghhghhhhbbhhjjjnxxccvvn1}, with the help of the H\"{o}lder  inequality yields  that
\begin{equation}
\begin{array}{rl}
&\disp\frac{1}{{2}}\disp\frac{d}{dt}\|\nabla{m  }\|^{{{2}}}_{L^{{2}}(\Omega)}+
\int_{\Omega} |\Delta m  |^2
\\
=&\disp{\int_{\Omega} m n  \Delta m  +\int_{\Omega} (u  \cdot\nabla m )\Delta m  }
\\
%=&\disp{-\int_{\Omega} m  \Delta c -\int_{\Omega}\nabla c  \nabla (u  \cdot\nabla c )}
%\\
=&\disp{\int_{\Omega} m n  \Delta m -\int_{\Omega}\nabla m  (\nabla u  \cdot\nabla m )}\\
\leq&\disp{\int_{\Omega} m n  \Delta m +\|\nabla u  \|_{L^{2}(\Omega)}\|\nabla m  \|_{L^{4}(\Omega)}^2~~\mbox{for all}~~ t\in(0,T_{max}),}
\end{array}
\label{uiiiihhxxcsssdfvvjjczddfddssssdfff2.5}
\end{equation}
where in the last inequality we have used the Cauchy-Schwarz inequality.
%where we have used the fact that
%$$
%\disp{\int_{\Omega}\nabla c  \cdot(D^2 c  \cdot u  )
%=\frac{1}{2}\int_{\Omega}  u  \cdot\nabla|\nabla c  |^2=0
%~~\mbox{for all}~~ t\in(0,T_{max}).}
%$$
%\begin{equation}
%\begin{array}{rl}
%\disp\int_{\Omega}(u\cdot\nabla c) \Delta c=&\disp{-\int_{\Omega} \nabla c\cdot\nabla(u\cdot\nabla c) }\\
%=&\disp{-\int_{\Omega} \nabla c\cdot\nabla(u\cdot\nabla c)-\int_{\Omega} \nabla c\cdot\nabla(u\cdot\nabla c) ~~\mbox{for all}~~ t\in(0, T_{max}),}\\
%%\leq&\disp{\|n \|_{L^{\frac{6}{5}}(\Omega)}\|c \|_{L^{6}(\Omega)}.}\\
%\end{array}
%\label{dddfggaassshhxxcdfvvjjcz2.5}
%\end{equation}
%\begin{equation}
%\begin{array}{rl}
%\disp\frac{1}{{2}}\disp\frac{d}{dt}\|{\nabla c }\|^{{{2}}}_{L^{{2}}(\Omega)}+
%\int_{\Omega} |\Delta c  |^2+\int_{\Omega} |\nabla c |^2=&\disp{-\int_{\Omega} m\Delta c  +\int_{\Omega}(u\cdot\nabla c) \Delta c  ~~\mbox{for all}~~ t\in(0, T_{max}),}\\
%%\leq&\disp{\|n \|_{L^{\frac{6}{5}}(\Omega)}\|c \|_{L^{6}(\Omega)}.}\\
%\end{array}
%\label{dddfggaassshhxxcdfvvjjcz2.5}
%\end{equation}
Next,  with the help of the Young inequality as well as Lemma \ref{lsssemma3788999.5} (or Lemma \ref{ggh788999hjllsssemmkkllla3.5}) and \dref{ddczhjjjj2.5ghju48cfg9ssdd24},
\begin{equation}
\begin{array}{rl}
\disp\int_{\Omega} n m \Delta m \leq&\disp{\|m \|^2_{L^\infty(\Omega)}\int_{\Omega} n^2 +\frac{1}{4}\int_{\Omega}|\Delta m |^2  }\\
\leq&\disp{\| m_0\|^2_{L^\infty(\Omega)}C_1+\frac{1}{4}\int_{\Omega}|\Delta m |^2  ~~\mbox{for all}~~ t\in(0, T_{max})}\\
%\leq&\disp{\|n \|_{L^{\frac{6}{5}}(\Omega)}\|c \|_{L^{6}(\Omega)}.}\\
\end{array}
\label{ssdddaassshhxxcdfsssssvvjjcz2.5}
\end{equation}
with some $C_1>0$.
In the last summand in \dref{hhxxcsssdfvvjjczddfdddfff2.5},  thanks to \dref{ddfgczhhhh2.5sddddghju48cfg924ghyuji} and in view of the Gagliardo-Nirenberg inequality, we can find $C_2> 0$ and  $C_3> 0$ fulfilling
%integrate by parts to find that
%\begin{equation}
%\begin{array}{rl}
%\disp \|\nabla c \|_{L^{4}(\Omega)}^2\leq\disp{C_{6}\|\Delta c  \|_{L^{2}(\Omega)}\|\nabla c \|_{L^{2}(\Omega)}
%~~\mbox{for all}~~ t\in(0,T_{max}).}\\
%\end{array}
%\label{hhxxcsssdfvvjjcddfffddffzddfdddfff2.5}
%\end{equation}
%
%Meanwhile, we can further use Gagliardo-Nirenberg inequality and the elliptic regularity (\cite{Gilbarg4441215}) to conclude that for some $C_{6}> 0$,
$$
\begin{array}{rl}
\disp \|\nabla m  \|_{L^{4}(\Omega)}^4\leq&\disp{C_{2}\|\Delta m  \|_{L^{2}(\Omega)}^2\|m  \|_{L^{\infty}(\Omega)}^2+C_{2}\|m  \|_{L^{\infty}(\Omega)}^4}\\
\leq&\disp{C_{3}\|\Delta m  \|_{L^{2}(\Omega)}+C_{3}
~~\mbox{for all}~~ t\in(0,T_{max}),}\\
\end{array}
$$
%and similarly,
which implies that
$$
\begin{array}{rl}
\disp \|\Delta m  \|_{L^{4}(\Omega)}^4\geq&\disp{\frac{1}{C_3}(\|\nabla m  \|_{L^{4}(\Omega)}^4-1)
~~\mbox{for all}~~ t\in(0,T_{max}).}\\
\end{array}
$$
%
%
%with some $C_4>0$ and $C_5>0.$
According to the  Young  inequality as well as  \dref{uiiiihhxxcsssdfvvjjczddfddssssdfff2.5} and \dref{ssdddaassshhxxcdfsssssvvjjcz2.5}, this means that with some $C_4>0$ we have
%This   together with the  Young  inequality as well as  \dref{uiiiihhxxcsssdfvvjjczddfddssssdfff2.5} and \dref{ssdddaassshhxxcdfsssssvvjjcz2.5}  yields to
\begin{equation}
\begin{array}{rl}
\disp\frac{1}{{2}}\disp\frac{d}{dt}\|\nabla{m  }\|^{{{2}}}_{L^{{2}}(\Omega)}+
\|\nabla{m  }\|^{{{2}}}_{L^{{2}}(\Omega)}
\leq&\disp{C_4\|\nabla u  \|_{L^{2}(\Omega)}^2+C_4~~\mbox{for all}~~ t\in(0,T_{max}).}\\
\end{array}
\label{uiiiihhxxcsssdfvvjjczddfddssssssdfff2.5}
\end{equation}
%with some positive constant $C_4$.
 This combined with Lemma \ref{leddddmma63kkllll0jklhhjj} yields to for some $C_5>0$ such that
\begin{equation}
\begin{array}{rl}
\disp{\|\nabla m  (\cdot, t)\|_{L^{2}(\Omega)}}
\leq&\disp{C_5~~\mbox{for all}~~ t\in(0,T_{max})}.\\
%\leq&\disp{\|\nabla e^{t(\Delta-1)} m_0\|_{L^{r}(\Omega)}+\int_{0}^t\|\nabla e^{(t-s)(\Delta-1)}(n(s)+\nabla \cdot(u(s) c(s))\|_{L^r(\Omega)}ds}\\
%\leq&\disp{C_{25}\|\nabla c_0\|_{L^{r}(\Omega)}+\int_{0}^t[1+(t-s)^{-\frac{1}{2}-\frac{N}{2}(\frac{1}{p_0}-\frac{1}{r})}] e^{-(t-s)}\|n(s)\|_{L^{p_0}(\Omega)}ds}\\
%&\disp{+\int_{0}^t\|\nabla e^{(t-s)(\Delta-1)}\nabla \cdot(u(s)\nabla c(s)\|_{L^r(\Omega)}ds}\\
%\leq&\disp{C_{20}\tau^{-\theta}+C_{20}\int_{0}^t(t-s)^{-\theta}e^{-\mu(t-s)}+C_{20}\int_{0}^t(t-s)^{-\theta}e^{-\mu(t-s)}[\|n(s)\|_{L^4(\Omega)}+
%\|\nabla c(s)\|_{L^4(\Omega)}]ds}\\
%\leq&\disp{C_{21}~~ \mbox{for all}~~ t\in(\tau,T_{max})}\\
%\leq&\disp{\|A^\gamma
%e^{-tA}u_0\|_{L^2(\Omega)}+C_1|\kappa|\int_{0}^t(t-\tau)^{-\gamma}e^{-\mu(t-\tau)}\|(Yu(\cdot,\tau) \cdot \nabla)u(\cdot,\tau)\|_{L^2(\Omega)}d\tau}\\
% &+\disp{
% C_1\int_{0}^t(t-\tau)^{-\gamma}e^{-\mu(t-\tau)}[\|n(\cdot,\tau)\nabla\phi\|_{L^2(\Omega)}+\|g(\cdot,\tau)\|_{L^2(\Omega)}]d\tau}\\
\end{array}
\label{44444zjccfgghhhfgbhjcvvvbsssscz2.5297x96301ku}
\end{equation}
%\begin{equation}
%\begin{array}{rl}
%&\disp-\int_{\Omega}\nabla c  \nabla (\nabla u  \cdot\nabla c )
%\\
%\leq&\disp{\|\nabla u  \|_{L^{2}(\Omega)}[C_{2}\|\Delta c  \|_{L^{2}(\Omega)}+C_{2}]}
%\\
%%\leq&\disp{C_{6}\|\nabla u \|_{L^{2}(\Omega)}\|\Delta c  \|_{L^{2}(\Omega)}\|\nabla c \|_{L^{2}(\Omega)}}
%%\\
%\leq&\disp{C_{2}^2\|\nabla u  \|_{L^{2}(\Omega)}^2
%+\frac{1}{4}\|\Delta c  \|_{L^{2}(\Omega)}^2+C_3~~\mbox{for all}~~ t\in(0,T_{max}).}
%\end{array}
%\label{hhxxcsssdfvvjjcddffzddfdddfff2.5}
%\end{equation}
%Inserting \dref{ssdddaassshhxxcdfvvjjcz2.5} and \dref{hhxxcsssdfvvjjcddffzddfdddfff2.5} into  \dref{hhxxcsssdfvvjjczddfdddfff2.5}, we have
%\begin{equation}
%\begin{array}{rl}
%\disp\frac{1}{{2}}\disp\frac{d}{dt}\|\nabla{c  }\|^{{{2}}}_{L^{{2}}(\Omega)}+\frac{1}{2}
%\int_{\Omega} |\Delta c  |^2+ \int_{\Omega} | \nabla c  |^2
%\leq&\disp{C_{2}^2\|\nabla u  \|_{L^{2}(\Omega)}^2+C_4.}
%\end{array}
%\label{hhxxcsssdfvvjjczddfssdddddfff2.5}
%\end{equation}
%As a consequence of \dref{hhxxcsssdfvvjjczddfssdddddfff2.5}, \dref{czfvgb2.5ghhjuyuccvviihjj} is valid by a choice of
%$\varrho_1:= \max\{C_{2}^2, C_4\}$.

{\bf Step 2. The boundedness of $\|\nabla m  (\cdot, t)\|_{L^{{q}}(\Omega)}$ and $\|\nabla c  (\cdot, t)\|_{L^{{q}}(\Omega)}$   for all $q>8$ and  $t\in (0, T_{max})$}

For any $q>8$, one can choose $p\in(1,2)$ such that
$$-\frac{1}{2}-\frac{2}{2}(\frac{1}{p}-\frac{1}{q})>-1,$$
so that,
an application of the variation of
constants formula for $m $ and $c $  and using the $L^p$-$L^q$ estimates associated heat semigroup leads to
\begin{equation}
\begin{array}{rl}
&\disp{\|\nabla m  (\cdot, t)\|_{L^{{q}}(\Omega)}}\\
\leq&\disp{\|\nabla e^{t(\Delta-1)} m_0\|_{L^{{q}}(\Omega)}+
\int_{0}^t\|\nabla e^{(t-s)(\Delta-1)}(m  (\cdot,s)-n  (\cdot,s)m  (\cdot,s))\|_{L^{{q}}(\Omega)}ds}\\
&\disp{+\int_{0}^t\|\nabla e^{(t-s)(\Delta-1)}(u  (\cdot,s) \cdot \nabla m  (\cdot,s))\|_{L^{{q}}(\Omega)}ds}\\
\leq&\disp{C_1+
\int_{0}^t[1+(t-s)^{-\frac{1}{2}-\frac{2}{2}(\frac{1}{p}-\frac{1}{q})}] e^{-\lambda_1(t-s)}[\|m  (\cdot,s)\|_{L^{p}(\Omega)}+\|n  (\cdot,s) \|_{L^{p}(\Omega)}\|m  (\cdot,s))\|_{L^{\infty}(\Omega)}]ds}\\
&\disp{+\int_{0}^t[1+(t-s)^{-\frac{1}{2}-\frac{2}{2}(\frac{1}{p}-\frac{1}{q})}] e^{-\lambda_1(t-s)}\|u  (\cdot,s) \cdot \nabla m  (\cdot,s)\|_{L^{p}(\Omega)}ds}\\
\leq&\disp{C_1+
\int_{0}^t[1+(t-s)^{-\frac{1}{2}-\frac{2}{2}(\frac{1}{p}-\frac{1}{q})}] e^{-\lambda_1(t-s)}[\|m  (\cdot,s)\|_{L^{p}(\Omega)}+\|n  (\cdot,s) \|_{L^{p}(\Omega)}\|m  (\cdot,s))\|_{L^{\infty}(\Omega)}]ds}\\
&\disp{+\int_{0}^t[1+(t-s)^{-\frac{1}{2}-\frac{2}{2}(\frac{1}{p}-\frac{1}{q})}] e^{-\lambda_1(t-s)}\|u  (\cdot,s)\|_{L^{\frac{p}{2-p}}(\Omega)}\|\nabla m  (\cdot,s)\|_{L^{2}(\Omega)}ds}\\
\leq&\disp{C_2~~ \mbox{for all}~~ t\in(0,T_{max})}\\
%\leq&\disp{\|\nabla e^{t(\Delta-1)} m_0\|_{L^{r}(\Omega)}+\int_{0}^t\|\nabla e^{(t-s)(\Delta-1)}(n(s)+\nabla \cdot(u(s) c(s))\|_{L^r(\Omega)}ds}\\
%\leq&\disp{C_{25}\|\nabla c_0\|_{L^{r}(\Omega)}+\int_{0}^t[1+(t-s)^{-\frac{1}{2}-\frac{N}{2}(\frac{1}{p_0}-\frac{1}{r})}] e^{-(t-s)}\|n(s)\|_{L^{p_0}(\Omega)}ds}\\
%&\disp{+\int_{0}^t\|\nabla e^{(t-s)(\Delta-1)}\nabla \cdot(u(s)\nabla c(s)\|_{L^r(\Omega)}ds}\\
%\leq&\disp{C_{20}\tau^{-\theta}+C_{20}\int_{0}^t(t-s)^{-\theta}e^{-\mu(t-s)}+C_{20}\int_{0}^t(t-s)^{-\theta}e^{-\mu(t-s)}[\|n(s)\|_{L^4(\Omega)}+
%\|\nabla c(s)\|_{L^4(\Omega)}]ds}\\
%\leq&\disp{C_{21}~~ \mbox{for all}~~ t\in(\tau,T_{max})}\\
%\leq&\disp{\|A^\gamma
%e^{-tA}u_0\|_{L^2(\Omega)}+C_1|\kappa|\int_{0}^t(t-\tau)^{-\gamma}e^{-\mu(t-\tau)}\|(Yu(\cdot,\tau) \cdot \nabla)u(\cdot,\tau)\|_{L^2(\Omega)}d\tau}\\
% &+\disp{
% C_1\int_{0}^t(t-\tau)^{-\gamma}e^{-\mu(t-\tau)}[\|n(\cdot,\tau)\nabla\phi\|_{L^2(\Omega)}+\|g(\cdot,\tau)\|_{L^2(\Omega)}]d\tau}\\
\end{array}
\label{44444zjccfgghhhfgbhjcvvvbscz2.5297x96301ku}
\end{equation}
and
\begin{equation}
\begin{array}{rl}
&\disp{\|\nabla c  (\cdot, t)\|_{L^{q}(\Omega)}}\\
\leq&\disp{\|\nabla e^{t(\Delta-1)} c_0\|_{L^{q}(\Omega)}+
\int_{0}^t\|\nabla e^{(t-s)(\Delta-1)}(m  (\cdot,s))\|_{L^{q}(\Omega)}ds}\\
&\disp{+\int_{0}^t\|\nabla e^{(t-s)(\Delta-1)}(u (\cdot,s) \cdot\nabla c  (\cdot,s))\|_{L^{q}(\Omega)}ds}\\
\leq&\disp{C_3+
\int_{0}^t[1+(t-s)^{-\frac{1}{2}-\frac{2}{2}(\frac{1}{p}-\frac{1}{q})}] e^{-\lambda_1(t-s)}\|m  (\cdot,s)\|_{L^{p}(\Omega)}ds}\\
&\disp{+\int_{0}^t[1+(t-s)^{-\frac{1}{2}-\frac{2}{2}(\frac{1}{p}-\frac{1}{q})}] e^{-\lambda_1(t-s)}\|u (\cdot,s)\nabla c  (\cdot,s))\|_{L^{p}(\Omega)}ds}\\
\leq&\disp{C_3+
\int_{0}^t[1+(t-s)^{-\frac{1}{2}-\frac{2}{2}(\frac{1}{p}-\frac{1}{q})}] e^{-\lambda_1(t-s)}\|m  (\cdot,s)\|_{L^{p}(\Omega)}ds}\\
&\disp{+\int_{0}^t[1+(t-s)^{-\frac{1}{2}-\frac{2}{2}(\frac{1}{p}-\frac{1}{q})}] e^{-\lambda_1(t-s)}\|u (\cdot,s)\|_{L^{\frac{2p}{2-p}}(\Omega)}\|\nabla c  (\cdot,s))\|_{L^{2}(\Omega)}ds}\\
\leq&\disp{C_4~~ \mbox{for all}~~ t\in(0,T_{max})}\\
%\leq&\disp{\|\nabla e^{t(\Delta-1)} m_0\|_{L^{r}(\Omega)}+\int_{0}^t\|\nabla e^{(t-s)(\Delta-1)}(n(s)+\nabla \cdot(u(s) c(s))\|_{L^r(\Omega)}ds}\\
%\leq&\disp{C_{25}\|\nabla c_0\|_{L^{r}(\Omega)}+\int_{0}^t[1+(t-s)^{-\frac{1}{2}-\frac{N}{2}(\frac{1}{p_0}-\frac{1}{r})}] e^{-(t-s)}\|n(s)\|_{L^{p_0}(\Omega)}ds}\\
%&\disp{+\int_{0}^t\|\nabla e^{(t-s)(\Delta-1)}\nabla \cdot(u(s)\nabla c(s)\|_{L^r(\Omega)}ds}\\
%\leq&\disp{C_{20}\tau^{-\theta}+C_{20}\int_{0}^t(t-s)^{-\theta}e^{-\mu(t-s)}+C_{20}\int_{0}^t(t-s)^{-\theta}e^{-\mu(t-s)}[\|n(s)\|_{L^4(\Omega)}+
%\|\nabla c(s)\|_{L^4(\Omega)}]ds}\\
%\leq&\disp{C_{21}~~ \mbox{for all}~~ t\in(\tau,T_{max})}\\
%\leq&\disp{\|A^\gamma
%e^{-tA}u_0\|_{L^2(\Omega)}+C_1|\kappa|\int_{0}^t(t-\tau)^{-\gamma}e^{-\mu(t-\tau)}\|(Yu(\cdot,\tau) \cdot \nabla)u(\cdot,\tau)\|_{L^2(\Omega)}d\tau}\\
% &+\disp{
% C_1\int_{0}^t(t-\tau)^{-\gamma}e^{-\mu(t-\tau)}[\|n(\cdot,\tau)\nabla\phi\|_{L^2(\Omega)}+\|g(\cdot,\tau)\|_{L^2(\Omega)}]d\tau}\\
\end{array}
\label{44444zjccfgghhsssshfgbhjcvvvbscz2.5297x96301ku}
\end{equation}
with some $C_i(i=1,2,3,4)>0.$

{\bf Step 3. The boundedness of $\|n  (\cdot, t)\|_{L^{\infty}(\Omega)}$   for all  $t\in (0, T_{max})$}

%Fix $T\in (0, T_{max})$.
 Let %$M(T):=\sup_{t\in(0,T)}\|n  (\cdot,t)\|_{L^\infty(\Omega)}$ and
$\tilde{h} := S (n   )\nabla c  +u  $. Then by  \dref{44444zjccfgghhsssshfgbhjcvvvbscz2.5297x96301ku}, \dref{x1.73142vghf48gg}  and Lemma \ref{leddddmma63kkllll0jklhhjj},
there exists $C_{5} > 0$ such that
\begin{equation}
\begin{array}{rl}
\|\tilde{h}  (\cdot, t)\|_{L^{4}(\Omega)}\leq&\disp{C_{5}~~ t\in(0,T_{max}),}\\
\end{array}
\label{cz2ddff.57151ccvhhjjjkkkuuifghhhivhccvvhjjjkkhhggjjllll}
\end{equation}
where we have used the H\"{o}lder inequality.
%where we have used \dref{1.163072x} and  the boundedness of $\|c (\cdot, t)\|_{W^{1,\infty}(\Omega)}$  for all  $t\in (\tau, T_{max})$ with $\tau\in(0,T_{max})$.
%where $q<6$.
Hence, due to the fact that $\nabla\cdot u  =0$,  again,  by means of an
associate variation-of-constants formula for $n $, we can derive
\begin{equation}
n  (\cdot,t)=e^{(t-t_0)\Delta}n  (\cdot,t_0)-\int_{t_0}^{t}e^{(t-s)\Delta}\nabla\cdot(n  (\cdot,s)\tilde{h}  (\cdot,s)) ds-\int_{t_0}^{t}e^{(t-s)\Delta}(n  (\cdot,s)m  (\cdot,s)) ds
\label{sss5555fghbnmcz2.5ghjjjkkklu48cfg924ghyuji}
\end{equation}
for all
$t\in(t_0, T_{max})$, where $t_0 := (t-1)_{+}$.
As the last summand in \dref{5555fghbnmcz2.5ghjjjkkklu48cfg924ghyuji} is non-positive by the maximum principle, we can thus estimate
\begin{equation}
\|n  (\cdot,t)\|_{L^\infty(\Omega)}\leq \|e^{(t-t_0)\Delta}n  (\cdot,t_0)\|_{L^\infty(\Omega)}+\int_{t_0}^{t}\|
e^{(t-s)\Delta}\nabla\cdot(n  (\cdot,s)\tilde{h}  (\cdot,s))\|_{L^\infty(\Omega)} ds,~~ t\in(t_0, T_{max}).
\label{5555fghbnmcz2.5ghjjjkkklu48cfg924ghyuji}
\end{equation}
If $t\in(0,1]$,
by virtue of the maximum principle, we derive that
\begin{equation}
\begin{array}{rl}
\|e^{(t-t_0)\Delta}n  (\cdot,t_0)\|_{L^{\infty}(\Omega)}\leq &\disp{\|n_0\|_{L^{\infty}(\Omega)},}\\
%\leq&\disp{C_{20}\tau^{-\theta}+C_{20}\int_{0}^t(t-s)^{-\theta}e^{-\lambda(t-s)}+C_{20}\int_{0}^t(t-s)^{-\theta}e^{-\lambda(t-s)}[\|n(s)\|_{L^4(\Omega)}+
%\|\nabla c(s)\|_{L^4(\Omega)}]ds}\\
%\leq&\disp{C_{21}~~ \mbox{for all}~~ t\in(\tau,T_{max})}\\
%\leq&\disp{\|A^\gamma
%e^{-tA}u_0\|_{L^2(\Omega)}+C_1|\kappa|\int_{0}^t(t-\tau)^{-\gamma}e^{-\lambda(t-\tau)}\|(Yu(\cdot,\tau) \cdot \nabla)u(\cdot,\tau)\|_{L^2(\Omega)}d\tau}\\
% &+\disp{
% C_1\int_{0}^t(t-\tau)^{-\gamma}e^{-\lambda(t-\tau)}[\|n(\cdot,\tau)\nabla\phi\|_{L^2(\Omega)}+\|g(\cdot,\tau)\|_{L^2(\Omega)}]d\tau}\\
\end{array}
\label{zjccffgbhjffghhjcghhhjjjvvvbscz2.5297x96301ku}
\end{equation}
while if $t > 1$ then with the help of the  $L^p$-$L^q$ estimates for the Neumann heat semigroup and \dref{ddfgczhhhh2.5ghju48cfg924ghyuji}, we conclude that
\begin{equation}
\begin{array}{rl}
\|e^{(t-t_0)\Delta}n  (\cdot,t_0)\|_{L^{\infty}(\Omega)}\leq &\disp{C_{5}(t-t_0)^{-\frac{2}{2}}\|n  (\cdot,t_0)\|_{L^{1}(\Omega)}\leq C_{6}.}\\
%\leq&\disp{C_{20}\tau^{-\theta}+C_{20}\int_{0}^t(t-s)^{-\theta}e^{-\lambda(t-s)}+C_{20}\int_{0}^t(t-s)^{-\theta}e^{-\lambda(t-s)}[\|n(s)\|_{L^4(\Omega)}+
%\|\nabla c(s)\|_{L^4(\Omega)}]ds}\\
%\leq&\disp{C_{21}~~ \mbox{for all}~~ t\in(\tau,T_{max})}\\
%\leq&\disp{\|A^\gamma
%e^{-tA}u_0\|_{L^2(\Omega)}+C_1|\kappa|\int_{0}^t(t-\tau)^{-\gamma}e^{-\lambda(t-\tau)}\|(Yu(\cdot,\tau) \cdot \nabla)u(\cdot,\tau)\|_{L^2(\Omega)}d\tau}\\
% &+\disp{
% C_1\int_{0}^t(t-\tau)^{-\gamma}e^{-\lambda(t-\tau)}[\|n(\cdot,\tau)\nabla\phi\|_{L^2(\Omega)}+\|g(\cdot,\tau)\|_{L^2(\Omega)}]d\tau}\\
\end{array}
\label{zjccffgbhjffghhjcghghjkjjhhjjjvvvbscz2.5297x96301ku}
\end{equation}
%Again by the maximum principle we have
%\begin{equation}
%\begin{array}{rl}
%\disp\int_{t_0}^t \|e^{(t-s)\Delta}(an(\cdot,s)-\mu n^r(\cdot,s))\|_{L^\infty(\Omega)}ds\leq&
%\disp\int_{t_0}^t\sup_{n\geq0}(an-\mu n^r)_{+}ds\\
%\leq&\disp\int_{t_0}^ta_{+}^{\frac{r}{r-1}}\mu^{-\frac{1}{r-1}}r^{-\frac{r}{r-1}}(r-1)\\
%\leq&\disp 2a_{+}^{\frac{r}{r-1}}\mu^{-\frac{1}{r-1}}r^{-\frac{r}{r-1}}(r-1)~~\mbox{for all}~~ t\in(0, T).\\
%\end{array}
%\label{ccvbccvvbbnnnvddffvbbcvvbccvhhhhhjjvfbbnfgbghjjccmmllffvvggcvvvvbbjjkkdffzjscz2.5297x9630xxy}
%\end{equation}
Finally, we fix an arbitrary $\frac{7}{2}\in(2,4)$ and then once more invoke known smoothing
properties of the
Stokes semigroup  and the H\"{o}lder inequality to find $C_{7}>0,C_{8}>0$ as well as $C_{9}>0$ %and $C_{10}> 0$
such that
\begin{equation}
\begin{array}{rl}
&\disp\int_{t_0}^t\| e^{(t-s)\Delta}\nabla\cdot(n  (\cdot,s)\tilde{h}  (\cdot,s)\|_{L^\infty(\Omega)}ds\\
\leq&\disp C_{7}\int_{t_0}^t(t-s)^{-\frac{1}{2}-\frac{2}{7}}\|n  (\cdot,s)\tilde{h}  (\cdot,s)\|_{L^{\frac{7}{2}}(\Omega)}ds\\
\leq&\disp C_{8}\int_{t_0}^t(t-s)^{-\frac{1}{2}-\frac{2}{7}}\| n  (\cdot,s)\|_{L^{28}(\Omega)}\|\tilde{h}  (\cdot,s)\|_{L^{4}(\Omega)}ds\\
\leq&\disp C_{9}~~\mbox{for all}~~ t\in(0, T_{max}).\\
\end{array}
\label{ccvbccvvbbnnndffghhjjvcvvbccfbbnfgbghjjccmmllffvvggcvvvvbbjjkkdffzjscz2.5297x9630xxy}
\end{equation}
%where $b:=\frac{27}{28}\in(0,1)$
%and
%$$C_5:=C_4C_1^{2-b}\int_{0}^{1}\sigma^{-\frac{1}{2}-\frac{3}{7}}d\sigma.$$
%Since $p>3$, we conclude that
%$-\frac{1}{2}-\frac{3}{7}>-1$.
In combination with \dref{5555fghbnmcz2.5ghjjjkkklu48cfg924ghyuji}--\dref{ccvbccvvbbnnndffghhjjvcvvbccfbbnfgbghjjccmmllffvvggcvvvvbbjjkkdffzjscz2.5297x9630xxy} %and using the definition of $M(T)$
we obtain
$C_{10}> 0$ such that
%\begin{equation}
%\begin{array}{rl}
%&\disp  M(T)\leq C_{10}}~~\mbox{for all}~~ T\in(0, T_{max}).\\
%\end{array}
%\label{ccvbccvvbbnnndffghhjjvcvvfghhhbccfbbnfgbghjjccmmllffvvggcvvvvbbjjkkdffzjscz2.5297x9630xxy}
%\end{equation}
%Hence,  with  some basic calculation, in light of  $T\in (0, T_{max})$ was arbitrary,
%one can get
\begin{equation}
\begin{array}{rl}
\|n  (\cdot, t)\|_{L^{\infty}(\Omega)}\leq&\disp{C_{10}~~ \mbox{for all}~~ t\in(0,T_{max}).}\\
\end{array}
\label{cz2.57ghhhh151ccvhhjjjkkkffgghhuuiivhccvvhjjjkkhhggjjllll}
\end{equation}

{\bf Step 4. The boundedness of $\|A^\gamma u (\cdot, t)\|_{L^2(\Omega)}$ (with $\gamma\in ( \frac{1}{2}, 1)$) and $\| u (\cdot, t)\|_{L^{\infty}(\Omega)}$ for all $t\in (0, T_{max})$}

Applying  the variation-of-constants formula for the projected version of the third
equation in \dref{1.dddduiikkldffdffg1}, we derive that
$$u (\cdot, t) = e^{-tA}u_0 +\int_0^te^{-(t-\tau)A}
\mathcal{P}[n (\cdot,\tau)\nabla\phi-\kappa
(Y u (\cdot,\tau) \cdot \nabla)u (\cdot,\tau)]d\tau~~ \mbox{for all}~~ t\in(0,T_{max}).$$
Therefore, with the help of the  standard smoothing
properties of the Stokes semigroup we derive  that for all $t\in(0,T_{max})$ and $\gamma\in ( \frac{1}{2}, 1)$,
there exist $C_{11} > 0$ as well as  $C_{12} > 0$ and $\lambda>0$  such that
\begin{equation}
\begin{array}{rl}
\|A^\gamma u (\cdot, t)\|_{L^2(\Omega)}\leq&\disp{\|A^\gamma
e^{-tA}u_0\|_{L^2(\Omega)} +\int_0^t\|A^\gamma e^{-(t-\tau)A}h (\cdot,\tau)d\tau\|_{L^2(\Omega)}d\tau}\\
\leq&\disp{\|A^\gamma u_0\|_{L^2(\Omega)} +C_{11}\int_0^t(t-\tau)^{-\gamma-\frac{2}{2}(\frac{1}{p_0}-\frac{1}{2})}e^{-\lambda(t-\tau)}\|h (\cdot,\tau)\|_{L^{p_0}(\Omega)}d\tau}\\
\leq&\disp{C_{12}+C_{11}\int_0^t(t-\tau)^{-\gamma-\frac{2}{2}(\frac{1}{p_0}-\frac{1}{2})}e^{-\lambda(t-\tau)}\|h (\cdot,\tau)
\|_{L^{p_0}(\Omega)}d\tau}\\
\end{array}
\label{ssddcz2.571hhhhh51222ccvvhddfccvvhjjjkkhhggjjllll}
\end{equation}
by using \dref{ccvvx1.731426677gg},
where   $h (\cdot,\tau)=\mathcal{P}[n  (\cdot,\tau)\nabla\phi-\kappa
(Y u  (\cdot,\tau) \cdot \nabla)u (\cdot,\tau)]$ and $p_0\in (1,2)$ which satisfies %that
\begin{equation}p_0>\frac{2}{3-2\gamma}.
\label{cz2.571hhhhh51222ccvvhddfccffgghhhhvvhjjjkkhhggjjllll}
\end{equation}
Now, in order to estimate $\|h (\cdot,\tau)
\|_{L^{p_0}(\Omega)}$, we use the H\"{o}lder inequality and the continuity of $\mathcal{P}$ in $L^p(\Omega;\mathbb{R}^2)$ (see \cite{Fujiwara66612186}) as well as the boundedness of $\|n (\cdot, t)\|_{L^{\infty}(\Omega)}$ (for all  $t\in (0, T_{max})$), we see that
there exist $C_{13},C_{14},C_{15}, C_{16}> 0$ and  $C_{17} > 0$ that
\begin{equation}
\begin{array}{rl}
\|h (\cdot,t)\|_{L^{p_0}(\Omega)}\leq& C_{13}\|(  u  \cdot \nabla)u (\cdot,t)\|_{L^{p_0}(\Omega)}+ C_{13}\|n (\cdot,t)\|_{L^{p_0}(\Omega)}\\
\leq& C_{14}\|  u \|_{L^{\frac{2p_0}{2-p_0}}(\Omega)} \|\nabla u (\cdot,t)\|_{L^{2}(\Omega)}+ C_{24}\\
\leq& C_{15}\|\nabla   u \|_{L^{2}(\Omega)} \|\nabla u (\cdot,t)\|_{L^{2}(\Omega)}+ C_{24}\\
\leq& C_{16}\|\nabla u \|_{L^{2}(\Omega)}^2+ C_{24}\\
\leq& C_{17}~~~\mbox{for all}~~ t\in(0,T_{max})\\
\end{array}
\label{cz2.571hhhhh5122ddfddfffgg2ccvvhddfccffgghhhhvvhjjjkkhhggjjllll}
\end{equation}
by using $W^{1,2}(\Omega)\hookrightarrow L^\frac{2p_0}{2-p_0}(\Omega)$ and the boundedness of $\|\nabla u (\cdot,t)\|_{L^{2}(\Omega)}.$
%is the same as Lemma \ref{lemma45566645630223}.
Inserting \dref{cz2.571hhhhh5122ddfddfffgg2ccvvhddfccffgghhhhvvhjjjkkhhggjjllll} into \dref{ssddcz2.571hhhhh51222ccvvhddfccvvhjjjkkhhggjjllll} and applying \dref{cz2.571hhhhh51222ccvvhddfccffgghhhhvvhjjjkkhhggjjllll}, we conclude that for some $C_{28}$ such that for all $t\in(0,T_{max})$,
\begin{equation}
\begin{array}{rl}
\|A^\gamma u (\cdot, t)\|_{L^2(\Omega)}\leq&\disp{C_{12}+C_{18}\int_0^t(t-\tau)^{-\gamma-\frac{2}{2}(\frac{1}{p_0}-\frac{1}{2})}e^{-\lambda(t-\tau)}\|h (\cdot,\tau)
\|_{L^{p_0}(\Omega)}d\tau,}\\
\end{array}
\label{ssddcz2.571hhjjjhhhhh51222ccvvhddfccvvhjjjksdddffkhhggjjllll}
\end{equation}
where we have used the fact that
$$\begin{array}{rl}\disp\int_{0}^t(t-\tau)^{-\gamma-\frac{2}{2}(\frac{1}{p_0}-\frac{1}{2})}e^{-\lambda(t-\tau)}ds
\leq&\disp{\int_{0}^{\infty}\sigma^{-\gamma-\frac{2}{2}(\frac{1}{p_0}-\frac{1}{2})} e^{-\lambda\sigma}d\sigma<+\infty}\\
\end{array}
$$
by \dref{cz2.571hhhhh51222ccvvhddfccffgghhhhvvhjjjkkhhggjjllll}.
%Observe that $\gamma>\frac{3}{4},$
Observe that  $D(A^\gamma)$ is continuously embedded into $L^\infty(\Omega)$ by $\gamma>\frac{1}{2},$ so that,  \dref{ssddcz2.571hhjjjhhhhh51222ccvvhddfccvvhjjjksdddffkhhggjjllll} yields to
 \begin{equation}
\begin{array}{rl}
\|u (\cdot, t)\|_{L^\infty(\Omega)}\leq  C_{19}~~ \mbox{for all}~~ t\in(0,T_{max}).\\
\end{array}
\label{cz2.5jkkcvvvhjdsdfffffdkfffffkhhgll}
\end{equation}

{\bf Step 4. The boundedness of $\|c  (\cdot, t)\|_{W^{1,\infty}(\Omega)}$ and $\|m  (\cdot, t)\|_{W^{1,\infty}(\Omega)}$
 for all  $t\in (\tau, T_{max})$ with $\tau\in(0,T_{max})$}

Choosing $\theta\in(\frac{1}{2}+\frac{1}{4},1),$ %where $q<\frac{2N}{(N-2)_+}.$
 then the domain of the fractional power $D((-\Delta + 1)^\theta)\hookrightarrow W^{1,\infty}(\Omega)$ (see e.g.  \cite{Horstmann791,Winkler792}). Thus,  %in light of $\alpha>0$,
 using the H\"{o}lder inequality and the $L^p$-$L^q$ estimates associated heat semigroup, %there exists $\lambda$
\begin{equation}
\begin{array}{rl}
&\| c  (\cdot, t)\|_{W^{1,\infty}(\Omega)}\\
\leq&\disp{C_{20}\|(-\Delta+1)^\theta c  (\cdot, t)\|_{L^{4}(\Omega)}}\\
\leq&\disp{C_{21}t^{-\theta}e^{-\lambda_1 t}\|c_0\|_{L^{4}(\Omega)}+C_{21}\int_{0}^t(t-s)^{-\theta}e^{-\lambda_1(t-s)}
\|(m  -u   \cdot \nabla c  )(\cdot,s)\|_{L^{4}(\Omega)}ds}\\
\leq&\disp{C_{22}+C_{22}\int_{0}^t(t-s)^{-\theta}e^{-\mu(t-s)}[\|m (\cdot,s)\|_{L^{\infty}(\Omega)}+\|c  (\cdot,s)\|_{L^{\infty}(\Omega)}+\|u  (\cdot,s)\|_{L^\infty(\Omega)}
\|\nabla c  (\cdot,s)\|_{L^{4}(\Omega)}]ds}\\
\leq&\disp{C_{23}~~ \mbox{for all}~~ t\in(\tau,T_{max})}\\
%\leq&\disp{\|A^\gamma
%e^{-tA}u_0\|_{L^2(\Omega)}+C_1|\kappa|\int_{0}^t(t-\tau)^{-\gamma}e^{-\mu(t-\tau)}\|(Yu(\cdot,\tau) \cdot \nabla)u(\cdot,\tau)\|_{L^2(\Omega)}d\tau}\\
% &+\disp{
% C_1\int_{0}^t(t-\tau)^{-\gamma}e^{-\mu(t-\tau)}[\|n(\cdot,\tau)\nabla\phi\|_{L^2(\Omega)}+\|g(\cdot,\tau)\|_{L^2(\Omega)}]d\tau}\\
\end{array}
\label{zjccffgbhjcvvvbscz2.5297x96301ku}
\end{equation}
and
\begin{equation}
\begin{array}{rl}
&\| m  (\cdot, t)\|_{W^{1,\infty}(\Omega)}\\
\leq&\disp{C_{24}\|(-\Delta+1)^\theta m  (\cdot, t)\|_{L^{4}(\Omega)}}\\
\leq&\disp{C_{25}t^{-\theta}e^{-\lambda_1 t}\|m_0\|_{L^{4}(\Omega)}+C_{25}\int_{0}^t(t-s)^{-\theta}e^{-\lambda_1(t-s)}
\|(m  -m  n  -u   \cdot \nabla m  )(\cdot,s)\|_{L^{4}(\Omega)}ds}\\
\leq&\disp{C_{26}+C_{26}\int_{0}^t(t-s)^{-\theta}e^{-\mu(t-s)}[\|m  (\cdot,s)\|_{L^{\infty}(\Omega)}+\|n  (\cdot,s)\|_{L^{4}(\Omega)}+\|u  (\cdot,s)\|_{L^\infty(\Omega)}
\|\nabla m  (\cdot,s)\|_{L^{4}(\Omega)}]ds}\\
\leq&\disp{C_{27}~~ \mbox{for all}~~ t\in(\tau,T_{max})}\\
%\leq&\disp{\|A^\gamma
%e^{-tA}u_0\|_{L^2(\Omega)}+C_1|\kappa|\int_{0}^t(t-\tau)^{-\gamma}e^{-\mu(t-\tau)}\|(Yu(\cdot,\tau) \cdot \nabla)u(\cdot,\tau)\|_{L^2(\Omega)}d\tau}\\
% &+\disp{
% C_1\int_{0}^t(t-\tau)^{-\gamma}e^{-\mu(t-\tau)}[\|n(\cdot,\tau)\nabla\phi\|_{L^2(\Omega)}+\|g(\cdot,\tau)\|_{L^2(\Omega)}]d\tau}\\
\end{array}
\label{zjccffgbhjcvbscz97x96u}
\end{equation}
with $\tau\in(0,T_{max})$,
where  we have used \dref{ddfgczhhhh2.5ghju48cfg924ghyuji} as well as  the H\"{o}lder inequality and
$$\int_{0}^t(t-s)^{-\theta}e^{-\lambda_1(t-s)}\leq \int_{0}^{\infty}\sigma^{-\theta}e^{-\lambda_1\sigma}d\sigma<+\infty.$$
Finally,
 by virtue of Lemma \ref{lemma70} and
\dref{cz2.571hhhhh51ccvvhddfccvvhjjjkkhhggjjllll}, \dref{cz2.57ghhhh151ccvhhjjjkkkffgghhuuiivhccvvhjjjkkhhggjjllll}, \dref{cz2.5jkkcvvvhjdsdfffffdkfffffkhhgll}--\dref{zjccffgbhjcvbscz97x96u},  the local solution can be extend to the global-in-time solutions.
Employing almost exactly the same arguments as in the proof of Lemma 3.1 in  \cite{LiLiLiLisssdffssdddddddgssddsddfff00} (see also \cite{Zhenddsdddddgssddsddfff00} and \cite{Wanssssssg21215}), and taking advantage of \dref{1.163072xggttsdddyyu}, we conclude the regularity theories for the Stokes semigroup and the   H\"{o}lder estimate for local solutions of parabolic equations, we can
obtain weak solution $(n,c,m,u)$   is a classical solution.
\end{proof}

{\bf Data availability}

No data was used for the research described in the article.

{\bf Acknowledgments}:
This work is partially supported by  Shandong Provincial Natural Science Foundation
 (No. ZR2022JQ06 and No. ZR2021QA052) as well as the National Natural
Science Foundation of China (No. 11601215 and No.12101534).


\begin{thebibliography}{00}

\bibitem{Bellomo1216} N. Bellomo,  A. Belloquid,   Y. Tao, M. Winkler,  \textit{Toward a mathematical theory of
Keller--Segel models of pattern formation in biological tissues}, Math. Mod. Meth. Appl. Sci.,  25(2015), 1663--1763.

%\bibitem{Blackttyyy216} T. Black, \textit{Global solvability of chemotaxis-fluid systems with nonlinear diffusion and matrix-valued sensitivities in
%three dimensions}, Nonlinear Anal., 180(2019), 129--153.
%
%\bibitem{Cao}
%X.~Cao,
%\textit{Global classical solutions in chemotaxis(--Navier)--Stokes system with rotational flux term},
%J. Diff. Eqns., 261(2016), 6883--6914.
%
%
%\bibitem{CaoCaoLiitffg11}
%X.~Cao, J.~Lankeit,
%\textit{Global classical small--data solutions for a 3D chemotaxis Navier-Stokes system involving matrix--valued sensitivities},
%Calculus of Variations and Part. Diff. Eqns., 55(2016), 55--107.


\bibitem{BlackFFGG}  T. Black,  \textit{Global solvability of chemotaxis-fluid systems with nonlinear
diffusion and matrix-valued sensitivities in three dimensions}, Nonlinear Anal., 180(2019),  129--153.


\bibitem{Chaexdd12176}
M.~Chae, K.~Kang, J.~Lee,
\textit{Global Existence and temporal decay in Keller--Segel models coupled to fluid equations},
Comm. Part. Diff. Eqns., 39(2014), 1205--1235.

 \bibitem{Cie791} T. Cie\'{s}lak,  C. Stinner, \textit{Finite-time blowup and global-in-time unbounded solutions to a parabolic--parabolic quasilinear
Keller--Segel system in higher dimensions,} J. Diff. Eqns., 252(2012), 5832--5851.

 \bibitem{Cie201712791} T. Cie\'{s}lak,  C. Stinner, \textit{New critical exponents in a fully parabolic quasilinear Keller--Segel system and applications
to volume filling models}, J. Diff. Eqns., 258(2015), 2080--2113.


%\bibitem{Cie72}
%T.~Cie\'{s}lak, M.~Winkler,
%\textit{Finite--time blow-up in a quasilinear system of chemotaxis,}
%Nonlinearity, 21(2008), 1057--1076.

\bibitem{Duan12186}
R.~Duan, A.~Lorz, P.~A.~Markowich,
\textit{Global solutions to the coupled chemotaxis--fluid equations},
Comm. Part. Diff. Eqns., 35(2010), 1635--1673.

\bibitem{Espejoss12186}  E. E. Espejo, T. Suzuki, \textit{Reaction terms avoiding aggregation in slow fluids}, Nonlinear Anal. RWA., 21(2015), 110--126.
%


\bibitem{EspejojjEspejojainidd793} E. E. Espejo, M. Winkler,  \textit{Global classical solvability and stabilization in a two-dimensional
chemotaxis-Navier-Stokes system modeling coral fertilization}, Nonlinearity, 31(2018), 1227--1259.


%
%\bibitem{Francesco791}  M. Di Francesco, A. Lorz and P. Markowich, \textit{Chemotaxis-fluid coupled model for swimming
%bacteria with nonlinear diffusion: Global existence and asymptotic behavior}, Discrete Cont. Dyn. Syst., 28(2010), 1437--1453.
%
%
%
%%\bibitem{Duanx41215}R. Duan, Z. Xiang, \textit{A note on global existence for the chemotaxis--Stokes model with
%%nonlinear diffusion}, Int. Math. Res. Not. IMRN, (2014), 1833--1852.
%
%
%
%
%%\bibitem{Giga1215}  Y. Giga, \textit{Solutions for semilinear parabolic equations in $L^p$ and regularity of weak solutions of the
%%Navier-Stokes system}, J. Diff. Eqns., 61(1986), 186--212.
%%
%%
%%
%%\bibitem{Hillen} T. Hillen,   K. Painter, \textit{A user's guide to PDE models for chemotaxis,} J. Math. Biol., 58(2009), 183--217.
%%
%%\bibitem{Horstmann2710} D. Horstmann,  \textit{From $1970$ until present: The Keller--Segel model in chemotaxis and its consequences,} I.
%% Jahresberichte der Deutschen Mathematiker-Vereinigung, 105(2003), 103--165.

\bibitem{Fujiwara66612186}  D. Fujiwara, H. Morimoto,  \textit{An $L^r$-theorem of the Helmholtz decomposition of vector fields},
J. Fac. Sci. Univ. Tokyo, 24(1977), 685--700.

%\bibitem{Fujiefffmann2710} K. Fujie, T. Senba, \textit{Application of an Adams type inequality to a two-chemical substances chemotaxis system}, J.
%Differ. Equ., 263(2017), 88---148.

\bibitem{Herrero710} M. Herrero,   J.  Vel\'{a}zquez, \textit{A blow-up mechanism for a chemotaxis model,} Ann.
Scuola Norm. Super. Pisa Cl. Sci., 24(1997),    633--683.


\bibitem{Hieber}  M. Hieber, J. Pr\"{u}ss, \textit{Heat kernels and maximal $L^p$-$L^q$ estimate for parabolic evolution equations}, Comm. Part. Diff. Eqns., 22(1997), 1647--1669.

%
%\bibitem{Dombrowskio12186} Christopher Dombrowski, Luis Cisneros, Sunita Chatkaew, Raymond E. Goldstein, and John O. Kessler. Self-
%concentration and large-scale coherence in bacterial dynamics. Physical Review Letters, 93(9):098103, 2004.
%


%
%\bibitem{Duanx41215}
%R.~Duan, Z.~Xiang,
%\textit{A note on global existence for the chemotaxis--Stokes model with nonlinear diffusion},
%Int. Math. Res. Not. IMRN, (2014), 1833--1852.
%
%
%
%\bibitem{Francesco791}
%M.~Di Francesco, A.~Lorz and P.~Markowich,
%\textit{Chemotaxis--fluid coupled model for swimming bacteria with nonlinear diffusion: Global existence and asymptotic behavior},
%Discrete Cont. Dyn. Syst., 28(2010), 1437--1453.
%
%
%\bibitem{Friedmanddfesco791}
%A.~Friedman,
%\textit{Part. Diff. Eqns.}, Holt, Rinehart and Winston, New York, 1969.
%
%
%\bibitem{Giga1215}
%Y.~Giga,
%\textit{Solutions for semilinear parabolic equations in $L^p$ and regularity of weak solutions of the Navier-Stokes system},
%J. Diff. Eqns., 61(1986), 186--212.
%
%
%\bibitem{Herrero22710}
%M.~Herrero, J.~Vel\'{a}zquez,
%\textit{A blow--up mechanism for a chemotaxis model,}
%Ann. Scuola Norm. Super. Pisa Cl. Sci., 24(1997), 633--683.
%
%
%
%\bibitem{Hillen}
%T.~Hillen, K.~Painter,
%\textit{A user's guide to PDE models for chemotaxis,}
%J. Math. Biol., 58(2009), 183--217.
%


\bibitem{Horstmann791}
D.~Horstmann, M.~Winkler,
\textit{Boundedness vs. blow--up in a chemotaxis system},
J. Diff. Eqns., 215(2005), 52--107.


\bibitem{Ishida}
S.~Ishida, K.~Seki, T.~Yokota,
\textit{Boundedness in quasilinear Keller--Segel systems of parabolic--parabolic type on non--convex bounded domains},
J. Diff. Eqns., 256(2014), 2993--3010.

%

\bibitem{Kegssddsddfff00}
Y.~Ke, J.~Zheng,
\textit{An optimal result for global existence  in a three-dimensional Keller--Segel--Navier-Stokes system involving tensor--valued sensitivity with saturation},
Calc. Var.  Part. Diff. Eqns., 58(2019), 58--109.

\bibitem{Keller79}
E.~Keller, L.~Segel,
\textit{Initiation of slime mold aggregation viewed as an instability, }
J. Theor. Biol., 26(1970), 399--415.

\bibitem{Kiselevsssdd793}A. Kiselev, L. Ryzhik, \textit{Biomixing by chemotaxis and efficiency of biological reactions: the
critical reaction case}, J. Math. Phys., 53(2012), 115609.


\bibitem{Kiselevdd793} A. Kiselev, L. Ryzhik,   \textit{Biomixing by chemotaxis and enhancement of biological reactions},
Comm. Part. Diff. Eqns., 37(2012), 298--318.


%\bibitem{Kowalczyk7101}
%R.~Kowalczyk,
%\textit{Preventing blow--up in a chemotaxis model},
%J. Math. Anal. Appl., 305(2005), 566--585.
%
%
%\bibitem{Ladyzenskaja710}O. A. Ladyzenskaja, V. A. Solonnikov,   N. N. Ural'eva, \textit{Linear and Quasi-linear
%Equations of Parabolic Type,} Amer. Math. Soc. Transl. 23, AMS, Providence, RI, 1968.
%

\bibitem{Lankeitffg11}
J.~Lankeit,
\textit{Long--term behaviour in a chemotaxis--fluid system with logistic source},
Math. Mod. Meth. Appl. Sci., 26(2016), 2071--2109.

\bibitem{Zhenddsdddddgssddsddfff00}K. Li,  J. Zheng, \textit{An optimal result for global classical and bounded solutions in a two-dimensional Keller-Segel-Navier-Stokes system with sensitivity},
 Commu. Pure  Appl. Anal.,  21(2022), 4147--4172.

%\bibitem{Lankeitjkkkllffg11}
%J.~Lankeit,
%\textit{Locally bounded global solutions to a chemotaxis consumption model with singular sensitivity and nonlinear
%diffusion}, J. Diff. Eqns., 262(2017), 4052--4084.


 \bibitem{Lidfff00}   J. Li,  P. Pang,  Y. Wang, \textit{Global boundedness and decay property of a three-dimensional Keller--Segel--Stokes system modeling coral fertilization}, Nonlinearity,  32(2019), 2815--2847.


\bibitem{LiLiLiLisssdffssdddddddgssddsddfff00} T. Li, A. Suen, M. Winkler, C. Xue, \textit{Global small-data solutions of a two-dimensional chemotaxis system with rotational flux term}, Math. Mod. Meth. Appl. Sci., 25(2015), 721--746.

%\bibitem{LiLiLiLisssdffssdddddddgssddsddfff00}
%T.~Li, A.~Suen, C.~Xue, M.~Winkler,
%\textit{Global small-data solutions of a two-dimensional chemotaxis system with rotational flux term},
%Math. Mod. Meth. Appl. Sci., 25 (2015), 721--746.





\bibitem{LiuansdddWang}
J.~Liu, \textit{Boundedness in a chemotaxis-(Navier-) Stokes system modeling coral fertilization with slow p-Laplacian diffusion}, J. Math. Fluid Mech.,
 22(2020), 10.

\bibitem{Liua22nsdddWang}
J.~Liu, \textit{Global weak solutions in a three-dimensional degenerate chemotaxis-Navier-
Stokes system modeling coral fertilization}, Nonlinearity, 33(2020), 3237--3297.

\bibitem{Liuans333dddWang}
J.~Liu, \textit{Large time behavior in a three-dimensional degenerate chemotaxis-Stokes
system modeling coral fertilization}, J. Diff. Eqns., 269(2020), 1--55.

%
%\bibitem{LiuandWang}
%J.~Liu, Y.~Wang,
%\textit{Global existence and boundedness in a Keller--Segel--(Navier--)Stokes system with signaldependent sensitivity},
%J. Math. Anal. Appl., 447(2017), 499--528.

\bibitem{Liucvb12176}
J.~Liu, A.~Lorz,
\textit{A coupled chemotaxis--fluid model: Global existence},
Ann. Inst. H. Poincar\'{e} Anal. Non Lin\'{e}aire, 28(2011), 643--652.


\bibitem{Lidddfffjkkkkucvb12176}
J.~Liu, Y.~Wang,
\textit{Boundedness and decay property in a three--dimensional Keller--Segel--Stokes system involving tensor--valued sensitivity with saturation},
J. Diff. Eqns., 261(2016), 967--999.

\bibitem{LiuZhLiuLiuandddgddff4556} J. Liu, Y. Wang, \textit{Global weak solutions in a three-dimensional Keller--Segel--Navier-Stokes system involving a
tensor-valued sensitivity with saturation}, J. Diff. Eqns., 262(2017), 5271--5305.



\bibitem{Lorz79}
A.~Lorz,
\textit{Global solutions to the coupled chemotaxis-fluid equations},
Math. Mod. Meth. Appl. Sci., 20(2010), 987--1004.

\bibitem{Nagaixcdf791} T. Nagai, T. Senba, K. Yoshida,    \textit{Application of the Trudinger-Moser inequality to a parabolic system of
chemotaxis}, Funkc. Ekvacioj, Ser. Int., 40(1997), 411--433.


%\bibitem{Miyakawakk79} T. Miyakawa, H. Sohr, \textit{On energy inequality, smoothness and large time behaviour in L2 for weak solutions of the Navier-Stokes equations},
%Math. Z., 199(1988), 465--478.
%
%\bibitem{Mizoguchik79} N. Mizoguchi and P. Souplet, \textit{Nondegeneracy of blow-up points for the parabolic Keller-Segel system}, Ann. Inst. H.
%Poincar\'{e} Anal. Non Lin\'{e}aire, 31(2014), 851--875.
%
%\bibitem{Simon}
%J.~Simon,
%\textit{Compact sets in the space $L^{p}(O, T;B)$},
%Annali di Matematica Pura ed Applicata, 146(1986), 65--96.
%

\bibitem{Sohr}
H.~Sohr,
\textit{The Navier-Stokes equations, An elementary functional analytic approach},
Birkh\"{a}user Verlag, Basel, (2001).

%\bibitem{Tao72} Y. Tao, M. Winkler,  \textit{A chemotaxis--haptotaxis model: the roles of porous medium
%diffusion and logistic source}, SIAM J. Math. Anal., 43(2011), 685--704.
%

\bibitem{Tao794}
Y.~Tao, M.~Winkler,
\textit{Boundedness in a quasilinear parabolic--parabolic Keller--Segel system with subcritical sensitivity},
J. Diff. Eqns., 252(2012), 692--715.

%
%\bibitem{Taohhjj794}
%Y.~Tao, M.~Winkler,
%\textit{Global existence and boundedness in a Keller--Segel--Stokes model with arbitrary porous medium diffusion},
%Discrete Cont. Dyn. Syst. A, 32(2012), 1901--1914.

\bibitem{Tao71215}
Y.~Tao, M.~Winkler,
\textit{Locally bounded global solutions in a three--dimensional chemotaxis--Stokes system with nonlinear diffusion},
Ann. Inst. H. Poincar\'{e} Anal. Non Lin\'{e}aire, 30(2013), 157--178.

\bibitem{Tao41215}  Y. Tao, M. Winkler, \textit{Boundedness and decay enforced by quadratic degradation in a three-dimensional
chemotaxis--fluid system}, Z. Angew. Math. Phys., 66(2015), 2555--2573.

\bibitem{Tao34555571215}
Y.~Tao, M.~Winkler,
\textit{Effects of signal-dependent motilities in a Keller-Segel-type
reaction-diffusion system}, Math. Mod. Meth. Appl. Sci., 27(2017), 1645--1683.

%\bibitem{Tello710}
%J.~I.~Tello, M.~Winkler,
%\textit{A chemotaxis system with logistic source},
%Comm. Part. Diff. Eqns., 32(2007), 849--877.

\bibitem{Tuval1215}
I.~Tuval, L.~Cisneros, C.~Dombrowski, et al.,
\textit{Bacterial swimming and oxygen transport near contact lines},
Proc. Natl. Acad. Sci. USA, 102(2005), 2277--2282.

%\bibitem{Wang76}
%L.~Wang, C.~Mu, P.~Zheng,
%\textit{On a quasilinear parabolic--elliptic chemotaxis system with logistic source},
%J. Diff. Eqns., 256(2014), 1847--1872.
%
%

\bibitem{Wangjjk5566ddfggghjjkk1}
W.~Wang,
\textit{Global boundedness of weak solutions for a three--dimensional chemotaxis--Stokes system with nonlinear diffusion and rotation},
J. Diff. Eqns., 268(2020), 7047--7091.

\bibitem{Wanssssssg21215} W. Wang, M. Zhang, S. Zheng, \textit{To what extent is cross-diffusion controllable
in a two-dimensional chemotaxis-(Navier-)Stokes system modeling coral fertilization},  Calc. Var. Part. Diff. Eqns., 60(2021),
29.

\bibitem{Wangssddss21215}  Y. Wang,  \textit{Global weak solutions in a three-dimensional
Keller--Segel--Navier-Stokes system
with subcritical sensitivity},  Math. Mod. Meth. Appl. Sci., 27(2017), 2745--2780.

%\bibitem{Wangjjk5566ddfggghjjkk1}
%Y.~Wang,
%\textit{Global weak solutions in a three-dimensional Keller-Segel-Navier-Stokes system
%with subcritical sensitivity}, Math. Mod. Meth. Appl. Sci., 27(2017), 2745--2780.

\bibitem{Wangssddff21215}
Y.~Wang, X.~Li,
\textit{Boundedness for a 3D chemotaxis--Stokes system with porous medium diffusion and tensor--valued chemotactic sensitivity},
Z. Angew. Math. Phys., 68(2017), Art. 29, 23 pp.


 \bibitem{Wajjjjngssddff21215}Y. Wang, M. Winkler,  Z. Xiang, \textit{Global solvability in a three-dimensional Keller-
Segel-Stokes system involving arbitrary superlinear logistic degradation}, Adv. Nonlin. Anal., 10(2021), 707--731.

  \bibitem{Wang21215}  Y. Wang, Z. Xiang, \textit{Global existence and boundedness in a Keller--Segel--Stokes system involving a
tensor-valued sensitivity with saturation}, J. Diff. Eqns., 259(2015), 7578--7609.




\bibitem{Wangss21215}  Y. Wang, Z. Xiang, \textit{Global existence and boundedness in a Keller-Segel-Stokes system involving a tensor-valued
sensitivity with saturation: the 3D case}, J. Diff. Eqns., 261(2016), 4944--4973.

%
%\bibitem{Winkler21215}
%M.~Winkler,
%\textit{Boundedness in the higher--dimensional parabolic--parabolic chemotaxis system with logistic source},
%Comm.  Part. Diff. Eqns., 35(2010), 1516--1537.
%


\bibitem{Winkler79}
M.~Winkler,
\textit{Does a volume--filling effect always prevent chemotactic collapse},
Math. Meth. Appl. Sci., 33(2010), 12--24.


 \bibitem{Winkler792} M. Winkler, \textit{Aggregation vs. global diffusive behavior in the higher-dimensional Keller--Segel model}, J. Diff. Eqns., 248(2010), 2889--2905.

%

\bibitem{Winkler31215}
M.~Winkler,
\textit{Global large--data solutions in a chemotaxis--(Navier--)Stokes system modeling cellular swimming in fluid drops},
Comm. Part. Diff. Eqns., 37(2012), 319--351.


\bibitem{Winkler793}
M.~Winkler,
\textit{Finite--time blow-up in the higher--dimensional parabolic--parabolic Keller--Segel system},
J. Math. Pures Appl., 100(2013), 748--767.


\bibitem{Winkler61215}
M.~Winkler,
\textit{Stabilization in a two--dimensional chemotaxis-Navier-Stokes system},
Arch. Ration. Mech. Anal., 211(2014), 455--487.


\bibitem{Winkler11215}
M.~Winkler,
\textit{Boundedness and large time behavior in a three--dimensional chemotaxis--Stokes system with nonlinear diffusion and general sensitivity},
Calc. Var. Part. Diff. Eqns., 54(2015), 3789--3828.

%\bibitem{Winksssdjjkkler51215}
%M.~Winkler,
%\textit{Large--data global generalized solutions in a chemotaxis system with tensor--valued sensitivities},
%SIAM J. Math. Anal., 47(2015), 3092--3115.

\bibitem{Winkler51215}
M.~Winkler,
\textit{Global weak solutions in a three--dimensional chemotaxis-Navier-Stokes system},
Ann. Inst. H. Poincar\'{e} Anal. Non Lin\'{e}aire, 33(2016), 1329--1352.




\bibitem{Winklerssdff11215}
M.~Winkler,
\textit{How far do chemotaxis--driven forces influence regularity in the Navier-Stokes system}?
Trans. Am. Math. Soc., 369(2017), 3067--3125.

\bibitem{Winklerssscvb12176}
M.~Winkler,
\textit{Global existence and stabilization in a degenerate chemotaxis--Stokes system with mildly strong diffusion enhancement},
J. Diff. Eqns., 264(2018), 6109--6151.

%\bibitem{Winklerglobalmass}
%M.~Winkler,
%\textit{Global mass--preserving solutions in a two--dimensional chemotaxis--Stokes system with rotation flux components},
%J. Evol. Eqns., 18(2018), 1267--1289.


      \bibitem{Winkler444sdddssdff51215}  M. Winkler, \textit{A three-dimensional Keller-Segel-Navier-Stokes system with logistic source: global weak solutions and asymptotic stabilization}, J.  Func. Anal., 276(2019), 1339--1401.

         \bibitem{Winklessddffr444sdddssdff51215}  M. Winkler, \textit{Small-mass solutions in the two-dimensional Keller-Segel system coupled to the Navier-Stokes equations},  SIAM J. Math. Anal., 52(2020), 2041--2080.
%
%\bibitem{Winklerwweerr444ssdff51215}  M. Winkler, \textit{Boundedness in a three-dimensional Keller-Segel-Stokes system with subcritical sensitivity},
% Appl. Math. Lett., 112(2021),  106785.
%
% \bibitem{Xiewweerr444ssdff51215} L.  Xie, \textit{Global classical solutions in a Keller-Segel(-Navier)-Stokes system modeling coral fertilization},
%J. Diff. Eqns., 267(2019), 6290--6315.
%
% \bibitem{Zhengweerr444ssdff51215} J. Zheng, \textit{A new result for the global existence (and boundedness) and regularity of a three-dimensional
%Keller-Segel-Navier-Stokes system modeling coral fertilization},  J. Diff. Eqns., 272(2021), 164--202.
%
%

%

\bibitem{Winkleghhhr51215}
M.~Winkler,
\textit{Can rotational fluxes impede the tendency toward spatial homogeneity in nutrient taxis(--Stokes) systems}?
Int. Math. Res. Not. IMRN, (2021), 8106--8152.



%
%\bibitem{Zhangddffcvb12176}
%Q.~Zhang, Y.~Li,
%\textit{Global weak solutions for the three--dimensional chemotaxis-Navier-Stokes system with nonlinear diffusion},
%J. Diff. Eqns., 259(2015), 3730--3754.


\bibitem{Winkler72} M. Winkler, K. C. Djie, \textit{Boundedness and finite-time
collapse in a  chemotaxis system with volume-filling effect}, Nonlinear Anal. TMA.,  72(2010),  1044--1064.

\bibitem{Xieghhhr51215} L. Xie, \textit{Global classical solutions in a Keller-Segel(-Navier)-Stokes system modeling
coral fertilization}, J. Diff. Eqns., 267(2019), 6290--6315.

\bibitem{Xusddeddff345511215} C. Xue, H. G. Othmer, \textit{Multiscale models of taxis-driven patterning in bacterial populations}, SIAM
J. Appl. Math., 70(2009), 133--167.


\bibitem{Zhangcvb12176}
Q.~Zhang, X.~Zheng,
\textit{Global well--posedness for the two--dimensional incompressible chemotaxis-Navier-Stokes equations},
SIAM J. Math. Anal., 46(2014), 3078--3105.


%

\bibitem{Zheng00}
J.~Zheng,
\textit{Boundedness of solutions to a quasilinear parabolic--elliptic Keller--Segel system with logistic source},
J. Diff. Eqns., 259(2015), 120--140.

\bibitem{Zhengsdsd6}
J.~Zheng,
\textit{Boundedness in a three-dimensional chemotaxis--fluid system involving tensor-valued sensitivity with saturation},
J. Math. Anal. Appl., 442(2016), 353--375.

\bibitem{Zhengssdefr23}
J.~Zheng,
\textit{A note on boundedness of solutions to a higher--dimensional quasi--linear chemotaxis system with logistic source},
Z. Angew. Math. Phys., 97(2017),  414--421.

%
%
%
%\bibitem{Zhengsddfffsdddssddddkkllssssssssdefr23}
%J.~Zheng,
%\textit{Global weak solutions in a three--dimensional Keller--Segel--Navier-Stokes system with nonlinear diffusion},
%J. Diff. Eqns., 263(2017), 2606--2629.

\bibitem{Zhenddddgssddsddfff00}
J.~Zheng,
\textit{An optimal result for global existence and boundedness in a three--dimensional Keller--Segel--Stokes system with nonlinear diffusion},
J. Diff. Eqns., 267(2019),  2385--2415.

%\bibitem{Zhenyyyuuddsdddddgssddsddfff00}
%J.~Zheng, Y.~Ke,
%\textit{Blow--up prevention by nonlinear diffusion in a 2D Keller--Segel--Navier-Stokes system with rotational flux},
%J. Diff. Eqns., 268(2020), 7092--7120.
%




%\bibitem{Zhengddffsdsd6}
%J.~Zheng, Y.~Wang,
%\textit{A note on global existence to a higher--dimensional quasilinear chemotaxis system with consumption of chemoattractant},
%Discrete Cont. Dyn. Syst. B, 22(2017), 669--686.
%

\bibitem{Zhengssssddd0}  J.~Zheng,
\textit{A new result for the global existence (and boundedness) and regularity of a three-dimensional Keller-Segel-Navier-Stokes system modeling coral fertilization},
\newblock {J. Diff. Eqns.}, {272}(2021),  164--202.

%
%

\bibitem{Zhekkllndsssdddgssddsddfff00}
J.~Zheng,
\textit{Global existence and boundedness in a three--dimensional chemotaxis--Stokes system with nonlinear diffusion and general sensitivity},
Annali di Matematica Pura ed Applicata, { 201(2022),  243--288}.

\bibitem{Zhengsssssssddd0}  J.~Zheng,
\textit{Eventual smoothness and stabilization in a three-dimensional Keller-Segel-Navier-Stokes system with rotational flux},
{Calc. Var. Part. Diff. Eqns.},  61(2022),  52.

\bibitem{Zhengddfggghjjkk1}
J.~Zheng, Y.~Ke,
\textit{Large time behavior of solutions to a fully parabolic chemotaxis--haptotaxis model in $N$ dimensions},
J. Diff. Eqns., 266(2019), 1969--2018.





\bibitem{Zheklllkkkkllllndsssdddgssddsddfff00}
J.~Zheng, Y.~Ke,
\textit{Global bounded weak solutions for a chemotaxis--Stokes system with nonlinear diffusion and rotation},
J. Diff. Eqns.,  289(2021), 182--235.

%
%\bibitem{CaoCaoLiitffg11} X. Cao, J. Lankeit, \textit{Global classical small-data solutions for a 3D chemotaxis Navier-Stokes
%system involving matrix-valued sensitivities}, Calculus of Variations and Part. Diff. Eqns.,  55(2016), 55--107.
%
%
%
%
%\bibitem{Chaexdd12176} M. Chae, K. Kang, J. Lee, \textit{Global Existence and temporal decay in Keller--Segel models
%coupled to
%fluid equations}, Comm. Part. Diff. Eqns., 39(2014), 1205--1235.
%
%\bibitem{Cie72}T. Cie\'{s}lak, M. Winkler,  \textit{Finite-time blow-up in a quasilinear system of chemotaxis,}
%Nonlinearity, 21(2008), 1057--1076.
%
%\bibitem{Duan12186} R. Duan, A. Lorz, P. A. Markowich, \textit{Global solutions to the coupled chemotaxis-
%fluid
%equations}, Comm. Part. Diff. Eqns., 35(2010), 1635--1673.
%
%
%
%\bibitem{Duanx41215}R. Duan, Z. Xiang, \textit{A note on global existence for the chemotaxis--Stokes model with
%nonlinear diffusion}, Int. Math. Res. Not. IMRN, (2014), 1833--1852.
%
%
%
%
%\bibitem{Francesco791}  M. Di Francesco, A. Lorz and P. Markowich, \textit{Chemotaxis-fluid coupled model for swimming
%bacteria with nonlinear diffusion: Global existence and asymptotic behavior}, Discrete Cont.
%Dyn. Syst., 28(2010), 1437--1453.
%
%
%
%
%\bibitem{Giga1215} Y. Giga, \textit{Solutions for semilinear parabolic equations in $L^p$ and regularity of weak solutions of the
%Navier-Stokes system}, J. Diff. Eqns., 61(1986), 186--212.
%
%
%
%
%%
%\bibitem{Gigaasss1215} Y. Giga, \textit{The Stokes operator in  $L^r$ spaces}, Proc. Japan Acad. S., 2(1981), 85--89.
%%
%
%\bibitem{Herrero22710} M. Herrero,   J.  Vel\'{a}zquez, \textit{A blow-up mechanism for a chemotaxis model,} Ann.
%Scuola Norm. Super. Pisa Cl. Sci., 24(1997),    633--683.
%
%
%\bibitem{Hillen}T. Hillen,   K. Painter, \textit{A user's guide to PDE models for chemotaxis,} J. Math. Biol., 58(2009), 183--217.
%
%
%
%
%
%\bibitem{Horstmann791} D. Horstmann, M. Winkler, \textit{Boundedness vs. blow-up in a chemotaxis system}, J. Diff. Eqns, 215(2005), 52--107.
%
%
%\bibitem{Ishida}S. Ishida, K. Seki,  T, Yokota, \textit{Boundedness in quasilinear Keller--Segel systems of parabolic--parabolic type on
%non-convex bounded domains}, J. Diff. Eqns., 256(2014), 2993--3010.
%
%%
%%\bibitem{Ishida1215}  S. Ishida, \textit{Global existence and boundedness for chemotaxis-Navier-Stokes system with position-dependent sensitivity in $2d$ bounded domains}, Discrete Contin. Dyn. Syst. Ser. A, 32(2015), 3463-3482.
%%
%    \bibitem{Kegssddsddfff00} Y. Ke, J. Zheng, \textit{An optimal result for global existence  in a three-dimensional Keller--Segel--Navier-Stokes system
%involving tensor-valued sensitivity with saturation}, Calculus of Variations and Part. Diff. Eqns.,   58(2019), 58--109.
%
%    \bibitem{Keller79} E. Keller, L. Segel, \textit{Initiation of slime mold aggregation viewed as an instability, }  J. Theor. Biol., 26(1970), 399--415.
%
%\bibitem{Kowalczyk7101}R. Kowalczyk, \textit{Preventing blow-up in a chemotaxis model}, J. Math. Anal. Appl., 305(2005),
%   566--585.
%
%
%   \bibitem{Lankeitffg11} J. Lankeit,
%\textit{Long-term behaviour in a chemotaxis-fluid
%system with logistic source}, Math. Mod. Meth. Appl. Sci., 26(2016), 2071--2109.
%
%
%%
%\bibitem{LiLiLiLisssdffssdddddddgssddsddfff00}T. Li, A. Suen, C. Xue, M. Winkler, \textit{Global small-data solutions of a two-dimensional
%chemotaxis system with rotational
%flux term},  Math. Mod. Meth. Appl. Sci., 25 (2015),
%721--746.
%
%
%
%\bibitem{Liucvb12176}J. Liu, A. Lorz, \textit{A coupled chemotaxis-fluid model: Global existence}, Ann. Inst. H. Poincar\'{e} Anal. Non Lin\'{e}aire, 28(2011), 643--652.
%
%    \bibitem{Lorz79}  A. Lorz, \textit{
%Global solutions to the coupled chemotaxis-fluid equations}, Math. Mod. Meth. Appl. Sci., 20(2010), 987--1004.
%
%    \bibitem{Osaki710} K. Osaki, T. Tsujikawa, A. Yagi,   M. Mimura, \textit{Exponential attractor for a chemotaxisgrowth
%system of equations}, Nonlinear Anal. TMA., 51(2002),  119--144.
%
%  \bibitem{Wangssddff21215} Y. Wang, X. Li, \textit{Boundedness for a 3D chemotaxis-Stokes system with porous medium diffusion and tensor-valued
%chemotactic sensitivity}, Z. Angew. Math. Phys., 68(2017), Art. 29, 23 pp.
%
%%
%
%\bibitem{Sohr} H. Sohr, \textit{The Navier-Stokes equations,
% An elementary functional analytic approach},
% Birkh\"{a}user Verlag, Basel, (2001).
%
%
%\bibitem{Tao794} Y. Tao, M. Winkler,  \textit{Boundedness in a quasilinear parabolic--parabolic Keller--Segel system with subcritical sensitivity}, J.
%Diff. Eqns., 252(2012), 692--715.
%
%
%
% \bibitem{Tao71215}  Y. Tao, M. Winkler, \textit{Locally bounded global solutions in a three-dimensional chemotaxis--Stokes system
%with nonlinear diffusion}, Ann. Inst. H. Poincar\'{e} Anal. Non Lin\'{e}aire, 30(2013), 157--178.
%
%
%\bibitem{Tello710} J. I. Tello,   M. Winkler, \textit{A chemotaxis system with logistic source}, Comm. Part.
%Diff. Eqns., 32(2007),    849--877.
%
%\bibitem{Tuval1215}  I. Tuval, L. Cisneros, C. Dombrowski, et al., \textit{Bacterial swimming and oxygen transport near contact
%lines}, Proc. Natl. Acad. Sci. USA, 102(2005), 2277--2282.
%
%\bibitem{Wang76}L. Wang, C. Mu, P. Zheng, \textit{On a quasilinear parabolic--elliptic chemotaxis system with logistic source},
%   J. Diff. Eqns., 256(2014), 1847--1872.
%
%
%
%\bibitem{Wang79k1}  L.  Wang, C.  Mu, S.  Zhou, \textit{Boundedness in a parabolic--parabolic chemotaxis system with nonlinear diffusion}, Z. Angew. Math. Phys.,
%65(2014), 1137--1152.
%
%   \bibitem{Wangdd79k1}  L.  Wang, C.  Mu,  K. Lin, J. Zhao,  \textit{Global existence to a higher-dimensional quasilinear chemotaxis system with consumption of chemoattractant}, Z. Angew. Math. Phys., 66(2015), 1--16.
%
%
%\bibitem{Wangjjk5566ddfggghjjkk1} W. Wang,  \textit{Global boundedness of weak solutions for a
%three-dimensional chemotaxis-Stokes system with
%nonlinear diffusion and rotation},
%J. Diff. Eqns., doi.org/10.1016/j.jde.2019.11.072.
%
%\bibitem{Wang11215} Y. Wang, X. Cao,  \textit{Global classical solutions of a $3d$ chemotaxis--Stokes system with rotation}, Discrete
%Contin. Dyn. Syst. Ser. B, 20(2015), 3235--3254.
%
%
%%
%  \bibitem{Wangssddff21215} Y. Wang, X. Li, \textit{Boundedness for a 3D chemotaxis-Stokes system with porous medium diffusion and tensor-valued
%chemotactic sensitivity}, Z. Angew. Math. Phys., 68(2017), Art. 29, 23 pp.
%
%



% \bibitem{Winkler21215}M. Winkler, \textit{Boundedness in the higher-dimensional parabolic--parabolic chemotaxis system with
%logistic source}, Comm.  Part. Diff. Eqns., 35(2010), 1516--1537.
%
%
% \bibitem{Winkler79} M. Winkler, \textit{Does a volume-filling effect always prevent chemotactic collapse}, Math. Methods Appl. Sci., 33(2010), 12--24.
%
%
%
%
% \bibitem{Winkler792} M. Winkler, \textit{Aggregation vs. global diffusive behavior in the higher-dimensional Keller--Segel model}, J. Diff.
%Eqns., 248(2010), 2889--2905.
%
%\bibitem{Winkler31215}  M. Winkler, \textit{Global large-data solutions in a chemotaxis--(Navier--)Stokes system modeling cellular swimming in
%fluid drops}, Comm. Part. Diff. Eqns., 37(2012), 319--351.
%
%
%
%\bibitem{Winkler793} M. Winkler, \textit{Finite-time blow-up in the higher-dimensional parabolic--parabolic Keller--Segel system}, J. Math. Pures
%Appl., 100(2013),  748--767.
%
%
%\bibitem{Winkler61215}  M. Winkler, \textit{Stabilization in a two-dimensional chemotaxis-Navier-Stokes system}, Arch. Ration. Mech.
%Anal., 211(2014), 455--487.
%
%
%\bibitem{Winkler11215}  M. Winkler, \textit{Boundedness and large time behavior in a three-dimensional chemotaxis--Stokes system
%with nonlinear diffusion and general sensitivity},
%Calculus of Variations and Part. Diff. Eqns., 54(2015),   3789--3828.
%
%\bibitem{Winksssdjjkkler51215}  M. Winkler, \textit{Large-data global generalized solutions in a chemotaxis system with tensor-valued sensitivities}, SIAM J. Math. Anal., 47(2015),
%3092--3115.
%
%\bibitem{Winkler51215}  M. Winkler, \textit{Global weak solutions in a three-dimensional chemotaxis-Navier-Stokes system}, Ann. Inst.
%H. Poincar\'{e} Anal. Non Lin\'{e}aire, 33(2016),  1329--1352.
%
%\bibitem{Winklerssdff11215}  M. Winkler, \textit{How far do chemotaxis-driven forces influence regularity in the Navier-Stokes system}?, Trans. Am.
%Math. Soc., 369(2017), 3067--3125.
%
%\bibitem{Winklerssscvb12176} M. Winkler, \textit{Global existence and stabilization in a degenerate chemotaxis-Stokes system with mildly strong diffusion enhancement}, J. Diff.
%Eqns., 264(2018), 6109--6151.
%
%
%
%
%
%
%
%
%\bibitem{Winkleghhhr51215} M. Winkler, \textit{Can rotational fluxes impede the tendency toward spatial homogeneity in nutrient taxis(-Stokes) systems}?, Int. Math. Res. Not., (2019), doi.org /10.1093 /imrn /rnz056.



% \bibitem{Xue1215} C. Xue, H. G. Othmer, \textit{Multiscale models of taxis-driven patterning in bacterial population}, SIAM J.
%Appl. Math., 70(2009), 133--167.
%
% \bibitem{Xuefff1215} C. Xue, \textit{Macroscopic equations for bacterial chemotaxis: integration of detailed biochemistry of cell signaling}, J.
%Math. Biol., 70(2015), 1--44.
%
%
%\bibitem{Zhangddffcvb12176}Q. Zhang, Y. Li,   \textit{Global weak solutions for the three-dimensional chemotaxis-Navier-Stokes system with nonlinear diffusion},
% J. Diff.
%Eqns., 259(2015), 3730--3754.
%
%
%\bibitem{Zhangcvb12176} Q. Zhang, X. Zheng, \textit{Global well-posedness for the two-dimensional incompressible chemotaxis-Navier-Stokes
%equations}, SIAM J. Math. Anal., 46(2014), 3078--3105.
%
%
%
%
%\bibitem{Zheng00} J. Zheng, \textit{Boundedness of solutions to a quasilinear parabolic--elliptic Keller--Segel system with logistic source}, J. Diff.
%Eqns., 259(2015), 120--140.
%
%
%
%\bibitem{Zhenssdssdddfffgghjjkk1} J. Zheng, \textit{Boundedness of solution of a  higher-dimensional parabolic--ODE--parabolic chemotaxis--haptotaxis model with generalized logistic source}, Nonlinearity, 30(2017), 1987--2009.
%
%
%
%\bibitem{Zhengssdefr23} J. Zheng, \textit{A note on boundedness of solutions to a higher-dimensional quasi-linear chemotaxis system with logistic source},
%Zeitsc.  Angew. Mathe. Mech., 97(2017),  414--421.
%
%
%
%\bibitem{Zhenddddgssddxxxxccsddfff00} J. Zheng, \textit{Global boundedness of weak solutions for a
%three-dimensional chemotaxis-Stokes system with
%nonlinear diffusion and rotation},   Preprint.
%
%\bibitem{Zhendllldjjjkkkddgssddxxxxccsddfff00} J. Zheng, \textit{Boundedness  and stabilization in a degenerate chemotaxis-Stokes system with mildly strong diffusion enhancement},   Preprint.
%
%\bibitem{Zhengsddfffsdddssddddkkllssssssssdefr23} J.~Zheng,
%\textit{Global weak solutions in a three-dimensional Keller-Segel-Navier-Stokes system with nonlinear diffusion},
%J. Diff. Eqns., 263(2017), 2606--2629.
%
%\bibitem{Zhenddddgssddsddfff00} J.~Zheng,
%\textit{An optimal result for global existence and boundedness in a three-dimensional Keller-Segel-Stokes system with nonlinear diffusion},
%J. Diff. Eqns., 267(2019),  2385--2415.
%
%\bibitem{Zhenyyyuuddsdddddgssddsddfff00} J.~Zheng, Y.~Ke,
%\textit{Blow-up prevention by nonlinear diffusion in a 2D Keller-Segel-Navier-Stokes system with rotational flux},
%J. Diff. Eqns., 268(2020), 7092--7120.
%
%
%    \bibitem{Zhengddfggghjjkk1}J. Zheng, Y. Ke,  \textit{Large time behavior of solutions to a fully parabolic chemotaxis--haptotaxis model in $N$ dimensions},
%J. Diff. Eqns., 266(2019), 1969--2018.
%
%
%\bibitem{Zhengddffsdsd6} J. Zheng,  Y. Wang, \textit{A note on global existence to a
%higher-dimensional quasilinear chemotaxis
%system with consumption of chemoattractant}, Discr. Cont. Dyn. Syst. B, 22(2017), 669-686.
%
%
%\bibitem{CaoCaoLiitffg11}
%X.~Cao, J.~Lankeit,
%\textit{Global classical small-data solutions for a 3D chemotaxis Navier-Stokes system involving matrix-valued sensitivities},
%Calc. Var. PDE., 55(2016), 55--107.
%
%
%\bibitem{Duanx41215}
%R.~Duan, Z.~Xiang,
%\textit{A note on global existence for the chemotaxis-Stokes model with nonlinear diffusion},
%Int. Math. Res. Not. IMRN, (7)(2014), 1833--1852.
%
%
%\bibitem{Francesco791}
%M.~Di Francesco, A.~Lorz and P.~Markowich,
%\textit{Chemotaxis-fluid coupled model for swimming bacteria with nonlinear diffusion: Global existence and asymptotic behavior},
%Discrete Cont. Dyn. Syst., 28(2010), 1437--1453.
%
%
%\bibitem{Giga1215}
%Y.~Giga,
%\textit{Solutions for semilinear parabolic equations in $L^p$ and regularity of weak solutions of the Navier-Stokes system},
%J. Diff. Eqns., 61(1986), 186--212.
%
%
%\bibitem{Ishida}
%S.~Ishida, K.~Seki, T,~Yokota,
%\textit{Boundedness in quasilinear Keller--Segel systems of parabolic--parabolic type on non-convex bounded domains},
%J. Diff. Eqns., 256(2014), 2993--3010.
%
%\bibitem{Kegssddsddfff00}
%Y.~Ke, J.~Zheng,
%\textit{An optimal result for global existence  in a three-dimensional Keller--Segel--Navier-Stokes system involving tensor-valued sensitivity with saturation},
%Calculus of Variations and Part. Diff. Eqns., 58(2019), 58--109.
%
%
%\bibitem{Ladyzenskaja710}
%O.~A.~Ladyzenskaja, V.~A.~Solonnikov, N.~N.~Ural'eva,
%\textit{Linear and Quasi-linear Equations of Parabolic Type,}
%Amer. Math. Soc. Transl. 23, AMS, Providence, RI, 1968.
%
%
%
%\bibitem{Lankeitffg11}
%J.~Lankeit,
%\textit{Long-term behaviour in a chemotaxis-fluid system with logistic source},
%Math. Mod. Meth. Appl. Sci., 26(2016), 2071--2109.
%


%
%
%\bibitem{Sohr}
%H.~Sohr,
%\textit{The Navier-Stokes equations, An elementary functional analytic approach},
%Birkh\"{a}user Verlag, Basel, (2001).
%
%\bibitem{Stinnerffgg} C. Stinner, C. Surulescu, M. Winkler, \textit{Global weak solutions in a PDE-ODE system
%modeling multiscale cancer cell invasion}, SIAM J. Math. Anal., 46(2014), 1969--2007.
%
%
%\bibitem{Tao794}
%Y.~Tao, M.~Winkler,
%\textit{Boundedness in a quasilinear parabolic--parabolic Keller--Segel system with subcritical sensitivity},
%J. Diff. Eqns., 252(2012), 692--715.
%
%\bibitem{Tao79ssdd4}
%Y.~Tao, M.~Winkler,
%\textit{Global existence and boundedness in a Keller-Segel-Stokes model with arbitrary porous medium diffusion},
%Discrete Contin. Dyn. Syst., 32(2012), 1901--1914.
%
%
%\bibitem{Tao71215}
%Y.~Tao, M.~Winkler,
%\textit{Locally bounded global solutions in a three-dimensional chemotaxis--Stokes system with nonlinear diffusion},
%Ann. Inst. H. Poincar\'{e} Anal. Non Lin\'{e}aire, 30(2013), 157--178.
%
%
%
%\bibitem{Wangjjk5566ddfggghjjkk1}
%W.~Wang,
%\textit{Global boundedness of weak solutions for a three-dimensional chemotaxis-Stokes system with nonlinear diffusion and rotation},
%J. Diff. Eqns.,  268(2020), 7047-7091
%
%
%
%\bibitem{Wangssddff21215}
%Y.~Wang, X.~Li,
%\textit{Boundedness for a 3D chemotaxis-Stokes system with porous medium diffusion and tensor-valued chemotactic sensitivity},
%Z. Angew. Math. Phys., 68(2017), Art. 29, 23 pp.
%
%
%
%\bibitem{Wangss21215}
%Y.~Wang, Z.~Xiang,
%\textit{Global existence and boundedness in a Keller-Segel-Stokes system involving a tensor-valued sensitivity with saturation: the 3D case},
%J. Diff. Eqns., 261(2016), 4944--4973.
%
%
%\bibitem{Wang23421215}
%Y.~Wang, M.~Winkler, Z.~Xiang,
%\textit{Global classical solutions in a two-dimensional chemotaxis-Navier-Stokes system with subcritical sensitivity},
%Annali della Scuola Normale Superiore di Pisa-Classe di Scienze, XVIII, (2018), 2036--2145.
%
%
%
%
%\bibitem{Winkler792}
%M.~Winkler,
%\textit{Aggregation vs. global diffusive behavior in the higher-dimensional Keller--Segel model},
%J. Diff. Eqns., 248(2010), 2889--2905.
%
%\bibitem{Winkler31215}
%M.~Winkler,
%\textit{Global large-data solutions in a chemotaxis-(Navier-)Stokes system modeling cellular swimming in fluid drops},
%Comm. Part. Diff. Eqns., 37(2012), 319--351.
%
%
%
%\bibitem{Winkler793}
%M.~Winkler,
%\textit{Finite-time blow-up in the higher-dimensional parabolic--parabolic Keller--Segel system},
%J. Math. Pures Appl., 100(2013), 748--767.
%
%
%
%
%\bibitem{Winkler11215}
%M.~Winkler,
%\textit{Boundedness and large time behavior in a three-dimensional chemotaxis--Stokes system with nonlinear diffusion and general sensitivity},
%Calculus of Variations and Part. Diff. Eqns., 54(2015), 3789--3828.
%
%\bibitem{Winksssdjjkkler51215}
%M.~Winkler,
%\textit{Large-data global generalized solutions in a chemotaxis system with tensor-valued sensitivities},
%SIAM J. Math. Anal., 47(2015), 3092--3115.
%
%
%\bibitem{Winklerssscvb12176}
%M.~Winkler,
%\textit{Global existence and stabilization in a degenerate chemotaxis-Stokes system with mildly strong diffusion enhancement},
%J. Diff. Eqns., 264(2018), 6109--6151.
%
%\bibitem{Winklersssjkkkllcvb12176}
%M.~Winkler,
%\textit{Global mass-preserving solutions in a two-dimensional chemotaxis-Stokes system with rotational flux components},
%J. Evol. Eqns., 18(2018), 1267--1289.
%
%\bibitem{Winklersssjkkkssddllcvb12176}
%M.~Winkler,
%\textit{Does Fluid Interaction Affect Regularity in the Three-Dimensional Keller-Segel System with Saturated Sensitivity?}
%Journal of Mathematical Fluid Mechanics, 20(2018), 1889--1909.
%
%
%\bibitem{Winkleghhhr51215}
%M.~Winkler,
%\textit{Can rotational fluxes impede the tendency toward spatial homogeneity in nutrient taxis(-Stokes) systems?}
%Int. Math. Res. Not., (2019), doi.org /10.1093 /imrn /rnz056.
%
%
%
%
%\bibitem{Zhangddffcvb12176}
%Q.~Zhang, Y.~Li,
%\textit{Global weak solutions for the three-dimensional chemotaxis-Navier-Stokes system with nonlinear diffusion},
%J. Diff. Eqns., 259(2015), 3730--3754.
%



%\bibitem{Zhenssdssdddfffgghjjkk1}
%J.~Zheng,
%\textit{Boundedness of solution of a  higher-dimensional parabolic--ODE--parabolic chemotaxis--haptotaxis model with generalized logistic source},
%Nonlinearity, 30(2017), 1987--2009.
%
%
%\bibitem{Zhengssddssdefr23} J.~Zheng, \textit{Global existence and boundedness  in a
%three-dimensional chemotaxis-Stokes system with
%nonlinear diffusion and general sensitivity},   Preprint.
%
%\bibitem{Zhenddddgssddsddfff00} J.~Zheng,
%\textit{An optimal result for global existence and boundedness in a three-dimensional Keller-Segel-Stokes system with nonlinear diffusion},
%J. Diff. Eqns., 267(2019),  2385--2415.
%
%
%\bibitem{Zhengddffsdsd6}
%J.~Zheng, Y.~Wang,
%\textit{A note on global existence to a higher-dimensional quasilinear chemotaxis system with consumption of chemoattractant},
%Discr. Cont. Dyn. Syst. B, 22(2017), 669-686.
%
%
%
%\bibitem{Bellomo1216} N. Bellomo,  A. Belloquid,   Y. Tao, M. Winkler,  \textit{Toward a mathematical theory of
%Keller--Segel models of pattern formation in biological tissues}, Math. Mod. Meth. Appl. Sci., 25(2015), 1663--1763.
%




%\bibitem{BlackssddFFGG}  T. Black,  \textit{Global very weak solutions to a chemotaxis-fluid system with nonlinear diffusion}, SIAM J. Math. Anal., 50(2018),
% 4087--4116.
%
% \bibitem{BlackFFGG} T. Black,  \textit{Global solvability of chemotaxis-fluid systems with nonlinear diffusion and matrix-valued sensitivities in three dimensions},  Nonlinear Anal.,  180(2019), 129--153.
%
%
%
%\bibitem{Cao22119} X. Cao, \textit{Global classical solutions in chemotaxis(--Navier)--Stokes system
%with rotational flux term},
%J. Diff. Eqns., 261(2016), 6883--6914.
%
%
%
%\bibitem{Chaexdd12176} M. Chae, K. Kang, J. Lee, \textit{Global existence and temporal decay in Keller--Segel models
%coupled to
%fluid equations}, Commun. Part. Diff. Eqns., 39(2014), 1205--1235.
%
%
%


%
%\bibitem{Dzai5667}F.  Dai, B.  Liu, \textit{Global weak solutions in a three-dimensional
%Keller-Segel-Navier-Stokes system with indirect signal
%production}, J.  Diff.  Eqns., 333(2022),  436--488.
%
%\bibitem{Djjzai5667}F.  Dai, B.  Liu, \textit{Global solvability and asymptotic stabilization
%in a three-dimensional Keller-Segel-Navier-Stokes
%system with indirect signal production}, Mathematical Models and Methods in Applied Sciences,  31(2021),  2091--2163.
%
%
%
%\bibitem{Ding12176} M. Ding, W. Wang, \textit{Global boundedness in a quasilinear fully parabolic chemotaxis system with indirect signal
%production}, Discrete Contin. Dyn. Syst., Ser. B.,  24(2019),   4665--4684.
%
%
%\bibitem{Dong5667} Y. Dong, Y. Peng, \textit{Global boundedness in the higher-dimensional chemotaxis system with indirect signal production
%and rotational flux}, Appl. Math. Lett., 112(2021),  106700.
%
%
%\bibitem{Duan12186} R. Duan, A. Lorz, P. A. Markowich, \textit{Global solutions to the coupled chemotaxis-fluid
%equations}, Commun. Part. Diff. Eqns., 35(2010), 1635--1673.






%
%
%\bibitem{Fujie201712791} K. Fujie, T. Senba,  \textit{Blowup of solutions to a two-chemical substances chemotaxis system in the critical dimension},
%J. Differ. Equ., 266(2019), 942--976.
%
%\bibitem{Haroske} D. D. Haroske, H. Triebel,
%\textit{Distributions, Sobolev Spaces, Elliptic Equations, European Mathematical Society,}
%Zurich, 2008.




%
%
%
%\bibitem{Hilleddffgggn} T. Hillen and A. Potapov, \textit{The one-dimensional chemotaxis model: Global existence and
%asymptotic profile}, Math. Methods Appl. Sci., 27(2004), 1783--1801.
%
%
%\bibitem{Horstmann791} D. Horstmann, M. Winkler, \textit{Boundedness vs. blow-up in a chemotaxis system}, J. Diff. Eqns., 215(2005), 52--107.
%
%\bibitem{HuHunn791} B. Hu and Y. Tao, \textit{To the exclusion of blow-up in a three-dimensional chemotaxis growth model with indirect attractant production}, Math. Mod. Meth. Appl. Sci.,
%26(2016), 2111--2128.
%
%
%
%%\bibitem{Ishida1215}  S. Ishida, \textit{Global existence and boundedness for chemotaxis-Navier-Stokes system with position-dependent sensitivity in $2d$ bounded domains}, Discrete Contin. Dyn. Syst. Ser. A, 32(2015), 3463--3482.
%
%
%\bibitem{Jger317} W. J\"{a}ger, S. Luckhaus,  \textit{On explosions of solutions to a system of Part. Diff. Eqns. modelling
%chemotaxis}, Trans. Am. Math. Soc., 329(1992), 819--824.


%\bibitem{Kiselklllevdd793} A. Kiselev and X. Xu,   \textit{Suppression of chemotactic explosion by mixing},  Arch. Ration. Mech. Anal., 222(2016), 1077--1112.
%
%
%
%     \bibitem{Zhengssdddd00}Y. Ke,  J. Zheng, \textit{An optimal result for global existence  in a three-dimensional Keller--Segel--Navier-Stokes system involving tensor-valued sensitivity with saturation}, Calculus of Variations and Part. Diff. Eqns.,  58(2019), 109.
%
%
%
%
%\bibitem{Keller2710}E. Keller, L. Segel,  \textit{Model for chemotaxis}, J. Theor. Biol., 30(1970),  225--234.
%
%
%   \bibitem{Lankeitffg11} J. Lankeit,
%\textit{Long-term behaviour in a chemotaxis-fluid
%system with logistic source}, Math. Mod. Meth. Appl. Sci., 26(2016), 2071--2109.
%
%
% \bibitem{Nagaiddtffg11} T. Nagai, \textit{Blow-up of radially symmetric solutions to a chemotaxis system}, Adv. Math. Sci.
%Appl., 5(1995), 581--601.
%
%
%
%     \bibitem{LiLiLi791} Y. Li, J. Lankeit, \textit{Boundedness in a chemotaxis-haptotaxis model with nonlinear diffusion}, Nonlinearity, 29(2016), 1564--1595.
%
%
%   \bibitem{LordLiLi791} A. Lorz, \textit{Coupled Keller-Segel-Stokes model: global existence for small initial data and blow-up delay}, Commun. Math.
%Sci., 10(2012), 555--574.
%
% \bibitem{Li1216} F. Li,  Y. Li, \textit{Global existence of weak solution in a chemotaxis-fluid system with nonlinear diffusion and rotational flux},   Discrete Contin. Dyn. Syst., 24(2019), 5409--5436.
%
%
%%\bibitem{LiTaoWinkler}  G. Li, Y. Tao, M. Winkler, \textit{Large time behavior in a predator-prey system with indirect pursuit-evasion interaction},  Disc. Cont. Dyna. Syst. B, 25(2020), 4383.
%
%
%
%
%
%     \bibitem{LiTaoWinkler}  G. Li, Y. Tao, M. Winkler, \textit{Large time behavior in a predator-prey system with indirect pursuit-evasion interaction},  Disc. Cont. Dyna. Syst. B., 25(2020), 4383.
%
%
%
%
%%\bibitem{LiWang11} W. Li, Z. Wang, \textit{Traveling fronts in diffusive and cooperative Lotka-Volterra system
%%with non-local delays}, Z. Angew. Math. Phys., 58(2007), 571--591.
%
%
%
%
%
%\bibitem{Liu1215} J. Liu, A. Lorz, \textit{A coupled chemotaxis--fluid model: Gobal existence}, Ann. Inst. H.
%Poincar\'{e} Anal. Non Lin\'{e}aire, 28(2011), 643--652.
%
%
%
%
%\bibitem{LiuZhLiuLiuandddgddff4556} J. Liu, Y. Wang, \textit{Global weak solutions in a three-dimensional Keller--Segel--Navier-Stokes system involving a
%tensor-valued sensitivity with saturation}, J. Diff. Eqns., 262(2017), 5271--5305.
%
%
%\bibitem{LiuZheng} X. Liu, J. Zheng, \textit{Convergence rates of solutions in a predator-preysystem with indirect pursuit-evasion interaction in domains of arbitrary dimension}, Disc. Cont. Dyna. Syst. B., 28(2023),  2269--2293.



%
%%\bibitem{Lorz1215}  A. Lorz, \textit{Coupled chemotaxis fluid equations}, Math. Mod. Meth. Appl. Sci., 20(2010), 987--1004.
%%
%%
%%
%%
%%
%%
%%\bibitem{Painter55677}  K. Painter, T. Hillen, \textit{Volume-filling and quorum-sensing in models for chemosensitive movement}, Can. Appl. Math.
%%Q., 10(2002), 501--543.
%
%
%
%
%
% \bibitem{Osaki79} K. Osaki,   A. Yagi,  \textit{Finite dimensional attractors for one-dimensional Keller--Segel equations}, Funkcial. Ekvac., 44(2001), 441--469.
%
% \bibitem{Osakihhjj79} K. Osaki, Tsujikawa, T., Yagi, A., Mimura, M.:  \textit{Exponential attractor for a chemotaxis-growth system
%of equations},  Nonlinear Anal. 51, 119-144 (2002)
%
%\bibitem{Peng55667}Y. Peng, Z. Xiang, \textit{Global existence and boundedness in a 3D Keller--Segel--Stokes system with nonlinear
%diffusion and rotational flux}, Z. Angew. Math. Phys., (2017), 68:68.
%
%
%
%\bibitem{Simon} J. Simon, \textit{Compact sets in the space $L^{p}(O, T;B)$}, Annali di Matematica Pura ed Applicata, 146(1986), 65--96.
%
%
%
%\bibitem{Sohr} H. Sohr, \textit{The Navier-Stokes Equations: An Elementary Functional Analytic Approach},
% Birkh\"{a}user Verlag, Basel, (2001).
%
%
%\bibitem{Tao794} Y. Tao, M. Winkler,  \textit{Boundedness in a quasilinear parabolic--parabolic Keller--Segel system with subcritical sensitivity}, J. Diff. Eqns., 252(2012), 692--715.
%
%
%
%
%% \bibitem{Tao61215}  Y. Tao, M. Winkler, \textit{Global existence and boundedness in a Keller--Segel--Stokes model with arbitrary
%%porous medium diffusion}, Discrete Contin. Dyn. Syst. Ser. A, 32(2012), 1901--1914.
%
%
% \bibitem{Tao71215}  Y. Tao, M. Winkler, \textit{Locally bounded global solutions in a three-dimensional chemotaxis--Stokes system
%with nonlinear diffusion}, Ann. Inst. H. Poincar\'{e} Anal. Non Lin\'{e}aire, 30(2013), 157--178.



% \bibitem{Taosdffggg41215}  Y. Tao, M. Winkler, \textit{Blow-up prevention by quadratic degradation in a two-dimensional Keller-Segel-
%Navier-Stokes system},  Z. Angew. Math. Phys. 67(6), 23 (2016)
%
%
%
%
%\bibitem{Tellonn791} J. I. Tello and D. Wrzosek, \textit{Predator-prey model with diffusion and indirect prey-taxis},
%Math. Mod. Meth. Appl. Sci., 26(2016), 2129--2162.
%
%
%
%
%\bibitem{Tuval1215}  I. Tuval, L. Cisneros, C. Dombrowski, et al., \textit{Bacterial swimming and oxygen transport near contact
%lines}, Proc. Natl. Acad. Sci. USA, 102(2005), 2277--2282.
%
%\bibitem{Wangscd331629} Y.  Wang,  \textit{Boundedness in the higher-dimensional chemotaxis-haptotaxis model with nonlinear diffusion},
%J. Diff. Eqns., 260(2016), 1975--1989.
%
%




%\bibitem{Wanssssssg21215}  W. Wang, M. Zhang, S. Zheng, \textit{To what extent is cross-diffusion controllable in a two-dimensional chemotaxis-(Navier-) Stokes system modeling coral fertilization},  Calculus of Variations and Part. Diff. Eqns.,, 60(2021), 143.

%\bibitem{Wddffang11215} Y. Wang, M. Winkler,  Z. Xiang, \textit{Global classical solutions in a two-dimensional
%chemotaxis-Navier-Stokes system with subcritical sensitivity}, Annali della Scuola Normale Superiore di Pisa--Classe di Scienze. XVIII, (2018), 2036--2145.
%
%
%
%
%%\bibitem{Wang11215} Y. Wang, X. Cao,  \textit{Global classical solutions of a $3d$ chemotaxis--Stokes system with rotation}, Discrete
%%Contin. Dyn. Syst. Ser. B, 20(2015), 3235--3254.
%
%  \bibitem{Wangssddff21215} Y. Wang, X. Li, \textit{Boundedness for a 3D chemotaxis-Stokes system with porous medium diffusion and tensor-valued
%chemotactic sensitivity}, Z. Angew. Math. Phys., 68(2017), Art. 29, 23 pp.
%
%
%%\bibitem{Wddffang11215} Y. Wang, M. Winkler,  Z. Xiang, \textit{Global classical solutions in a two-dimensional
%%chemotaxis-Navier-Stokes system with subcritical sensitivity}, Annali della Scuola Normale Superiore di Pisa--Classe di Scienze. XVIII, (2018), 2036--2145.




%
%
%\bibitem{Wangss21215}  Y. Wang, Z. Xiang, \textit{Global existence and boundedness in a Keller--Segel--Stokes system involving a tensor-valued
%sensitivity with saturation: The 3D case}, J. Diff. Eqns., 261(2016), 4944--4973.
%
%\bibitem{Wangjjk5566ddfggghjjkk1} W. Wang,  \textit{Global boundedness of weak solutions for a
%three-dimensional chemotaxis-Stokes system with
%nonlinear diffusion and rotation},
%J. Diff. Eqns., 268(2020),  7047--7091.
%
%
%
% \bibitem{Winkler37103}M. Winkler, \textit{Boundedness in the higher-dimensional parabolic--parabolic chemotaxis system with
%logistic source}, Comm.  Part. Diff. Eqns., 35(2010), 1516--1537.
%
%
% %\bibitem{Wiegnerdd79} M. Wiegner, \textit{The Navier-S-tokes equations: A neverending challenge}? Jahresber. Deutsch.
%%Math.-Verein., 101(1999), 1--25.
%
% \bibitem{Winkler79} M. Winkler, \textit{Does a volume-filling effect always prevent chemotactic collapse?}, Math. Methods Appl. Sci., 33(2010), 12--24.
%
%
% \bibitem{Winkler792} M. Winkler, \textit{Aggregation vs. global diffusive behavior in the higher-dimensional Keller--Segel model}, J. Diff. Eqns., 248(2010), 2889--2905.
%
% \bibitem{Winkler793} M. Winkler, \textit{Finite-time blow-up in the higher-dimensional parabolic--parabolic Keller--Segel system}, J. Math. Pures
%Appl., 100(2013),  748--767.
%
%
%\bibitem{Winkler31215}  M. Winkler, \textit{Global large-data solutions in a chemotaxis--(Navier--)Stokes system modeling cellular swimming in
%fluid drops}, Commun. Part. Diff. Eqns., 37(2012), 319--351.
%
%
%  \bibitem{Winkler61215}  M. Winkler, \textit{Stabilization in a two-dimensional chemotaxis-Navier-Stokes system}, Arch. Ration. Mech.
%Anal., 211(2014), 455--487.
%
%
%\bibitem{Winklerddfff51215} M. Winkler, \textit{Large-data global generalized solutions in a chemotaxis system with
%tensor-valued sensitivities}, SIAM J. Math. Anal., 47(2015), 3092--3115.
%
%
% \bibitem{Winkler11215}  M. Winkler, \textit{Boundedness and large time behavior in a three-dimensional chemotaxis--Stokes system
%with nonlinear diffusion and general sensitivity},
%Calculus  Var. Part. Diff. Eqns., 54(2015),   3789--3828.
%
%
%\bibitem{Winkler51215}  M. Winkler, \textit{Global weak solutions in a three-dimensional chemotaxis-Navier-Stokes system}, Ann. Inst.
%H. Poincar\'{e} Anal. Non Lin\'{e}aire, 33(2016),  1329--1352.
%
%\bibitem{Winklerawerrr51215}  M. Winkler, \textit{How far do chemotaxis-driven forces influence regularity in the Navier-Stokes system}?,  Trans.  Am. Math. Soc.,  369(2017), 3067--3125.
%
%\bibitem{Winklerssdff51215}  M. Winkler, \textit{Global existence and stabilization in a degenerate chemotaxis--Stokes system with mildly strong diffusion enhancement}, J. Diff. Eqns., 264(2018), 6109--6151.
%
%    \bibitem{Winklerssdff512jjj15}  M. Winkler, \textit{Global mass-preserving solutions in a two-dimensional chemotaxis-Stokes system with rotational flux components}, J. Evol. Eqns., 18(2018), 1267--1289.
%
%     \bibitem{Winkler444ssdff51215}  M. Winkler, \textit{Does fluid interaction affect regularity in the three-dimensional Keller-Segel System with saturated sensitivity?},
%     J. Math. Fluid Mechanics, 20(2018), 1889--1909.



%\bibitem{Zhanggggg5667} W. Zhang, P. Niu, S. Liu, \textit{Large time behavior in a chemotaxis model with logistic growth and indirect signal
%production}, Nonlinear Anal., RWA., 50(2019), 484--497.
%
%
%
%%\bibitem{Zhang12176}  Q. Zhang, X. Zheng, \textit{Global well-posedness for the two-dimensional incompressible
%%chemotaxis-Navier-Stokes equations}, SIAM J. Math. Anal., 46(2014), 3078--3105.
%
%
%\bibitem{Zheng00} J. Zheng, \textit{Boundedness of solutions to a quasilinear parabolic--elliptic Keller--Segel system with logistic source},
%J. Diff. Eqns., 259(2015), 120--140.
%
%
%
%%\bibitem{Zheng33312186} J. Zheng, \textit{Boundedness of solutions to a quasilinear parabolic--parabolic Keller--Segel system with logistic source},
%%J. Math. Anal. Appl.,  431(2015),  867--888.
%
%
%\bibitem{Zhengssdefr23} J. Zheng, \textit{A note on boundedness of solutions to a higher-dimensional quasi-linear chemotaxis system with logistic source},
%Z.  Angew. Math. Mech., 97(2017),  414--421.
%
%
%
%
%%\bibitem{Zhengddkkllssssssssdefr23} J. Zheng, \textit{Boundedness and global asymptotic stability of constant equilibria in a fully parabolic chemotaxis system with nonlinear logistic source},
%%J. Math. Anal. Appl., 450(2017), 1047--1061.
%
%
%
%
%% \bibitem{Zhddengssdeeezseeddd0} J. Zheng, \textit{Boundedness of solutions to a quasilinear higher-dimensional chemotaxis--haptotaxis model with nonlinear diffusion},
%%Discrete Cont. Dyn. Syst.,  37(2017), 627--643.
%
%
%%\bibitem{Zhengsddfff00} J. Zheng, \textit{A new result for global existence and boundedness in a three-dimensional
%% Keller--Segel(--Navier)--Stokes system with nonlinear diffusion}, Preprint.
%
%
%  \bibitem{Zhenddddgssddsddfff00} J. Zheng, \textit{An optimal result for global existence and boundedness in a three-dimensional
% Keller-Segel-Stokes system   with  nonlinear diffusion},  J. Diff. Eqns., 267(2019),  2385--2415.
%
%
%   \bibitem{Zhenjjjjjddddgssddsddfff00} J. Zheng, \textit{Global solvability and bounededness
%in a two-dimensional Keller-Segel-Navier-Stokes
%system with indirect signal production},  preprint.
%
%   \bibitem{Zhenjjjjssssjdssssdddgssddsddfff00} J. Zheng, \textit{Some progress on
% a three-dimensional Keller-Segel-Navier-Stokes
%system with indirect signal production},  preprint.
%
%
%%\bibitem{Zhengssssdefr23} J. Zheng, Y. Wang, \textit{Boundedness and decay behavior in a higher-dimensional quasilinear chemotaxis system with nonlinear logistic source},
%%Comput. Math. Appl., 72(10)(2016), 2604--2619.
%
%
%%
%%\bibitem{Zhengssssssdefr23} J. Zheng, Y. Wang, \textit{A note on global existence to a higher-dimensional quasilinear chemotaxis system with consumption of chemoattractant},
%%Discrete Contin. Dyn. Syst. Ser. B,  22(2)(2017),
%%669--686.
%
% \bibitem{Zhengsddfffsdddssddddkkllssssssssdefr23} J. Zheng, \textit{Global weak solutions in a three-dimensional Keller-Segel-Navier-Stokes system with nonlinear diffusion},
%J. Diff. Eqns., 263(2017), 2606--2629.
%
%
%\bibitem{Zhenddddgssddsddfff00} J. Zheng, \textit{An optimal result for global existence and boundedness in a three-dimensional
% Keller-Segel-Stokes system   with  nonlinear diffusion},  J. Diff. Eqns., 267(2019), 2385--2415.
%

% \bibitem{Zheklllkkkkllllndsssdddgssddsddfff00}J. Zheng,   Y. Ke,   \textit{Global bounded weak solutions for a chemotaxis-Stokes
%system with nonlinear diffusion and rotation},
%J. Diff. Eqns., 289(2021),  182--235.

%\bibitem{Zhenjjjjjjjgddfggghjjkk1}J. Zheng, D. Qi,    \textit{Global existence and boundedness in an N-dimensional chemotaxis-Navier-Stokes system with nonlinear diffusion and rotation}, J. Diff. Eqns., 335(2022),  347--397.
%
%\bibitem{Zhenjjjjgddfggghjjkk1}J. Zheng, D. Qi,  Y. Ke,   \textit{Global Existence, Regularity and Boundedness in a Higher-dimensional Chemotaxis-Navier-Stokes System with Nonlinear Diffusion and General Sensitivity}, Calculus  Var. Part. Diff. Eqns.,  61(2022), 150.






\end{thebibliography}
\end{document}